\documentclass[10pt]{amsart}
\usepackage[margin=1.2in]{geometry}
\usepackage{colortbl}
\usepackage{color}
\usepackage{cite}
\newtheorem{theorem}{Theorem}[section]

\newtheorem{proposition}{Proposition}[section]

\theoremstyle{definition}

\theoremstyle{remark}
\newtheorem{remark}[theorem]{Remark}

\numberwithin{equation}{section}

\usepackage{graphicx}
\graphicspath{{./figures/}}
\usepackage{subcaption}
\usepackage{multirow}
\usepackage{bbm}
\usepackage{lipsum}
\usepackage{amsfonts}
\usepackage{graphicx}
\usepackage{epstopdf}
\usepackage{algorithmic}
\usepackage{subcaption}
\usepackage{amssymb}
\usepackage{mathtools}
\usepackage{bbm}

\begin{document}
	
	\title{Conservative semi-Lagrangian schemes for kinetic equations \\Part I: Reconstruction}

	\author[S. Y. Cho]{Seung Yeon Cho}
	\address{Seung Yeon Cho\\
		Department of Mathematics and Computer Science  \\
		University of Catania\\
		95125 Catania, Italy}
	\email{chosy89@skku.edu}
	
	\author[S. Boscarino]{Sebastiano Boscarino}
	\address{Sebastiano Boscarino\\
		Department of Mathematics and Computer Science  \\
		University of Catania\\
		95125 Catania, Italy} \email{boscarino@dmi.unict.it}
	
	\author[G. Russo]{Giovanni Russo}
	\address{Giovanni Russo\\
		Department of Mathematics and Computer Science  \\
		University of Catania\\
		95125 Catania, Italy} \email{russo@dmi.unict.it}
	
	\author[S.-B. Yun]{Seok-Bae Yun}
	\address{Seok-Bae Yun\\
		Department of Mathematics\\ 
		Sungkyunkwan University\\
		Suwon 440-746, Republic of Korea}
	\email{sbyun01@skku.edu}

	\maketitle
	
	\begin{abstract}
	In this paper, we propose and analyse a reconstruction technique which enables one to design high-order conservative semi-Lagrangian schemes for kinetic equations. The proposed reconstruction can be obtained by taking the sliding average of a given polynomial reconstruction of the numerical solution. A compact representation of the high order conservative reconstruction in one and two space dimension is provided, and its mathematical properties are analyzed. To demonstrate the performance of proposed technique, we consider implicit semi-Lagrangian schemes for kinetic-like equations such as the Xin-Jin model and the Broadwell model, and then solve related shock problems which arise in the relaxation limit. Applications to BGK and Vlasov-Poisson equations will be presented in the second part of the paper.
	\end{abstract}

	\section{Introduction}
	Kinetic equations and quasi-linear systems of conservation laws are strongly related. For example, the behavior of rarefied gas is well described by the Boltzmann transport equation (BTE) \cite{Cercignani}. Once 
	velocity space is discretized, BTE has the mathematical structure of a semi-linear hyperbolic system of balance laws. In the so-called fluid dynamic limit, the distribution function approaches the Maxwellian whose parameters satisfy the Euler equations of gas dynamics, which is a quasi-linear system of conservation laws. The Broadwell model of the BTE in one space dimension is a semi-linear $3\times 3$ relaxation system. As the relaxation parameter vanishes, the model relaxes to a $2\times 2$ quasi-linear hyperbolic system of conservation laws. An implicit treatment of the collision term
	using L-stable schemes allows the construction of asymptotic preserving schemes which become consistent schemes of the relaxed limit \cite{CJR,pareschi2005implicit,Jin2}.

	Quasi-linear hyperbolic systems generically develop jump discontinuities in finite time. Most schemes for their numerical solutions are based on two fundamental ingredients: conservation and non-oscillatory reconstruction. Finite volume and finite difference methods have been widely used for the discretization of the convective terms of kinetic models (Eulerian approach), which are usually treated explicitly. In this way, it is relatively easy to construct conservative schemes. Conservation is relevant especially in the relaxed limit: lack of conservation will prevent weak consistency of the method for discontinuous solutions leading, for example, to $\mathcal{O}(1)$ errors in the propagation of shocks.

	
	Conservative non-oscillatory reconstruction such as
	the ENO or WENO methology \cite{Shu98} have been widely adopted in many practical problems \cite{CFR,CCD,LSZ}. The approach has been extended to a compact WENO (CWENO \cite{C-2008,CPSV,CPSV-2017,LPR2,LPR1,LPR}) reconstruction which gives uniform accuracy in a whole cell, and it allows the construction of efficient high order finite volume scheme in several space dimensions \cite{dumbser2017central}. Unfortunately, explicit Eulerian schemes cannot avoid CFL-type time step restrictions imposed by converction-like terms in hyperbolic equations.
	
	To treat this difficulty, semi-Lagrangian approaches recently have gained popularity because they do not suffer from such CFL-type time step restriction which arises in the treatment of Eulerian counterparts. Instead, since the semi-Lagrangian method is obtained by integrating the equations along its characteristics, this approach necessarily requires the computation of numerical solutions on off-grid points by a reconstruction which makes use of the numerical solutions on grid points.
	
	If one uses piecewise Lagrange polynomial reconstruction, then conservation is guaranteed if the same stencil is used in each cell, because of translation invariance (we shall call this a \emph{linear reconstruction}). On the other hand, such linear reconstruction may introduce spurious oscillations of may cause loss of positivity. If one wants to prevent appearing of spurious oscillations, then one can use high-order non-oscillatory reconstruction, such as ENO of WENO \cite{Shu98,CFR,carrillo2007nonoscillatory}. Similarly, positivity of the numerical solution can be maintained by positivity-preserving reconstructions \cite{campos2019algorithms,schmidt1988positivity}. Unfortunately these non-linear reconstructions destroy the translation invariance guaranteed by linear reconstruction, causing lack of conservation \cite{BCRY}.
	

	Numerous approaches have been introduced to treat such difficulties, and maintain conservation even with non-linear reconstruction. In particular, in the context of Vlasov-Poisson system several techniques were proposed. Among them, we mention the work based on primitive polynomial reconstruction \cite{FSB-vlasov-2001,CMS-vlasov-2010,QS}. In \cite{FSB-vlasov-2001}, the authors developed the Positive and Flux Conservative scheme. The authors considered essentially non-oscillatory method (ENO) or reconstructions based on positive limiters. In \cite{CMS-vlasov-2010}, the authors took a similar approach in the construction of primitive functions using splines. An weighted essentially non-oscillatory method (WENO) is also proposed to construct high order conservative non-oscillatory schemes in \cite{QS}. All these method are either one-dimensional or they provide a dimension by dimension interpolation. A general technique to restore conservation in semi-Lagrangian schemes was presented in \cite{RQT}. The technique has been also applied to the BGK model \cite{BCRY}. Although quite general, the technique suffers from CFL-type stability restrictions.
	
	In this paper we present a general technique which allows the construction of high-order conservative non-oscillatory semi-Lagrangian schemes in one and several dimensions, which are not affected by CFL-type restriction. Given cell averaged values on uniform grids, the idea is to compute sliding average of a precomputed non-oscillatory piecewise polynomial reconstruction. 
	
	The resulting reconstruction inherits the non-oscillatory properties of the precomputed polynomial and guarantees conservation of all discrete moments. The technique requires characteristic lines are parallel, which is the case of kinetic equations in which velocity space is discretized on the same velocity grid throughout space. An advantage of our method is that one can easily adopt previous techniques such as ENO, WENO, CWENO polynomials as our basic reconstructions.
	
	The mathematical properties of the proposed reconstruction are analyzed. In particular, we show that if we take CWENO polynomials of even degree $k$, for example $k=2,4$ \cite{C-2008,LPR1}, as a basic reconstruction, our approach gives $k+2$th order accuracy. Similar properties are also generalized to two dimensional reconstruction with CWENO polynomial in two space dimensions \cite{LPR1}. The description of technique is provided in the sense of cell averages, however, the idea can be extended to the point-wise framework in a similar manner.
	
	To test the quality of the proposed reconstruction, we apply it to the finite difference implicit semi-Lagrangian schemes for semi-linear hyperbolic system such as Xin-Jin model or Broadwell model. Applications to more general equations will be presented in a companion paper.
	
	This paper is organized as follows:
	In section \ref{sec:Conservative reconstruction in 1D}, we present a general framework for our conservative reconstruction in 1D and its related properties. section \ref{sec:Conservative reconstruction in 2D} is devoted to the conserative reconstruction in 2D.
	Semi-Lagrangian methods are described in section \ref{sec:A semi-lagrangian scheme for semi-linear hyperbolic relaxation}. In section \ref{sec:Numerical test}, several numerical tests are presented to verify the accuracy of the proposed schemes and its capability in treating shocks arising in the relaxation of semi-linear hyperbolic system.

	\section{Conservative reconstruction in 1D}
	\label{sec:Conservative reconstruction in 1D}
	
	Let $u:\mathbb{R} \rightarrow \mathbb{R}$ be a smooth function and $\bar{u}:\mathbb{R} \rightarrow \mathbb{R}$ be a corresponding sliding average function:
	\[
	\frac{1}{\Delta x}\int_{x-\Delta x/2}^{x+\Delta x/2}u(y)\,dy = \bar{u}(x).
	\] 
	Given cell averages on uniform grids $x_i=i\Delta x$:
	\[
	\frac{1}{\Delta x}\int_{I_i}u(x)\,dx = \bar{u}_{i}, \quad I_i=[x_{i-\frac{1}{2}},x_{i+\frac{1}{2}}],
	\] 
	for each $i \in \mathcal{I}$, 
	our goal is to construct an approximation $Q(x)$ of the sliding average $\bar{u}(x)$, which is conservative in the sense that for any periodic function $\bar{u}(x)$ with period $L=N\Delta x$, $N\in\mathbb{N}$, we have \[
	\sum_{i=1}^N Q(x_i+\theta)= \sum_{i=1}^N \bar{u}(x_i),\quad \theta \in [0,1).\]
	
	Assume we have a piecewise smooth reconstruction $R(x)=\sum_i R_i(x) \chi_i(x) $, for $i \in \mathcal{I}$, where $\chi_i(x)$ denotes the characteristic function of cell $i$ and each $R_i(x)$ denotes a polynomial of degree $k$ and has the following properties:
	\begin{enumerate}
		\item High order accuracy in the approximation of $u(x)$:
		\begin{align}\label{R accuracy}
		u(x) = R_i(x) + \mathcal{O}\left((\Delta x)^{k+1}\right), \quad x \in I_i.
		\end{align}
		
		\item Conservation in the sense of cell averages:
		\[
		\frac{1}{\Delta x}\int_{x_{i-\frac{1}{2}}}^{x_{i+\frac{1}{2}}}R_i(x)\,dx = \bar{u}_{i}.
		\] 
	\end{enumerate}    
	Consider a shifted interval $[y_{i-\frac{1}{2}},y_{i+\frac{1}{2}}]$ whose center is $x_{i+\theta}\equiv x_i + \theta \Delta x$, $\theta \in [0,1)$, and denote by $\bar{u}(x_{i+\theta})$ the sliding average of $u$ at $x_{i + \theta}$ (see Fig. \ref{1d interpolation}). 
	We see that 
	$$x_{i-\frac{1}{2}} \leq y_{i-\frac{1}{2}}  < x_{i+\frac{1}{2}} \leq y_{i+\frac{1}{2}}  < x_{i+\frac{3}{2}}.$$
	
	Our strategy is to approximate $\bar{u}(x_{i+\theta})$ by $Q_{i + \theta} \equiv Q(x_{i+\theta})$, where
	\begin{align}\label{idea}
	Q_{i+\theta} = \frac{1}{\Delta x}\int_{y_{i-\frac{1}{2}}} ^{y_{i+\frac{1}{2}}}R(x)\,dx = \frac{1}{\Delta x}\int_{x_{i-\frac{1}{2} + \theta}} ^{x_{i+\frac{1}{2}+\theta}}R(x)\,dx ,
	\end{align} 
	which is equivalent to
	\begin{align}\label{Q decom}
	Q_{i + \theta} = \frac{1}{\Delta x}\int_{x_{i-\frac{1}{2}+\theta}} ^{x_{i+\frac{1}{2}}}R_i(x)\,dx
	+ \frac{1}{\Delta x}\int_{x_{i+\frac{1}{2}}} ^{x_{i+\frac{1}{2} + \theta}}R_{i+1}(x)\,dx.
	\end{align} 
	
	From now on, we consider $R_i(x)$ to be piecewise polynomials of degree $k$ of the form:
	\begin{align}\label{R 1d taylor}
	R_i(x)= \sum_{\ell=0}^k \frac{R_i^{(\ell)}}{\ell !}(x-x_i)^{\ell}.
	\end{align}
	\begin{figure}[htbp]
		\centering
		\begin{subfigure}[b]{0.55\linewidth}
			\includegraphics[width=1\linewidth]{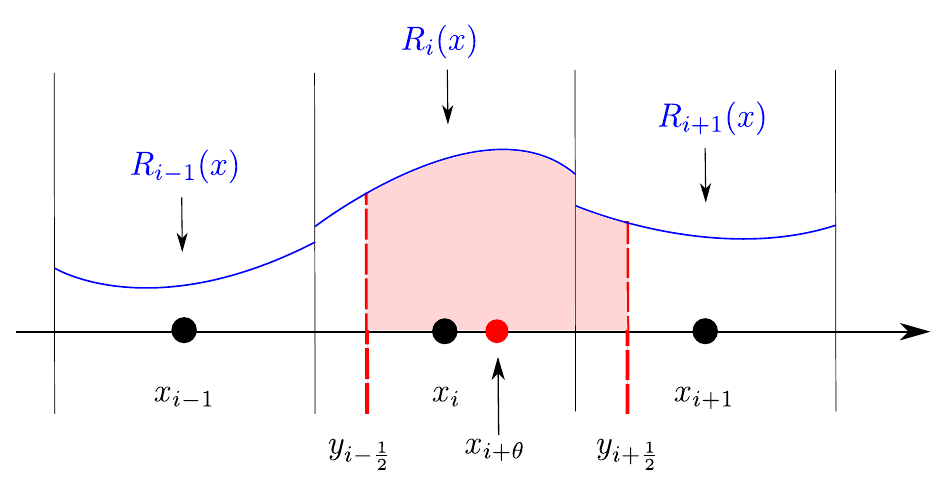}
		\end{subfigure}	
		\caption{Description of one-dimensional conservative reconstruction}\label{1d interpolation}  	
	\end{figure}
	Making use of \eqref{R 1d taylor} in the first term, we obtain
	\begin{align*}
	\frac{1}{\Delta x}\int_{x_{i-\frac{1}{2}+\theta}} ^{x_{i+\frac{1}{2}}}R_i(x)\,dx
	&= \frac{1}{\Delta x}\sum_{\ell=0}^k R_i^{(\ell)} \int_{x_{i-\frac{1}{2}+\theta}} ^{x_{i+\frac{1}{2}}} \frac{1}{\ell !}(x-x_i)^{\ell}\,dx= \sum_{\ell=0}^k (\Delta x)^{\ell} R_i^{(\ell)} \alpha_\ell(\theta)
	\end{align*} 
	where
	\begin{align}\label{form of alpha}
	\alpha_\ell(\theta) = \frac{1- (2\theta -1)^{\ell+1}}{2^{\ell+1}(\ell+1)!}.
	\end{align}
	Similarly, we can write
	\begin{align*}
	\frac{1}{\Delta x}\int_{x_{i+\frac{1}{2}}} ^{x_{i+\frac{1}{2} + \theta}}R_{i+1}(x)\,dx
	&:=\sum_{\ell=0}^k (\Delta x)^{\ell} R_{i+1}^{(\ell)} \beta_\ell(\theta),
	\end{align*} 
	with
	\begin{align}\label{form of beta}
	\beta_\ell(\theta) = \frac{(2\theta -1)^{\ell+1} - (-1)^{\ell+1} }  {2^{\ell+1}(\ell+1)!}.
	\end{align}
	Letting $Q_{i+\theta}$ denote the approximation of $\bar{u}(x_{i+\theta})$, we obtain 
	\begin{align}\label{scheme}
	Q_{i+\theta}:= \sum_{\ell=0}^k (\Delta x)^{\ell} \left(\alpha_{\ell}(\theta)R_{i}^{(\ell)} +\beta_\ell(\theta)R_{i+1}^{(\ell)}\right).
	\end{align}
	Here, we note that $\alpha_\ell(\theta)$ and $\beta_\ell(\theta)$ satisfy the following relations:
	\begin{itemize}
		\item If $\ell=2n$, $0 \leq n$, is a even number
		\begin{align}\label{ab even}
		\begin{split}
		\alpha_\ell(\theta) + \beta_\ell(\theta)
		=\frac{1}{(2n+1)!}\left(\frac{1}{2}\right)^{2n}.
		\end{split}
		\end{align}
		\item If $\ell=2n+1$, $0 \leq n$, is an odd number
		\begin{align}\label{ab odd}
		\alpha_\ell(\theta) + \beta_\ell(\theta)&= 0.
		\end{align}
	\end{itemize}
	We list the explicit form of $\alpha_\ell(\theta)$ and $\beta_\ell(\theta)$ for $\ell = 0,1,2$:
	\begin{align}\label{example ab}
	\begin{split}
	\alpha_0(\theta)&= 1-\theta,\quad	\alpha_1(\theta)= \frac{\theta(1-\theta)}{2},\quad	\alpha_2(\theta)= \frac{1-q(\theta)}{24}\cr
	\beta_0(\theta)&= \theta,\quad	\beta_1(\theta)= -\frac{\theta(1-\theta)}{2},\quad	\beta_2(\theta)= \frac{q(\theta)}{24},
	\end{split}
	\end{align}
	where $q(\theta)= 3\theta - 6\theta^2 + 4\theta^3$, for $\theta \in [0,1)$.

	\subsection{General Properties} 
	
	In this section, we provide several properties of the reconstruction \eqref{scheme} such as accuracy, conservation and consistency to the classical interpolation with a suitable choice of $R_i^{(\ell)}$ in the reconstruction.
	
	Recalling the assumption \eqref{R accuracy}, we have a function $R(x)$ which approximates point values of $u$ and our goal is to approximate sliding average function $\bar{u}$ with our reconstruction \eqref{scheme}. Before checking the accuracy order, we note that the cell average function $\bar{u}(x)$ can be expressed in terms of derivatives of function $u(x)$:
	\begin{align}\label{1d ubar expansion}
	\begin{split}
	\bar{u}(x) &= \frac{1}{\Delta x} \int_{x-\Delta x/2}^{x+\Delta x/2}u(y)\,dy= \frac{1}{\Delta x} \int_{x-\Delta x/2}^{x+\Delta x/2} \sum_{\ell=0}^\infty \frac{u^{(\ell)}(x)}{\ell !}(y-x)^{\ell}\,dy
	=\sum_{\ell=\text{even}}^\infty \left(\Delta x\right)^{\ell} u^{(\ell)}(x) 
	\frac{1}{(\ell+1)!}\left(\frac{1}{2}\right)^{\ell}.
	\end{split}
	\end{align}
	Inserting $x=x_{i+\theta}$ into \eqref{1d ubar expansion}, we obtain
	\begin{align*}
	\begin{split}
	\bar{u}(x_{i+\theta}) &=\sum_{\ell=\text{even}}^\infty \left(\Delta x\right)^{\ell} u^{(\ell)}(x_{i+\theta}) 
	\frac{1}{(\ell+1)!}\left(\frac{1}{2}\right)^{\ell}
	=u^{(0)}(x_{i+\theta}) +\frac{\left(\Delta x\right)^{2}}{24} u^{(2)}(x_{i+\theta}) 
	+\frac{\left(\Delta x\right)^{4}}{1920}u^{(4)}(x_{i+\theta}) +\cdots.
	\end{split}
	\end{align*}

	With this formula, in the following proposition, we provide a sufficient condition for a polynomial reconstruction $Q_{i+\theta}$ to be a $(k+2)$-th order accurate approximation of $\bar{u}(x+\theta)$ for $\theta \in [0,1)$.
	\begin{proposition}\label{Consistency}
		Let $k\geq 0$ be an even integer, $R_i \in \mathbb{P}^k$ be given by \eqref{R 1d taylor}, and $u$ be a smooth function $u:\mathbb{R} \in \mathbb{R}$. Suppose we have a piecewise polynomial $R(x)=\sum_i R_i(x) \chi_i(x) $, which satisfies
		\begin{align}\label{consistency condition}
		\begin{split}
		u_i^{(\ell)}&=R_i^{(\ell)}+ \mathcal{O}(\Delta x^{k+2-\ell}), \quad \text{$0 \leq \ell\leq k$, whenever $\ell$ is an even integer}\cr
		u_i^{(\ell)}-u_{i+1}^{(\ell)}&=R_i^{(\ell)}-R_{i+1}^{(\ell)}+ \mathcal{O}(\Delta x^{k+2-\ell}), \quad \text{$0\leq \ell < k$, whenever $\ell$ is an odd integer.} 
		\end{split}
		\end{align}
		Then, the reconstruction $Q_{i+\theta}$ gives a $(k+2)$-th order approximation of the sliding average $\bar{u}(x_{i+\theta})$ for any $\theta \in [0,1)$. 	
		
	\end{proposition}
	\begin{proof}
		For detailed proof, see \ref{Appendix accuracy proof}.
	\end{proof}
	
	\begin{remark}\label{rmk 2.1}
		\begin{enumerate}
			\item The reconstruction $Q_{i+\theta}$ approximates $\bar{u}(x_{i+\theta})$ on the basis of cell average values $\{\bar{u}_i\}_{i \in \mathcal{I}}$. Similarly, we can extend the idea of reconstruction to the framework of point values, which are used in conservative finite difference methods in section \ref{sec:A semi-lagrangian scheme for semi-linear hyperbolic relaxation}. 
			\item We also note that the second condition in \eqref{consistency condition} can be easily satisfied. Let $k\geq 0$ be an even integer, and consider a function $u(x) \in C^{k+2}(\mathbb{R})$, and its primitive function $U(x):= \int_{-\infty}^{x} u(y)\,dy \in  C^{k+3}(\mathbb{R})$. We first look for a polynomial $P_i(x)\in \mathbb{P}^{k+1}$ such that
			\[P_i(x_{i-\frac{1}{2}+j})=U(x_{i-\frac{1}{2}+j}), \quad j=-r, \cdots, s+1, \quad r+s=k.
			\]
			Then, the classical interpolation theory gives
			\[
			U(x) - P_i(x)= \frac{1}{(k+2)!} U^{(k+2)}(\xi_i) \prod_{j=-r}^{s+1}(x-x_{i-\frac{1}{2}+j}), \quad \xi_i \in (x_{i-\frac{1}{2}-r},x_{i+\frac{1}{2}+s}), 
			\]
			its first order derivative $p_i(x) \equiv P_i'(x) \in \mathbb{P}^k$ interpolates $u$ in the sense of cell-average:
			\begin{align*}
			\frac{1}{\Delta x}\int_{x_{i+j}-\Delta x/2}^{x_{i+j}+\Delta x/2}p_i(y)\,dy = \bar{u}_{i+j}, \quad j=-r,\cdots,s,
			\end{align*} 	
			and, for $0 \leq \ell \leq k$, its $(\ell+1)$-th derivative  $p_i^{(\ell)}(x) \equiv P_i^{(\ell+1)}(x) \in \mathbb{P}^{k-\ell}$ satisfies
			\begin{align}\label{p derivative}
			u^{(\ell)}(x) - p_i^{(\ell)}(x)=U^{(\ell+1)}(x) - P_i^{(\ell+1)}(x)= \frac{1}{(k+2)!} U^{(k+2)}(\xi_i) \frac{d^{\ell+1}}{dx^{\ell+1}} \left(\prod_{j=-r}^{s+1}(x-x_{i-\frac{1}{2}+j})\right). 
			\end{align} 	
			Similarly, we can find polynomials $p_{i+1}(x)\in \mathbb{P}^k$ and $P_{i+1}(x)\in \mathbb{P}^{k+1}$ such that
			\begin{align*}
			u^{(\ell)}(x+\Delta x) - p_{i+1}^{(\ell)}(x+\Delta x)&=U^{(\ell+1)}(x+\Delta x) - P_{i+1}^{(\ell+1)}(x+\Delta x)\cr
			&= \frac{1}{(k+2)!} U^{(k+2)}(\xi_{i+1}) \frac{d^{\ell+1}}{dx^{\ell+1}} \left(\prod_{j=-r}^{s+1}(x+\Delta x-x_{i+1-\frac{1}{2}+j})\right), 
			\end{align*} 	
			where $\xi_{i+1}\in (x_{i+\frac{1}{2}-r},x_{i+\frac{3}{2}+s})$.
			Then, the relation $U^{(k+2)}(\xi_i)  -  U^{(k+2)}(\xi_{i+1}) =\mathcal{O}\left(\Delta x\right)$, gives
			\begin{align*}
			&\left(u^{(\ell)}(x_i) - p_i^{(\ell)}(x_i)\right) - \left(u^{(\ell)}(x_{i+1}) - p_{i+1}^{(\ell)}(x_{i+1})\right)\cr
			&\qquad= \frac{1}{(k+2)!} \left(U^{(k+2)}(\xi_i)  -  U^{(k+2)}(\xi_{i+1}) \right) \left\{\frac{d^{\ell}}{dx^\ell} \left(\prod_{j=-r}^{s+1}(x-x_{i-\frac{1}{2}+j})\right)\right\}_{x=x_i}
			= \mathcal{O}\left((\Delta x)^{k+2-\ell}\right).
			\end{align*}
			\item If $R_i^{(\ell)}$ can be represented with a Lipschitz function $F_\ell$:
			\begin{align*}
			R_i^{(\ell)} =  F_\ell\left(\bar{u}_{i-r},\cdots,\bar{u}_{i+s}\right)
			\end{align*}
			which satisfies
			\begin{align*}
			F_\ell\left(\bar{u}_{i-r},\cdots,\bar{u}_{i+s}\right)- u^{(\ell)}(x_i) &= \mathcal{O}\left((\Delta x)^{k+1-\ell}\right),
			\end{align*}
			the condition \eqref{consistency condition} is also satisfied. For more details, we refer to \ref{appendix remark proof}.
		\end{enumerate} 	
		
		%
	\end{remark}
	
	In Proposition \ref{Consistency}, we see that the choice of an even integer $k \ge 0$ leads to the improvement of accuracy. In such a case, we show that the reconstruction $Q_{i+\theta}$ based on linear weights coincides with the classical interpolation. 
	\begin{proposition}\label{1d consistency prop}
		Let $k \ge 0$ be an even integer with $k=2r$. For each $i\in \mathcal{I}$, assume that we have a basic reconstruction $R_i(x) \in \mathbb{P}^k$, which is a polynomial of degree $k$ in \eqref{R 1d taylor} and interpolates the function $u$ in the sense of cell averages:
		\begin{align}\label{Lag condition}
		\begin{split}
		\frac{1}{\Delta x}\int_{x_{i+j-\frac{1}{2}}}^{x_{i+j+\frac{1}{2}}}R_i(x)\,dx = \bar{u}_{i+j},\quad -r\leq j \leq r.
		\end{split}
		\end{align}
		with a symmetric stencil $S_i:=\{i-r,i-r+1,\cdots,i+r\}$. Then, the reconstruction $Q_{i+\theta}$ in \eqref{scheme} based on $R_i$ and $R_{i+1}$, is the Lagrange polynomial $L(x)$ that interpolates $\bar{u}_{i+j}$, for $-r \leq j\leq r+1$, where $x = x_i + \theta \Delta x$ and $\theta \in [0,1)$.
	\end{proposition}
	The proof is based on the observation that interpolation in the sense of the cell averages is equivalent to point-wise interpolation of sliding averages at cell center, which in turn,  is equivalent to point-wise interpolation of primitive function at cell edges. A detailed proof, based on explicit representation obtained by Lagrange interpolation, is given in \ref{Appendix consistency proof}.
	
	\begin{remark}
		For $k=0$, the only possible choice is to set $R_i(x)\equiv \bar{u}_i$ and the resulting reconstruction $Q_{i+\theta}$ reduces to the linear interpolation constructed from two points $\bar{u}_i$ and $\bar{u}_{i+1}$. 
	\end{remark}

	In the following proposition, we show that total mass is preserved for any $\theta$-shifted summation, $\theta \in [0,1)$.
	\begin{proposition}\label{prop con 1d}
		Assume that $R_i(x)$ satisfies
		\begin{align}\label{conservation condition}
		\begin{split}
		\frac{1}{\Delta x}\int_{x_{i-\frac{1}{2}}}^{x_{i+\frac{1}{2}}}R_i(x)\, dx = \bar{u}_{i}, \quad i \in \mathcal{I}.
		\end{split}
		\end{align}
		Then, for periodic functions $\bar{u}(x)$ with period $L=N \Delta x,\, N \in \mathbb{N}$ 
		\begin{align}\label{shifted summation 1d}
		\sum_{i=1}^N Q_{i+\theta} = \sum_{i=1}^N \bar{u}_i, 
		\end{align}
		for any $\theta \in [0,1)$.  
	\end{proposition}
	\begin{proof}
		Since $\theta$ does not depend on $i$,
		\begin{align*}
		\sum_{i=1}^N Q_{i+\theta}&= \sum_{i=1}^N \left(\frac{1}{\Delta x}\int_{x_{i-\frac{1}{2}+\theta}} ^{x_{i+\frac{1}{2}}} R_i(x) \,dx
		+ \frac{1}{\Delta x}\int_{x_{i+\frac{1}{2}}} ^{x_{i+\frac{1}{2}+\theta}} R_{i+1}(x) \,dx \right)\cr
		&= \sum_{i=1}^N \left(\frac{1}{\Delta x}\int_{x_{i-\frac{1}{2}+\theta}} ^{x_{i+\frac{1}{2}}} R_i(x) \,dx
		+ \frac{1}{\Delta x}\int_{x_{i-\frac{1}{2}}} ^{x_{i-\frac{1}{2}+\theta}} R_i(x) \,dx \right)\cr
		&= \sum_{i=1}^N \frac{1}{\Delta x}\int_{x_{i-\frac{1}{2}}} ^{x_{i+\frac{1}{2}}} R_i(x) \,dx=\sum_{i=1}^N \bar{u}_i.
		\end{align*} 	
		Here we used the periodicity to write the second line and \eqref{conservation condition} for the last line. 
	\end{proof}
	
	\begin{remark}
		We remark that this summation preserving property can be useful when our reconstruction is applied to the semi-Lagrangian treatment of a constant convection term, where	characteristic curves are given by parallel lines for each grid point.	In such cases, the proposed reconstruction attains conservation at a discrete level, hence it can be applied to the simulation of physical models satisfying this conservation property. Considerable examples are the BGK type models of the Boltzmann equation of rarefied gas dynamics. We can also apply this to the splitting method for the Vlasov-Poisson system in plasma physics. These problems will be considered in the second part of this paper.  
	\end{remark}

	In the following section, we will show that our reconstruction \eqref{scheme} inherits some properties of the basic reconstruction $R_i(x)$ such as non-oscillatory property and positivity.
	\subsection{Choice of the basic reconstruction $R$}
	\subsubsection{Non-oscillatory property}
	We consider some specific choices of the basic reconstruction $R$. In particular, we consider CWENO \cite{C-2008}, \cite{LPR1} and CWENOZ \cite{CPSV-2017}. As an illustration, we consider the case $k=2$, and we take CWENO23 reconstruction in \cite{LPR1} as a basic reconstruction $R$. We start from a polynomial of degree two $P_{OPT}^i(x)$ which interpolates $\bar{u}_{i-1},\bar{u}_{i},\bar{u}_{i+1}$ in the sense of cell averages:
	\[
	\frac{1}{\Delta x}\int_{x_{i+l-\frac{1}{2}}}^{x_{i+l+\frac{1}{2}}}P_{OPT}^i(x)\,dx = \bar{u}_{i+l}, \quad l=-1,0,1.
	\] 
	Then, this polynomial can be written as 
	$P_{OPT}^i(x)=\tilde{u}_{i} + \tilde{u}'_{i}(x-x_i) + \frac{1}{2}\tilde{u}''_{i}(x-x_i)^2$
	with
	\begin{align*}
	\tilde{u}_{i}=\bar{u}_{i}-\frac{1}{24}(\bar{u}_{i+1}-2\bar{u}_{i}+\bar{u}_{i-1}),\quad  	\tilde{u}'_{i}=\frac{\bar{u}_{i+1}-\bar{u}_{i-1}}{2\Delta x}, \quad		\tilde{u}''_{i}=\frac{\bar{u}_{i+1}-2\bar{u}_i+\bar{u}_{i-1}}{({\Delta x})^2}, 
	\end{align*}  
	and it gives a third order accurate reconstruction of $u$ in $I_i$:
	\[
	P_{OPT}^i(x)= u(x) + \mathcal{O}(\Delta x)^3, \quad \forall x\in I_i.
	\]
	In the CWENO23 reconstruction, to avoid oscillations, we use the  following convex combination:
	\begin{align}\label{conv combi}
	R_i(x)=\sum_k \omega_k^i P_k^i(x), \quad \sum_k \omega_k^i =1, \quad \omega_k^i \geq 0, \quad k\in \{L,C,R\}
	\end{align}
	where $P_L^i$ and $P_R^i$ are first order polynomials such that
	\[
	\int_{x_{i+l-\frac{1}{2}}}^{x_{i+l+\frac{1}{2}}}P_L^i(x)\,dx = \bar{u}_{i+l}, \quad l=-1,0, \quad\quad \int_{x_{i+l-\frac{1}{2}}}^{x_{i+l+\frac{1}{2}}}P_R^i(x)\,dx = \bar{u}_{i+l}, \quad l=0,1,
	\]
	which gives
	\[P_L^i(x)=\bar{u}_{i}+\frac{\bar{u}_{i}-\bar{u}_{i-1}}{\Delta x}(x-x_i),\quad 
	P_R^i(x)=\bar{u}_{i}+\frac{\bar{u}_{i+1}-\bar{u}_{i}}{\Delta x}(x-x_i).\]
	The second order polynomial $P_C(x)$ is obtained from
	\begin{align*}
	P_{OPT}^i(x)= C_L P_L^i(x) + C_R P_R^i(x) + C_C P_C^i(x),
	\end{align*}
	with a choice of positive coefficients such that 
	\[
	C_L,C_R,C_C \geq 0,\quad C_L=C_R, \quad C_L + C_C + C_R=1.
	\]
	A common choice is to set $C_L=C_R=1/4$, $C_C=1/2$. 
	The non-linear weights $\omega_k^i$ in \eqref{conv combi} are chosen as follows: 
	\begin{align}\label{CWENO omega}
	\omega_k^i=\frac{\alpha_k^i}{\sum_\ell \alpha_\ell^i}, \quad \alpha_k^i=\frac{C_i}{(\epsilon + {\beta}_k^i)^p}, \quad k,\ell \in \{L,C,R\}
	\end{align}
	where the constant $\epsilon$ is used to avoid the denominator vanishing and the constant $p$ weights the smoothness indicator. We use $\epsilon=(\Delta x)^2$ or $10^{-6}$ and $p=2$ in the numerical tests. An explicit expression of smoothness indicators is the following:
	\begin{align*}
	{\beta}_L^i &= (\bar{u}_i - \bar{u}_{i-1})^2, \quad {\beta}_R^i = (\bar{u}_{i+1} - \bar{u}_i)^2,\cr
	{\beta}_C^i &= \frac{13}{3}(\bar{u}_{i+1} - 2\bar{u}_i + \bar{u}_{i-1})^2 +\frac{1}{4}(\bar{u}_{i+1} - \bar{u}_{i-1})^2.  
	\end{align*}
	We refer to \cite{LPR2} for details on CWENO reconstruction. 
	As a consequence, the reconstruction \eqref{conv combi} is third order accurate in smooth region and automatically becomes second order accurate in the presence of discontinuity.	
	The final form of the CWENO23 reconstruction $R_i(x)$ is given by
	\begin{align}\label{R 1d}
	R_i(x) = R_i^{(0)} + R_i^{(1)}(x-x_i) + \frac{1}{2}R^{(2)}_i(x-x_i)^2,
	\end{align}  
	where
	\begin{align*}
	&R_i^{(0)}= \bar{u}_i  -  \frac{1}{12}\omega_C^i\left(\bar{u}_{i+1}-2\bar{u}_i+\bar{u}_{i-1}\right) \cr
	&R_i^{(1)}=\omega_L^i \frac{\bar{u}_{i}-\bar{u}_{i-1}}{\Delta x} + \omega_R^i \frac{\bar{u}_{i+1}-\bar{u}_{i}}{\Delta x} + \omega_C^i\frac{\bar{u}_{i+1}-\bar{u}_{i-1}}{2\Delta x}\cr
	&R_i^{(2)}=2\omega_C^i \frac{\bar{u}_{i+1}-2\bar{u}_i+\bar{u}_{i-1}}{(\Delta x)^2}.		
	\end{align*}
	
	The CWENO23Z reconstruction also takes the form \eqref{R 1d}, but its non-linear weights are calculated as follows:
	\begin{align}\label{CWENOZ omega}
	\omega_k^i=\frac{\alpha_k^i}{\sum_\ell \alpha_\ell^i}, \quad \alpha_k^i=C_i\left(1+\frac{\tau}{\epsilon + {\beta}_k^i}\right)^p, \quad k,\ell \in \{L,C,R\}
	\end{align}
	where $p\geq 1$ and $\tau=\left|\beta_R^i-\beta_L^i\right|$.
	
	%
	
	\begin{remark}\label{CWENO23andZ}
		In \ref{appendix CWENO23andZ}, we prove that the condition \eqref{consistency condition} in Proposition \ref{Consistency} is satisfied both for CWENO23 and CWENO23Z if a given $u$ function is smooth enough. This shows that the corresponding reconstruction \eqref{R 1d} becomes a fourth order accurate reconstruction for smooth solutions.
	\end{remark}
	
	\begin{figure}[ht]
		\centering
		\begin{subfigure}[b]{0.45\linewidth}
			\includegraphics[width=1\linewidth]{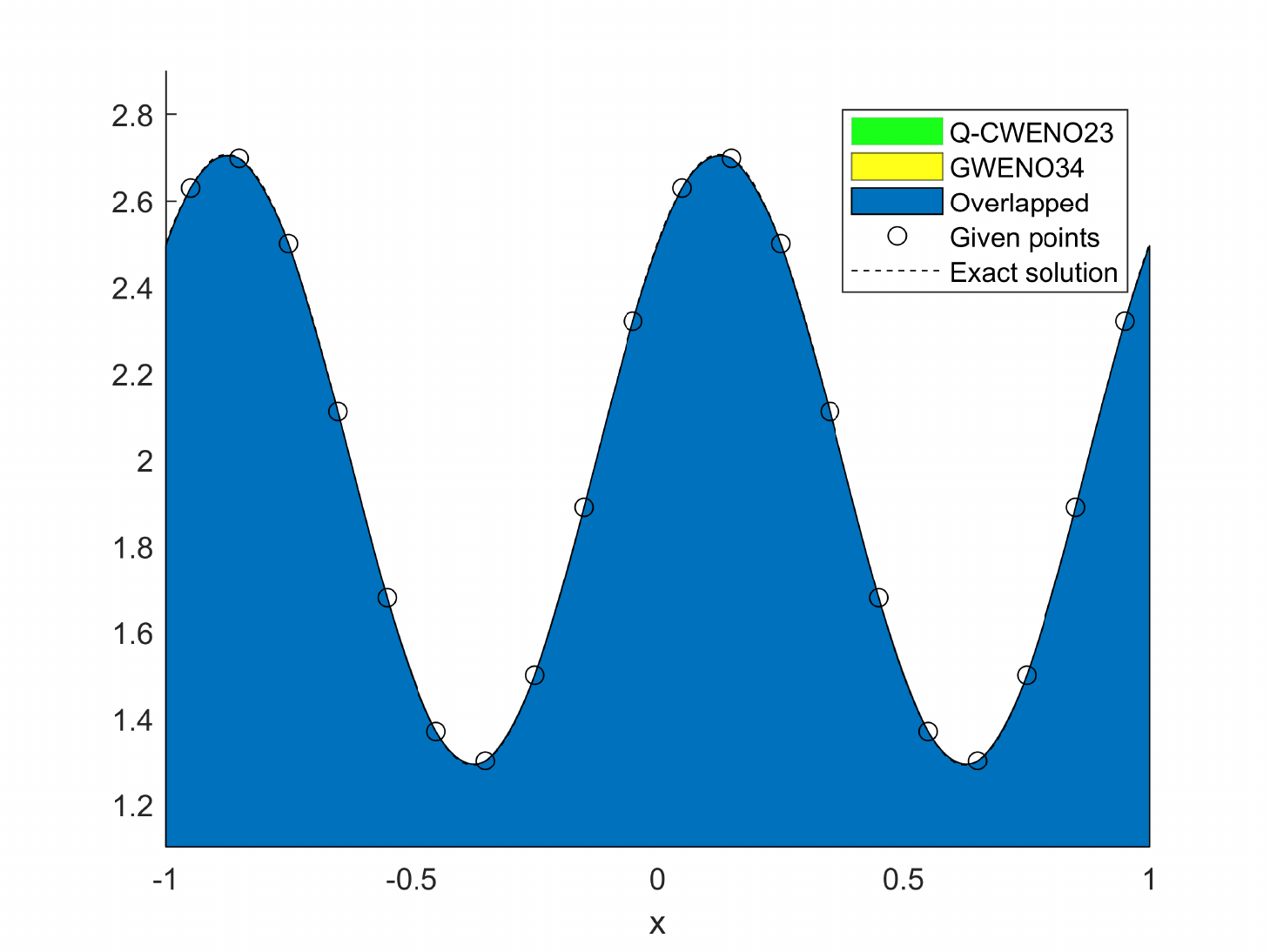}
			\subcaption{Comparison of $\bar{u}_1$ is given by \eqref{function for comparison2} and its reconstructions.}\label{fig: comparison c}
		\end{subfigure}		
		\begin{subfigure}[b]{0.45\linewidth}
			\includegraphics[width=1\linewidth]{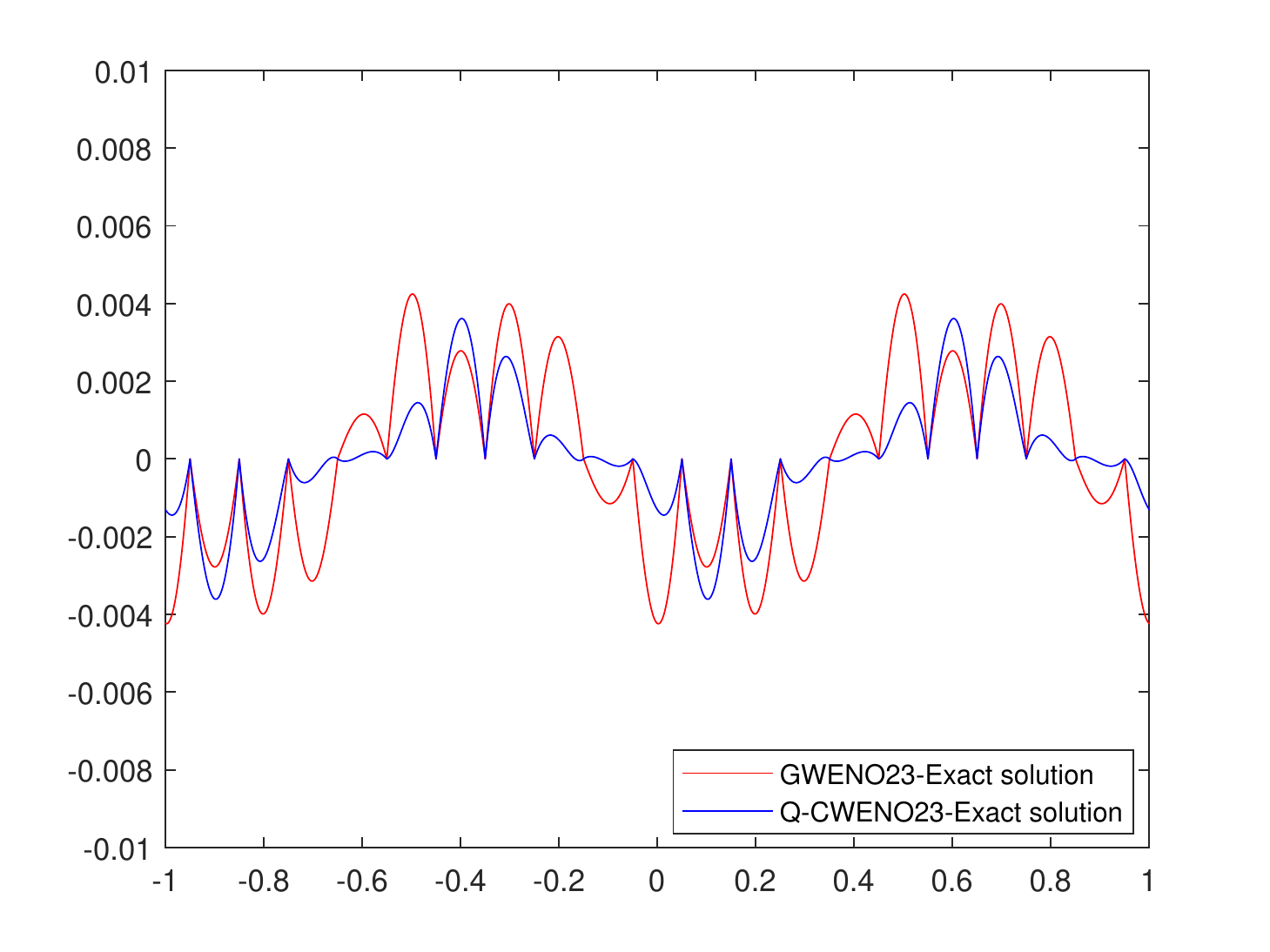}
			\subcaption{Errors between reconstructions and exact solutions \eqref{function for comparison2}.
			}\label{fig: comparison d}
		\end{subfigure}		
		\begin{subfigure}[b]{0.45\linewidth}
			\includegraphics[width=1\linewidth]{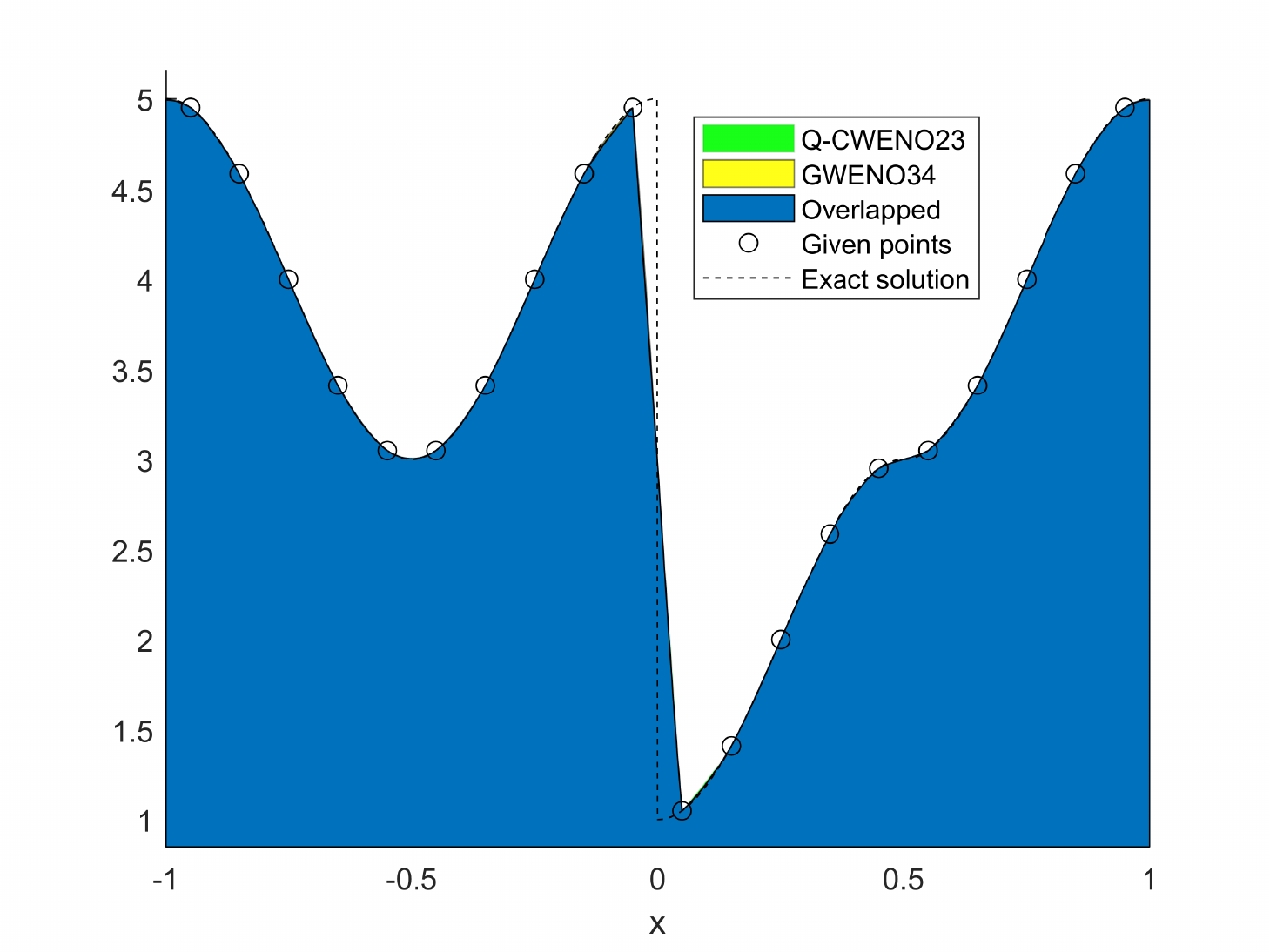}
			\subcaption{Comparison of $\bar{u}_2$ is given by \eqref{function for comparison} and its reconstructions.}\label{fig: comparison a}
		\end{subfigure}		
		\begin{subfigure}[b]{0.45\linewidth}
			\includegraphics[width=1\linewidth]{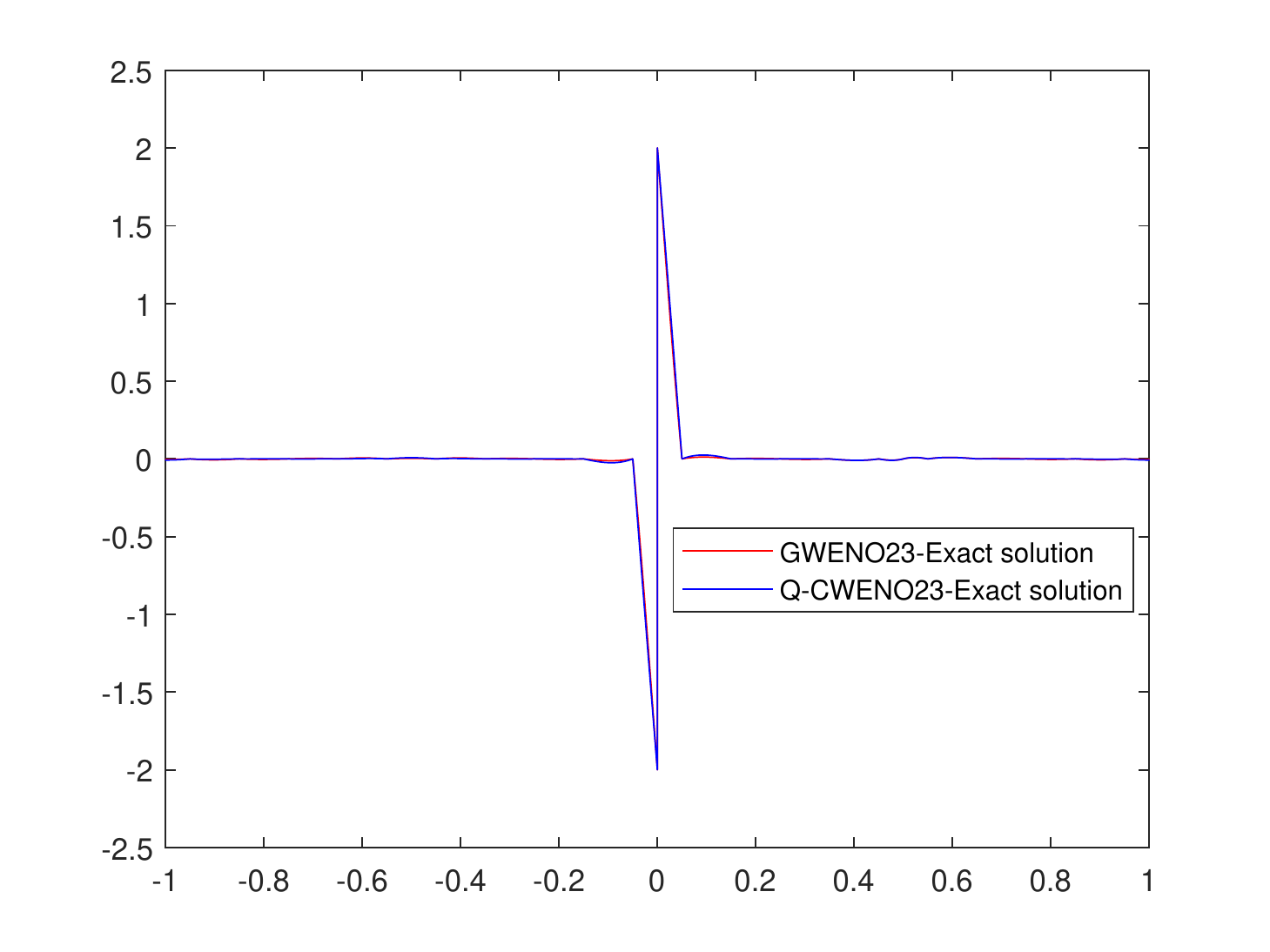}
			\subcaption{Errors between reconstructions and exact solutions \eqref{function for comparison}.}\label{fig: comparison b}
		\end{subfigure}
		\caption{Comparison of reconstructions between Q-CWENO23 and GWENO34. In Figs.\ref{fig: comparison c} and \ref{fig: comparison a}, dashed lines are exact solutions $\bar{u}_1(x)$, $\bar{u}_2{x}$ and black circles are given values on grid points of $\bar{u}_1{x}$, $\bar{u}_2{x}$ in \eqref{function for comparison2} and \eqref{function for comparison}.} \label{fig: comparison}
	\end{figure}
	In Fig. \ref{fig: comparison}, we compare the proposed conservative reconstruction \eqref{R 1d} using CWENO23 \cite{LPR} with a generalized WENO reconstruction originally proposed in \cite{CFR} in the context of semi-Lagrangian method. We shall denote it by GWENO34 obtained with four points, which achieves fourth order accuracy in the smooth solution. Hereafter we denote by Q-CWENO23 the conservative reconstruction based on CWENO23. To compute solutions with a few points $N=20$, we set $\epsilon=1$ for Q-CWENO23. We consider the following sliding average functions on the periodic domain $[-1,1]$:
	\begin{align}\label{function for comparison2}
	\bar{u}_1(x)= 4+\sin(2\pi x)+\cos(2\pi x), \quad -1\leq x < 1,
	\end{align} 
	\begin{align}\label{function for comparison}
	\bar{u}_2(x)= 
	\begin{cases}
	3+ 2\sin^2(\pi(x-0.5)), \quad -1\leq x<0\\
	3- 2\sin^2(\pi(x-0.5)), \quad 0 \leq x< 0.5\\
	3+ 2\sin^2(\pi(x-0.5)), \quad 0.5 \leq x< 1.
	\end{cases}
	\end{align}
	In Fig. \ref{fig: comparison}, one can observe that Q-CWENO23 and GWENO34 show similar results. 
	For a smooth function $\bar{u}_1(x)$ in \eqref{function for comparison2}, Figs. \ref{fig: comparison c} and \ref{fig: comparison d} implies that errors are relatively small, while for a discontinuous function $\bar{u}_2(x)$ in \eqref{function for comparison}, Figs. \ref{fig: comparison a}, \ref{fig: comparison b} show that errors are concentrated near a discontinuity. 

	In order to clarify the difference between solutions, in Table \ref{tab2}, we report the maximal relative conservation errors between the summation of reconstructed points $Q_{i+\theta}$ and that of given points $\bar{u}_\ell(x_i)$, $\ell=1,2$, over $\theta=0,0.001,\dots,0.999$ using the following measure: 
	\begin{align}\label{Err formula}
	{Err}_\ell= \frac{max_\theta \left|\sum_i Q_{i+\theta} - \sum_i \bar{u}_\ell(x_i)\right|}{\sum_i \bar{u}_\ell(x_i)}, \quad \ell=1,2.
	\end{align}
	From the Table \ref{tab2}, we conclude that Q-CWENO23 recovers the reference summation of $\bar{u}_\ell(x_i)$ for any values of $\theta \in[0,1)$ even in the presence of a discontinuity. The errors for Q-CWENO23 and GWENO34 are both within machine precision for the smooth function $\bar{u}_1$. In this case, the two reconstructions almost coincides the standard Lagrangian interpolation which is conservative. When the function is not smooth as in $\bar{u}_2$, Q-CWENO23 is still fully conservative within machine precision, hence it verifies Proposition \ref{prop con 1d}. Numerical experiments in which conservation is relevant will be discussed in section \ref{sec:Numerical test}.
	
	
	\begin{table}[t]
		\centering
		{\begin{tabular}{|ccc|}
				\hline
				\multicolumn{1}{ |c| }{Reconstruction}& \multicolumn{1}{ c|  }{Fig.\ref{fig: comparison c}} & \multicolumn{1}{ c| }{Fig.\ref{fig: comparison a}} \\ 
				\hline	
				\multicolumn{1}{ |c|  }{Q-CWENO23}& \multicolumn{1}{ |c|  }{
					6.6613e-16
				} & \multicolumn{1}{ c|  }{5.0753e-16 }  \\
				\hline	
				\multicolumn{1}{ |c|  }{GWENO34}& \multicolumn{1}{ |c|  }{ 6.6613e-16
				} & \multicolumn{1}{ c|  }{9.2766e-04}  \\
				\hline
		\end{tabular}}
		\captionof{table}{Relative conservation errors \eqref{Err formula} of the reconstruction for $\bar{u}_1$ \eqref{function for comparison2} and $\bar{u}_2$ \eqref{function for comparison}.} \label{tab2}
	\end{table}
	
	\begin{remark}
		As an example for the case $k=4$, we can use CWENO35 \cite{C-2008} as a basic reconstruction. The explicit form of $R_i^{(ell)},\, \ell=0,1,2,3,4$ is presented in \eqref{Appendix CWENO35}.
	\end{remark}
	

	\subsubsection{Positive preserving property}
	In several circumstances the solution one is looking for is a non negative function. This is the case, for example, of distribution function in kinetic equations. In such cases it may be important to preserve at a discrete level the positivity of the solution. Standard piecewise polynomial reconstructions (linear reconstructions) do not preserve positivity, however several techniques exist in the literature that can be adopted to ensure positivity in the reconstruction (\cite{campos2019algorithms,schmidt1988positivity}). Here we remark that if the basic reconstruction $R$ is positive preserving, that the sliding average of $R$ will provide a conservative and positivity preserving reconstruction. Given a non-negative basic reconstructions $R_i(x) \ge0$, obtained from positive cell averages $\bar{u}_i>0 \> \forall i$, the positivity of the reconstruction \eqref{scheme} directly follows from \eqref{idea}. Here we verify this with a numerical example. Let us consider a basic reconstruction $R_i$, obtained from positive cell averages $\{\bar{u}_i\}$, using the {\em Positive Flux Conservative} (PFC) technique explained in \cite{FSB-vlasov-2001}: 
	\begin{align}\label{parabola}
	R_i(x)=R_i^{(0)} + R_i^{(1)} \left( x-x_i \right) + \frac{R_i^{(2)}}{2} \left( x-x_i \right)^2,\quad x\in [x_{i-1/2},x_{i+1/2}].
	\end{align}
	Here $R_i^{(0)}$, $R_i^{(1)}$ and $R_i^{(2)}$ are given by
	\begin{align*}
	R_i^{(0)}&=\bar{u}_{i} - \frac{\varepsilon_i^+ (\bar{u}_{i+1} - \bar{u}_{i}) - \varepsilon_i^- (\bar{u}_{i} - \bar{u}_{i-1})}{24},\cr 
	R_i^{(1)}&= \frac{\varepsilon_i^+ (\bar{u}_{i+1} - \bar{u}_{i}) + \varepsilon_i^- (\bar{u}_{i} - \bar{u}_{i-1})}{2 \Delta x},\,\cr	R_i^{(2)}&= \frac{\varepsilon_i^+ (\bar{u}_{i+1} - \bar{u}_{i}) - \varepsilon_i^- (\bar{u}_{i} - \bar{u}_{i-1})}{(\Delta x)^2},
	\end{align*} 
	where slope limiters $\varepsilon_i^+$ and $\varepsilon_i^-$ are defined by
	\begin{align}\label{limiter}
	\begin{split}
	\varepsilon_i^+ &= \begin{cases}
	\min\big( 1; 2\bar{u}_i / (\bar{u}_{i+1} - \bar{u}_i) \big), \quad\quad\qquad\qquad\qquad \,\,\,\, \text{if $\bar{u}_{i+1} - \bar{u}_i>0$}\\
	\min \big( 1; -2(\bar{u}_\infty - \bar{u}_i) / (\bar{u}_{i+1} - \bar{u}_i) \big), \,\quad\qquad\qquad \text{if $\bar{u}_{i+1} - \bar{u}_i<0$}
	\end{cases}\cr
	\varepsilon_i^- &= \begin{cases}
	\min\big( 1; 2(\bar{u}_{\infty} - \bar{u}_i) \big) / (\bar{u}_{i} - \bar{u}_{i-1}) \big), \quad\qquad\qquad\,\,\, \text{if $\bar{u}_{i} - \bar{u}_{i-1}>0$}\\
	\min \big( 1; -2 \bar{u}_i / (\bar{u}_{i} - \bar{u}_{i-1}) \big), \,\qquad\qquad\qquad\qquad \text{if $\bar{u}_{i} - \bar{u}_{i-1}>0$}
	\end{cases},
	\end{split}
	\end{align}
	with $\bar{u}_\infty:= \max_i \bar{u}_i$. 
	This basic reconstruction has been proposed in \cite{FSB-vlasov-2001} in order to preserve positivity of the solution and maintain essentially non oscillatory property.
	
	Hereafter we denote by Q-Parabola the reconstruction \eqref{scheme} based on \eqref{parabola}. In Fig. \ref{fig: comparison pp}, we compare Q-Parabola with Q-CWENO23 reconstructions. For this, we use the following sliding average function on the periodic domain $[-1,1]$:
	\begin{align}\label{function for comparison pp}
	\bar{u}_3(x)= 
	\begin{cases}
	10^{-5} + 0.1\left(1 + \sin(\pi x)\right), \quad -0.5\leq x \leq 0.4\\
	10^{-5} ,\qquad\qquad\qquad\qquad\quad \text{otherwise}
	\end{cases}.	
	\end{align}	
	In Fig. \ref{fig: comparison pp 1} and \ref{fig: comparison pp 1 zoom}, the difference between two reconstructions appears near $[-0.65,-0.55]$ and $[0.45,0.55]$. In case of Q-Parabola, the use of positive limiter \eqref{limiter} always guarantees the positive reconstructions for any $x\in [-1,1]$, while very small oscillations appear near discontinuities. 
	On the other hand, although Q-CWENO23 always prevents spurious oscillation, negative solutions may occur depending on the choice of $\epsilon$ used for non-linear weights \eqref{CWENO omega}. In this case, we took $\epsilon=10^{-6}$, and  Eq.~\eqref{CWENO omega} of CWENO23 returns weights very close to the  linear ones on the cell $[-0.6,-0.5]$, which gives negative values on the interval $[-0.65,-0.55]$. 
	We remark that if CWENO23 reconstructions give linear polynomials on two consecutive cells, the corresponding reconstruction \eqref{scheme} is to be positive between the two cell centers. Consequently, the suitable choice of $\epsilon$ can enable Q-CWENO23 to avoid both negative reconstructions and spurious oscillations. 
	Other possible ways to guarantee the positivity of basic reconstructions are to adopt a linear scaling approach \cite{zhang2010maximum,zhang2011maximum,friedrich2019maximum} or use positive limiters \cite{filbet2001comparison,CMS-vlasov-2010}.

	\begin{figure}[t]
		\centering
		\begin{subfigure}[b]{0.49\linewidth}
			\includegraphics[width=1\linewidth]{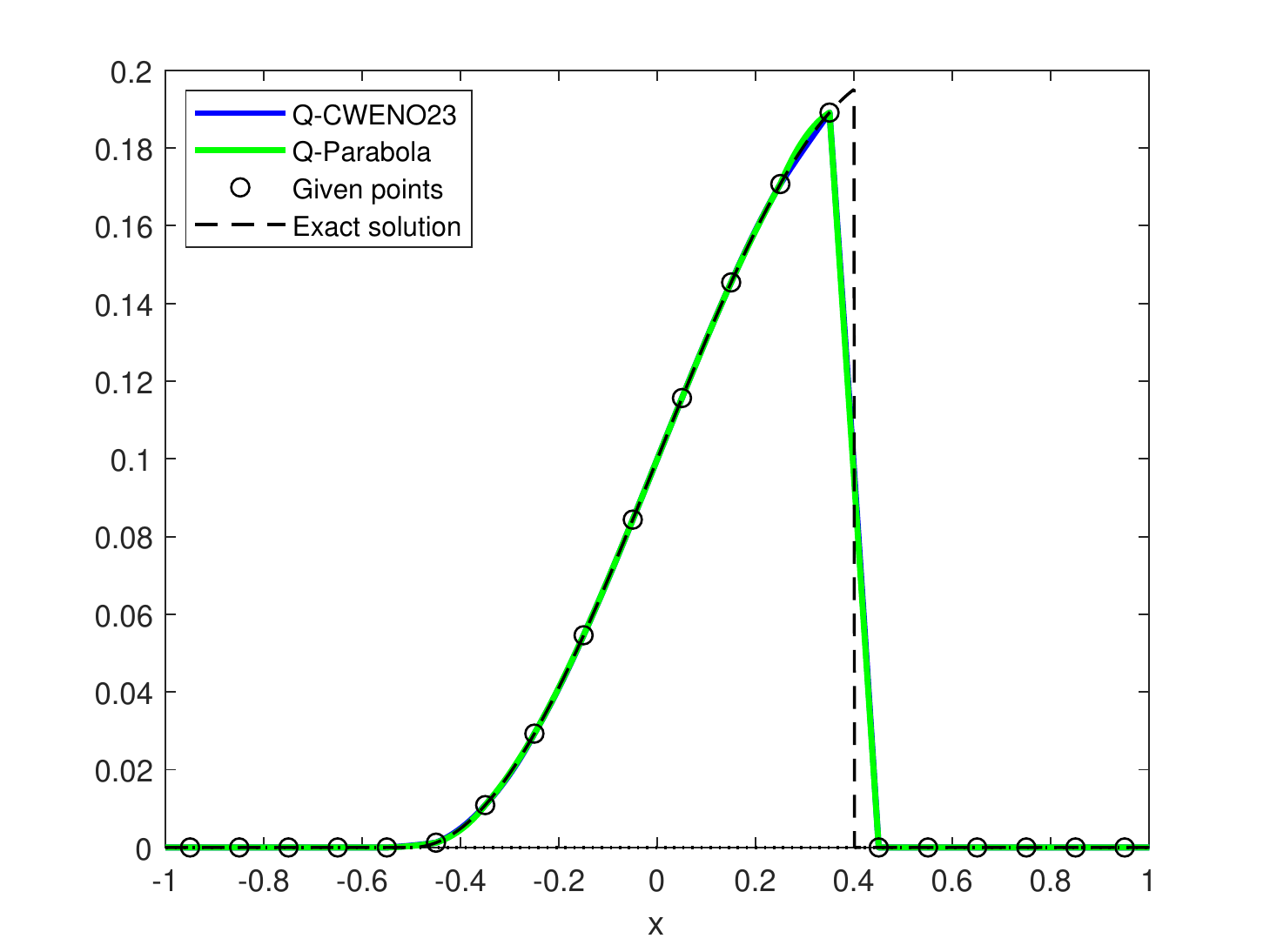}
			\subcaption{Comparison of $\bar{u}_3$ given by \eqref{function for comparison pp} and its reconstructions.}\label{fig: comparison pp 1}
		\end{subfigure}		
		\begin{subfigure}[b]{0.49\linewidth}
			\includegraphics[width=1\linewidth]{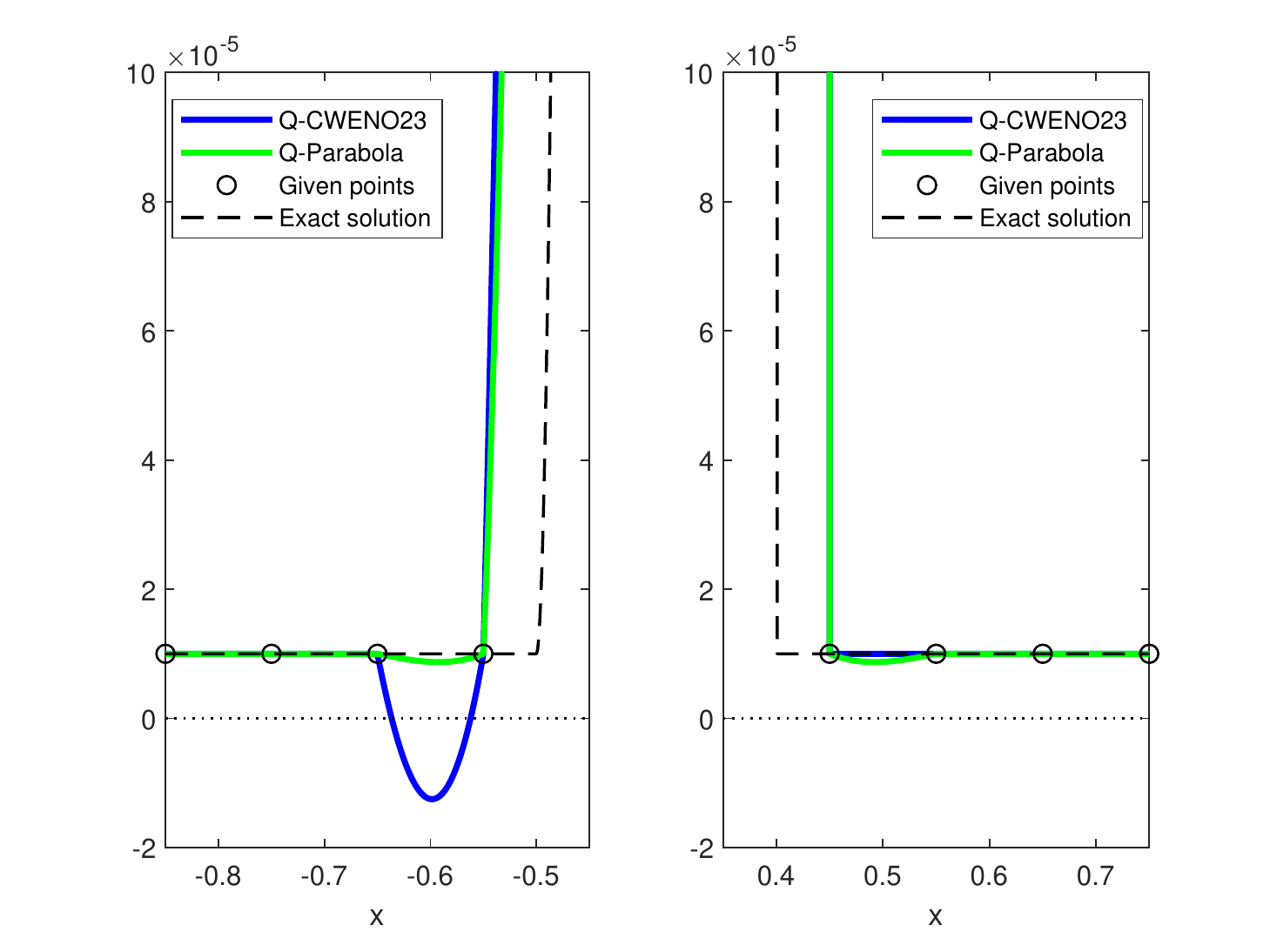}
			\subcaption{Comparison of $\bar{u}_3$ given by \eqref{function for comparison pp} and its reconstructions on $[-0.85,-0.45]$ and $[0.35,0.75]$.
			}\label{fig: comparison pp 1 zoom}
		\end{subfigure}
		\caption{Comparison of reconstructions between Q-Parabola and Q-CWENO23. Dashed lines are exact solutions $\bar{u}_3(x)$ and black circles are given values on grid points of $\bar{u}_3(x)$ given in \eqref{function for comparison pp}.} \label{fig: comparison pp}
	\end{figure}

	Summarizing, our reconstruction works as follows:
	\subsubsection{Algorithm for 1D case}\label{algorithm 1d}
	\begin{enumerate}
		\item Given cell average values $\{\bar{u}_i\}_{i \in \mathcal{I}}$  for each $i \in \mathcal{I}$, reconstruct a polynomial of even degree $k$:$$R_i(x)= \sum_{\ell=0}^k \frac{R_i^{(\ell)}}{\ell !}(x-x_i)^{\ell}$$ which is:
		
		\begin{itemize}
			\item High order accurate in the approximation of smooth $u(x)$:
			\begin{itemize}
				\item If $\ell$ is an even integer such that $0 \leq \ell\leq k$: $\, 					u_i^{(\ell)}=R_i^{(\ell)}+ \mathcal{O}(\Delta x^{k+2-\ell}).$
				\item If $\ell$ is an odd integer such that $0\leq \ell < k$: $ 	\,				u_i^{(\ell)}-u_{i+1}^{(\ell)}=R_i^{(\ell)}-R_{i+1}^{(\ell)}+ \mathcal{O}(\Delta x^{k+2-\ell}).$
			\end{itemize}
			\item Essentially non-oscillatory. 
			\item Positive preserving. 
			\item Conservative in the sense of cell averages: $\displaystyle \frac{1}{\Delta x}\int_{x_{i-\frac{1}{2}}}^{x_{i+\frac{1}{2}}}R_i(x)\,dx = \bar{u}_{i}.$
		\end{itemize}    
		\item Using the obtained values $R_i^{(\ell)}$ for $0 \leq \ell\leq k$, approximate $u(x_{i+\theta})$ with
		\begin{align*}
		Q_{i+\theta}= \sum_{\ell=0}^k (\Delta x)^{\ell} \left(\alpha_{\ell}(\theta)R_{i}^{(\ell)} +\beta_\ell(\theta)R_{i+1}^{(\ell)}\right),
		\end{align*}
		where $\alpha_\ell(\theta)$ and $\beta_\ell(\theta)$ are given in \eqref{form of alpha} and \eqref{form of beta}
	\end{enumerate}
	
	%

	\section{Conservative reconstruction in 2D}\label{sec:Conservative reconstruction in 2D}
	In this section, we introduce the conservative reconstruction technique in two space dimensions, following the one adopted in the previous section. 
	Let $u:\mathbb{R}^2 \rightarrow \mathbb{R}$ be a smooth function and $\bar{u}:\mathbb{R}^2 \rightarrow \mathbb{R}$ be a corresponding sliding average function:
	\[
	\bar{u}(x,y) = \frac{1}{\Delta x\Delta y}\int_{y-\Delta y/2}^{y+\Delta y/2}\int_{x-\Delta x/2}^{x+\Delta x/2}u(x,y)\,dx\,dy.
	\] 
	Given cell averages on grid points,
	\[
	\frac{1}{\Delta x\Delta y}\int_{I_{i,j}}u(x)\,dx = \bar{u}_{i,j}, \quad I_{i,j}=[x_{i-\frac{1}{2}},x_{i+\frac{1}{2}}] \times [y_{j-\frac{1}{2}},y_{j+\frac{1}{2}}],
	\] 
	for each $(i,j) \in \mathcal{I}$, our goal is to approximate the function $\bar{u}(x,y)$. Assume we have a piecewise polynomial reconstruction $R(x,y)=\sum_{i,j} R_{i,j}(x,y) \chi_{i,j}(x,y) $, for $(i,j) \in \mathcal{I}$, where $\chi_{i,j}(x,y)$ is the characteristic function of cell $I_{i,j}$ and each $R_{i,j}(x,y)$ denotes a polynomial of degree $k$ and has the following properties:
	\begin{enumerate}
		\item It is high order accurate in the approximation of $u(x,y)$:
		\begin{align}\label{R accuracy 2d}
		u(x,y) = R_{i,j}(x,y) + \mathcal{O}\left(h^{k+1}\right), \quad (x,y) \in I_{i,j},
		\end{align}
		where $\Delta x, \Delta y = \mathcal{O}(h)$.
		
		\item It is conservative in the sense of cell averages:
		\[
		\frac{1}{\Delta x\Delta y}\int_{y_{j-\frac{1}{2}}}^{y_{j+\frac{1}{2}}}\int_{x_{i-\frac{1}{2}}}^{x_{i+\frac{1}{2}}}R_{i,j}(x,y)\,dx\,dy  = \bar{u}_{i,j}.
		\] 
	\end{enumerate}
	We start from a polynomial of degree $k$, $R_{i,j}(x,y)$:
	\begin{align}\label{R 2d taylor}
	R_{i,j}(x,y)= \sum_{|\ell|=0}^k \frac{R_{i,j}^{(\ell)}}{\ell_1 ! \ell_2 !}(x-x_i)^{\ell_1}(y-y_j)^{\ell_2},
	\end{align}
	where we use a multi index $\ell=(\ell_1,\ell_2)$.
	\begin{figure}[t]
		\centering
		\includegraphics[width=0.55\linewidth]{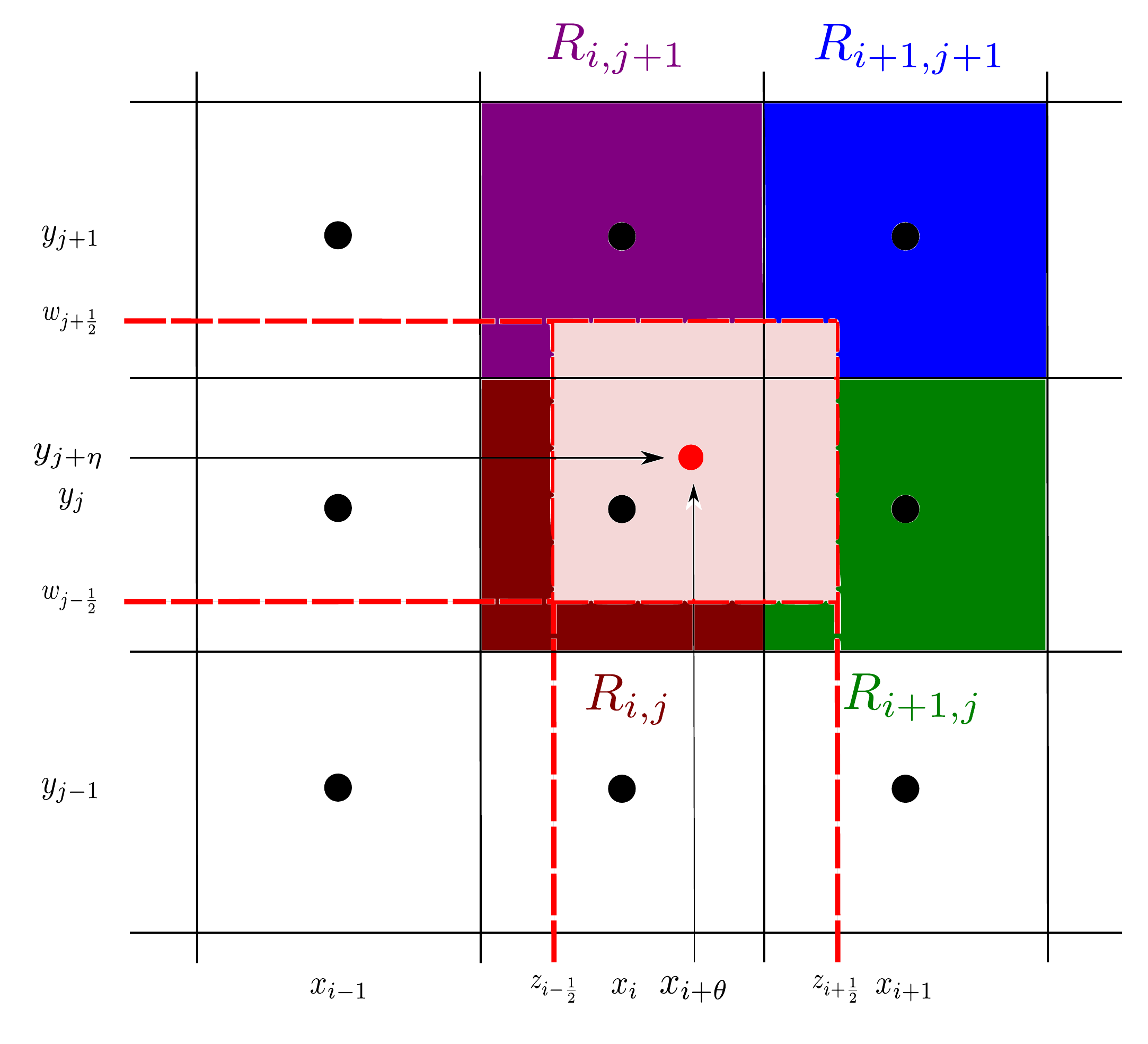}
		\caption{Description of two-dimensional conservative reconstruction}\label{2d interpolation}  	
	\end{figure}
	Consider a cell $I_{i,j}^{\theta,\eta}$ whose center is $(x_{i+\theta},y_{j+\eta})$ for some $\theta,\eta \in [0,1)$. In Fig. \ref{2d interpolation}, we note that $(x_{i+\theta},y_{j+\eta})$ lies inside one of $I_{i,j},I_{i+1,j},I_{i,j+1},I_{i+1,j+1}$. Let us denote a cell $I_{i+\theta,j+\eta}:=[z_{i-\frac{1}{2}},z_{i+\frac{1}{2}}] \times [w_{j-\frac{1}{2}},w_{j+\frac{1}{2}}]$ and a point $(x_{i+\theta},y_{j+\eta}):= (x_i+\theta\Delta x,y_j+\eta\Delta y)$. Now, we approximate $\bar{u}(x_{i+\theta},y_{j+\eta})$ by
	\begin{align*}
	\bar{u}(x_{i+\theta},y_{j+\eta}) &\approx \frac{1}{\Delta x\Delta y}\int_{w_{j-\frac{1}{2}}}^{w_{j+\frac{1}{2}}}\int_{z_{i-\frac{1}{2}}}^{z_{i+\frac{1}{2}}}R(x,y) \,dx\,dy\cr
	&=\frac{1}{\Delta x\Delta y}\int_{w_{j-\frac{1}{2}}}^{y_{j+\frac{1}{2}}}\int_{z_{i-\frac{1}{2}}}^{x_{i+\frac{1}{2}}}R_{i,j}(x,y) \,dx\,dy+\frac{1}{\Delta x\Delta y}\int_{w_{j-\frac{1}{2}}}^{y_{j+\frac{1}{2}}}\int_{x_{i+\frac{1}{2}}}^{z_{i+\frac{1}{2}}}R_{i+1,j}(x,y) \,dx\,dy\cr
	&+\frac{1}{\Delta x\Delta y}\int_{y_{j+\frac{1}{2}}}^{w_{j+\frac{1}{2}}}\int_{z_{i-\frac{1}{2}}}^{x_{i+\frac{1}{2}}}R_{i,j+1}(x,y) \,dx\,dy+\frac{1}{\Delta x\Delta y}\int_{y_{j+\frac{1}{2}}}^{w_{j+\frac{1}{2}}}\int_{x_{i+\frac{1}{2}}}^{z_{i+\frac{1}{2}}}R_{i+1,j+1}(x,y) \,dx\,dy.
	\end{align*}  
	The first integral becomes 
	\begin{align*}
	\frac{1}{\Delta x\Delta y}\int_{w_{j-\frac{1}{2}}}^{y_{j+\frac{1}{2}}}\int_{z_{i-\frac{1}{2}}}^{x_{i+\frac{1}{2}}}R_{i,j}(x,y) \,dx\,dy
	&=\frac{1}{\Delta x\Delta y}\sum_{|\ell|=0}^k R_{i,j}^{(\ell)}\int_{y_{j-\frac{1}{2}+\eta}}^{y_{j+\frac{1}{2}}}\int_{x_{i-\frac{1}{2}+\theta}}^{x_{i+\frac{1}{2}}} \frac{1}{\ell_1 ! \ell_2 !}(x-x_i)^{\ell_1}(y-y_j)^{\ell_2}
	\,dx\,dy\cr
	&=\sum_{|\ell|=0}^k R_{i,j}^{(\ell)} \left(\frac{1}{\Delta x}\int_{x_{i-\frac{1}{2}+\theta}}^{x_{i+\frac{1}{2}}} \frac{(x-x_i)^{\ell_1}}{\ell_1 !}\,dx\right)\left(\frac{1}{\Delta y}\int_{y_{j-\frac{1}{2}+\eta}}^{y_{j+\frac{1}{2}}}\frac{(y-y_j)^{\ell_2}}{\ell_2 !}\,dy\right)\cr
	&=\sum_{|\ell|=0}^k (\Delta)^{\ell} \alpha_{\ell_1}(\theta)\alpha_{\ell_2}(\eta)R_{i,j}^{(\ell)}
	\end{align*}  
	where $(\Delta)^{\ell}=(\Delta x)^{\ell_1}(\Delta y)^{\ell_2}$. Similarly, we obtain 
	\begin{align*}
	\frac{1}{\Delta x\Delta y}\int_{w_{j-\frac{1}{2}}}^{y_{j+\frac{1}{2}}}\int_{x_{i+\frac{1}{2}}}^{z_{i+\frac{1}{2}}}R_{i+1,j}(x,y) \,dx\,dy
	&=\sum_{|\ell|=0}^k (\Delta)^{\ell} \beta_{\ell_1}(\theta)\alpha_{\ell_2}(\eta)R_{i+1,j}^{(\ell)},\cr
	\frac{1}{\Delta x\Delta y}\int_{y_{j+\frac{1}{2}}}^{w_{j+\frac{1}{2}}}\int_{z_{i-\frac{1}{2}}}^{x_{i+\frac{1}{2}}}R_{i,j+1}(x,y) \,dx\,dy
	&=\sum_{|\ell|=0}^k (\Delta)^{\ell} \alpha_{\ell_1}(\theta)\beta_{\ell_2}(\eta)R_{i,j+1}^{(\ell)},\cr
	\frac{1}{\Delta x\Delta y}\int_{y_{j+\frac{1}{2}}}^{w_{j+\frac{1}{2}}}\int_{x_{i+\frac{1}{2}}}^{z_{i+\frac{1}{2}}}R_{i+1,j+1}(x,y) \,dx\,dy
	&=\sum_{|\ell|=0}^k (\Delta)^{\ell} \beta_{\ell_1}(\theta)\beta_{\ell_2}(\eta)R_{i+1,j+1}^{(\ell)}.
	\end{align*}
	Denoting the approximation of $\bar{u}(x_{i+\theta},y_{j+\theta})$ by $Q_{i+\theta,j+\eta}$, we write it as
	\begin{align}\label{scheme 2d}
	\begin{split}
	Q_{i+\theta,j+\eta}=\sum_{|\ell|=0}^k (\Delta)^{\ell}&\bigg(\alpha_{\ell_1}(\theta)\alpha_{\ell_2}(\eta)R_{i,j}^{(\ell)}+ 
	\beta_{\ell_1}(\theta)\alpha_{\ell_2}(\eta)R_{i+1,j}^{(\ell)}\cr
	&+ 
	\alpha_{\ell_1}(\theta)\beta_{\ell_2}(\eta)R_{i,j+1}^{(\ell)}+
	\beta_{\ell_1}(\theta)\beta_{\ell_2}(\eta)R_{i+1,j+1}^{(\ell)}\bigg),
	\end{split}
	\end{align}
	where the explicit forms of $\alpha_{\ell_1}(\theta)$, $\alpha_{\ell_2}(\eta)$, $\beta_{\ell_1}(\theta)$, $\beta_{\ell_2}(\eta)$ are given in \eqref{form of alpha} and \eqref{form of beta}.

	\subsection{General Properties}
	In the following proposition, as in Proposition \ref{prop con 1d}, we show that the approximation $Q_{i+\theta, j+\eta}$ is of order $(k+2)$ of accuracy for an even integer $k \geq 0$. For simplicity, we assume $\Delta x, \Delta y = h >0.$
	\begin{proposition}\label{Consistency 2d}
		Let $k\geq 0$ be an even integer and $u$ be smooth enough so that a piecewise polynomial $R(x,y)=\sum_{i,j} R_{i,j}(x,y) \chi_{i,j} $ satifies	
		\begin{align}\label{consistency condition 2d}
		\begin{split}
		u_{i,j}^{(\ell)}&=R_{i,j}^{(\ell)}+ \mathcal{O}(h^{k+2-|\ell|}), \quad \text{$\ell \in A$}\cr
		u_{i,j}^{(\ell)}-u_{i+1,j}^{(\ell)}&=R_{i,j}^{(\ell)}-R_{i+1,j}^{(\ell)}+ \mathcal{O}(h^{k+2-|\ell|}), \quad \text{$\ell \in B$}\cr
		u_{i,j}^{(\ell)}-u_{i,j+1}^{(\ell)}&=R_{i,j}^{(\ell)}-R_{i,j+1}^{(\ell)}+ \mathcal{O}(h^{k+2-|\ell|}), \quad \text{$\ell \in C$}
		\end{split}
		\end{align}
		where the set $A,B$ and $C$ are defined 
		\begin{align}\label{decomposition1}
		\begin{split}
		A&=\{\ell : |\ell|=\text{even},\quad  0 \leq |\ell|\leq k\},\cr
		B&=\{\ell :  \ell_1=\text{odd},\quad \ell_2=\text{even} ,\quad  0 \leq |\ell|\leq k\},\cr 
		C&=\{\ell :  \ell_1=\text{even},\quad \ell_2=\text{odd} ,\quad  0 \leq |\ell|\leq k\}.
		\end{split}
		\end{align}
		Then the reconstruction $Q_{i+\theta,j+\eta}$ gives a $(k+2)$-th-order approximation of sliding averages $\bar{u}_{i+\theta,j+\eta}$ for any $\theta,\eta \in [0,1).$
	\end{proposition}	
	\begin{proof}
		For detailed proof, see \ref{Appendix accuracy proof 2d}.
	\end{proof}

	The conservation property also holds in the 2D reconstruction \eqref{scheme 2d}:
	\begin{proposition}\label{prop con 2d}
		%
		%
		Assume that $R_{i,j}(x,y)$ satisfies
		\begin{align*}
		\frac{1}{\Delta x\Delta y}\int_{y_{j-\frac{1}{2}}}^{y_{j+\frac{1}{2}}}\int_{x_{i-\frac{1}{2}}}^{x_{i+\frac{1}{2}}}R_{i,j}(x,y)\,dx\,dy  = \bar{u}_{i,j},\quad (i,j) \in \mathcal{I}.
		\end{align*}
		Then, for periodic functions $\bar{u}(x,y)$ with period $(L,L)=(N h,N h),\, N \in \mathbb{N}$ 
		\[
		\sum_{1 \leq i,j \leq N} Q_{i+\theta,j+\theta} = \sum_{1 \leq i,j \leq N} \bar{u}_{i,j}, 
		\]
		for any $\theta,\eta \in [0,1)$.
	\end{proposition}
	\begin{proof}
		The proof is similar to the one dimensional case. 
	\end{proof}
	
	\subsubsection{Algorithm for 2D case}\label{algorithm 2d}
	\begin{enumerate}
		\item Given cell average values $\{\bar{u}_{i,j}\}_{(i,j) \in \mathcal{I}}$ for each $(i,j) \in \mathcal{I}$, reconstruct a polynomial of degree $k$: $$R_{i,j}(x,y)= \sum_{|\ell|=0}^k \frac{R_{i,j}^{(\ell)}}{\ell_1 ! \ell_2 !}(x-x_i)^{\ell_1}(y-y_j)^{\ell_2}$$  which is:
		\begin{itemize}
			\item High order accurate in the approximation of $u(x)$:
			\begin{align*}	
			\begin{split}
			u_{i,j}^{(\ell)}&=R_{i,j}^{(\ell)}+ \mathcal{O}(h^{k+2-|\ell|}), \quad \text{$\ell \in A$}\cr
			u_{i,j}^{(\ell)}-u_{i+1,j}^{(\ell)}&=R_{i,j}^{(\ell)}-R_{i+1,j}^{(\ell)}+ \mathcal{O}(h^{k+2-|\ell|}), \quad \text{$\ell \in B$}\cr
			u_{i,j}^{(\ell)}-u_{i,j+1}^{(\ell)}&=R_{i,j}^{(\ell)}-R_{i,j+1}^{(\ell)}+ \mathcal{O}(h^{k+2-|\ell|}), \quad \text{$\ell \in C$}
			\end{split}
			\end{align*}
			where sets $A,B,C$ are defined in \eqref{decomposition1}.
			\item Essentially non-oscillatory.
			\item Positive preserving. 		 
			\item Conservative in the sense of cell averages: $ 	\displaystyle	\frac{1}{\Delta x \Delta y}\int_{y_{j-\frac{1}{2}}}^{y_{j+\frac{1}{2}}}\int_{x_{i-\frac{1}{2}}}^{x_{i+\frac{1}{2}}}R_{i,j}(x,y)\,dx\,dy = \bar{u}_{i,j}.$
		\end{itemize}    
		\item Using the obtained values $R_{i,j}^{(\ell)}$ for $0 \leq |\ell|\leq k$, approximate $\bar{u}(x_{i+\theta},y_{j+\eta})$ with
		\begin{align*}
		Q_{i+\theta,j+\eta}=\sum_{|\ell|=0}^k (\Delta)^{\ell}&\bigg(\alpha_{\ell_1}(\theta)\alpha_{\ell_2}(\eta)R_{i,j}^{(\ell)}+ 
		\beta_{\ell_1}(\theta)\alpha_{\ell_2}(\eta)R_{i+1,j}^{(\ell)}\cr
		&+ 
		\alpha_{\ell_1}(\theta)\beta_{\ell_2}(\eta)R_{i,j+1}^{(\ell)}+
		\beta_{\ell_1}(\theta)\beta_{\ell_2}(\eta)R_{i+1,j+1}^{(\ell)}\bigg),
		\end{align*}
		where $\alpha_{\ell_1}(\theta)$, $\alpha_{\ell_2}(\eta)$, $\beta_{\ell_1}(\theta)$, $\beta_{\ell_2}(\eta)$ are computable using \eqref{form of alpha} and \eqref{form of beta}.
	\end{enumerate}

	
	\section{Semi-Lagrangian schemes for hyperbolic systems with relaxation}\label{sec:A semi-lagrangian scheme for semi-linear hyperbolic relaxation}
	
	In this section, as an application of the conservative reconstruction \eqref{scheme} and \eqref{scheme 2d}, we consider semi-Lagrangian methods to semi-linear hyperbolic relaxation systems. Two semi-linear hyperbolic relaxation system, namely, Xin-Jin system \cite{jin1996numerical} and Broadwell model \cite{broadwell1964shock}, where a relaxation parameter $\kappa$ makes each system stiff as $\kappa \to 0$.
	
	In order to treat the stiffness, we shall use L-stable $s$-stage DIRK methods or L-stable linear multi-step methods (in particular BDF methods) \cite{HW}. These methods provide a balanced performance between stability and efficiency. 
	
	From now on, we focus on L-stable $s$-stage DIRK methods represented by Butcher's tables:
	\begin{align*}
	\begin{array}{c|c}
	c & A \\
	\hline
	& b^{T} 
	\end{array}
	\end{align*}
	where $A=[a_{k\ell}]$ is a $s \times s$ lower triangle matrix such that $a_{k\ell}=0$ for $\ell>k$, $c= (c_1,...,c_s)^{T}$ and
	$b=(b_1,...,b_s)^{T}$ are coefficient vectors. (For BDF based methods, we refer to \ref{BDF section Xin-Jin}.)
	
	In order to guarantee L-stability, here we make use of \emph{stiffly accurate} schemes (SA), i.e. schemes for which the last row of matrix A is equal to the vector of weights $ a_{sj} = b_j$, $j = 1, ... ,s$. This will ensure that the absolute stability function vanishes at infinity. As a consequence, an A-stable scheme which is SA is also L-stable, \cite{HW}. 
	
	In the numerical tests for each order of accuracy, we will use the following high-order L-stable DIRK methods:
	\begin{itemize} 
		\item second-order DIRK method (DIRK2) \cite{HWN}, 
		\begin{align}\label{Butcher_1}
		\begin{array}{c|c c}
		\alpha & \alpha & 0 \\
		1 & 1-\alpha & \alpha \\
		\hline
		& 1-\alpha & \alpha
		\end{array},\quad \alpha= 1 -\frac{\sqrt{2}}{2}.
		\end{align} 	
		\item third-order DIRK method (DIRK43) \cite{KC},
		\small{
			\begin{align}\label{Butcher_2}
			\begin{split}
			\begin{array}{c|c c c c}
			0 & 0 & 0 & 0 & 0\\[1.5mm]
			\displaystyle
			2\gamma  & 
			\displaystyle
			\gamma  &
			\displaystyle
			\gamma  & 0&  0\\ [1.5mm]
			\displaystyle
			c &\displaystyle
			c - \delta -\gamma
			&\displaystyle
			\delta & \displaystyle
			\gamma& 0\\[1.5mm]
			1 & \displaystyle
			1-b_2 - b_3 -\gamma  & \displaystyle
			b_2 & \displaystyle
			b_3 &\displaystyle
			\gamma\\[1.5mm]
			\hline \\
			& \displaystyle
			1-b_2 - b_3 -\gamma  & \displaystyle
			b_2 & \displaystyle
			b_3 & \displaystyle
			\gamma\\		
			\end{array}
			\end{split}
			\end{align}}
		
	\end{itemize}
	with $\displaystyle\gamma = \frac{1767732205903}{4055673282236}\, $,  $\displaystyle c = \frac{3}{5}\,$,  $\displaystyle b_2= -\frac{4482444167858}{7529755066697}\, $,  $\displaystyle b_3= \frac{11266239266428}{11593286722821}\,$, $\displaystyle \delta = -\frac{640167445237}{6845629431997}$.

	\subsection{Xin-Jin relaxation system}
	Consider a simplified Xin-Jin relaxation system \cite{jin1996numerical}:
	\begin{align}\label{relaxation system}
	\begin{split}
	\frac{\partial u}{\partial t}+\sum_{i=1}^d \frac{\partial v}{\partial {x_i}}&=0,\cr
	\frac{\partial v}{\partial t}+a^2\sum_{i=1}^d \frac{\partial u}{\partial {x_i}} &=\frac{1}{\kappa}(F(u)-v),\cr
	\end{split}
	\end{align}
	where $d$ denotes the dimension of space variable. When $\kappa$ goes to zero, the solution in \eqref{relaxation system} converges to 
	\begin{align}\label{limit}
	\frac{\partial u}{\partial t} + \sum_{i=1}^d  \frac{\partial F(u)}{\partial {x_i}} =0, \quad v=F(u).
	\end{align}
	provided that the subcharacteristic condition is satisfied, i.e., $\max_u{|F'(u)|} \leq |a|$ (see \cite{chen1994hyperbolic}). For example, taking $F(u)=u^2/2$, the system \eqref{limit} formally becomes the Burgers equation:
	\begin{align}\label{burgers}
	\frac{\partial u}{\partial {t}} + \sum_{i=1}^d u\frac{\partial u}{\partial {x_i}} =0, \quad v=\frac{u^2}{2}.
	\end{align}
	In this equation, shocks may appear in a finite time and we need to impose our scheme to be conservative to capture the positions of such shocks correctly. We treat this shock problem in section \ref{sec:Numerical test}.
	
	\subsubsection{Semi-Lagrangian scheme for Xin-Jin relaxation system}
	Using $u-v=f$ and $u+v=g$, we rewrite \eqref{relaxation system} as
	\begin{align}\label{hyper system}
	\begin{split}
	\frac{\partial f}{\partial {t}}-\sum_{i=1}^d \frac{\partial f}{\partial {x_i}} &=-\frac{1}{\kappa}\left[F\left(\frac{g+f}{2}\right)-\frac{g-f}{2}\right]\cr
	\frac{\partial g}{\partial {t}}+\sum_{i=1}^d \frac{\partial g}{\partial {x_i}}   &=-\frac{1}{\kappa}\left[\frac{g-f}{2}-F\left(\frac{g+f}{2}\right)\right].
	\end{split}
	\end{align}
	Based on this, we consider its Lagrangian formulation:
	\begin{align}\label{hyper system lagrangian}
	\begin{split}
	\frac{df}{dt}(X_1(t),t)&=-\frac{1}{\kappa}\left[F\left(\frac{g+f}{2}\right)-\frac{g-f}{2}\right](X_1(t),t),\quad \frac{dX_1}{dt}=-\mathbbm{1}\cr
	\frac{dg}{dt}(X_2(t),t)&=-\frac{1}{\kappa}\left[\frac{g-f}{2}-F\left(\frac{g+f}{2}\right)\right](X_2(t),t),\quad \frac{dX_2}{dt}=\mathbbm{1},
	\end{split}
	\end{align}
	where $\mathbbm{1}=(1,\cdots,1) \in \mathbb{N}^d$ and $X_1(t^{n+1}) = X_2(t^{n+1})= x_i \in \mathbb{R}^d$.

	
	To clarify high order methods for \eqref{hyper system lagrangian}, we introduce the following notation:
	\begin{itemize}
		\item The $\ell$-th stage values of $f,g$ along the backward-characteristics which come from $x_i$ with characteristic speed $-1,1$ at time $t^n + c_k \Delta t$:
		\begin{align*}
		\tilde{f}_{i}^{(k,\ell)} &\approx f(x_{i}+(c_k-c_\ell)\Delta t, t^n + c_\ell\Delta t),\quad  		\tilde{g}_{i}^{(k,\ell)} \approx f(x_{i}-(c_k-c_\ell)\Delta t, t^n + c_\ell\Delta t)
		\end{align*}
		where "$\approx$" implies the necessity of suitable reconstructions. We also denote $k$-th stage value of $f,g$ on $x_i$ by
		\begin{align*}
		f_{i}^{(k)} &= f(x_{i}, t^n + c_k\Delta t),\quad g_{i}^{(k)} = g(x_{i}, t^n + c_k\Delta t)
		\end{align*}
		for $1 \leq k \leq s$ where $f_{i}^{(k)} = u_{i}^{(k)} - v_i^{(k)}$ and $g_{i}^{(k)} = u_{i}^{(k)} + v_i^{(k)}$. 
		\item
		For $\ell=0$, we set $c_\ell=0$ hence
		\begin{equation}\label{tilde}
		\tilde{f}_{i}^{(k,0)} \approx f(x_{i}+ c_k\Delta t, t^n),\quad \tilde{g}_{i}^{(k,0)} \approx g(x_{i}- c_k\Delta t, t^n).
		\end{equation}
		
		\item Define a RK flux function by $K_1:=F(u)-v$, $K_2:=-K_1$, then
		\[
		K_{i,j}^{(k,\ell)} \approx K_j(x_{i}-\lambda_j(c_k-c_\ell)\Delta t, t^n + c_\ell\Delta t), \quad j = 1,2 	\]
		where $\lambda_1=-1,\lambda_2=1$ and 
		$K_{i,j}^{(k)}  = K_j(x_{i}, t^n + c_k\Delta t)$. 
	\end{itemize} 
	
	With these, we can represent a high order method compactly. Applying a L-stable $s$-stage DIRK method to system (\ref{hyper system lagrangian}), we have $k$-stage values
	\begin{align}\label{hyper system1 lagrangian}
	\begin{split}
	f_{i}^{(k)}&=\tilde{f}_i^{(k,0)}-\frac{ \Delta t}{\kappa}\sum_{\ell = 1}^s a_{k\ell}K_{i,1}^{(k,\ell)},\cr
	g_i^{(k)}&=\tilde{g}_i^{(k,0)}-\frac{ \Delta t}{\kappa}\sum_{\ell = 1}^s a_{k\ell}K_{i,2}^{(k,\ell)},
	\end{split}
	\end{align}
	for $k=1,\dots,s$. Since we only consider SA DIRK schemes, the $s$-stage values become the numerical solutions: $f^{n+1}_{i} = f_{i}^{(s)}$ and $g^{n+1}_{i} = g_{i}^{(s)}$. It is worth mentioning that each $k$-stage value can be computed in an explicit way. After summing and subtracting two equations in (\ref{hyper system1 lagrangian}), we obtain
	\begin{align}\label{hyper system1 lagrangian 2}
	\begin{split}
	u_i^{(k)}&=\frac{\tilde{g}_i^{(k,0)} + \tilde{f}_i^{(k,0)}}{2} -\frac{\Delta t}{2\kappa}\sum_{\ell = 1}^{k -1}a_{k\ell} \left(K_{i,1}^{(k,\ell)}+ K_{i,2}^{(k,\ell)}\right),\cr
	v_i^{(k)}&=\frac{\tilde{g}_i^{(k,0)}-\tilde{f}_i^{(k,0)}}{2}  -\frac{\Delta t}{2\kappa} \left(\sum_{\ell = 1}^{k -1}a_{k\ell}\left(K_{i,2}^{(k,\ell)}-K_{i,1}^{(k,\ell)}\right)\right) + 
	\frac{a_{kk} \Delta t}{\kappa} \left(F(u_i^{(k)}) -v_i^{(k)}\right).
	\end{split}
	\end{align} 
	Here we first compute $u_i^{(k)}$, and use it obtain $v_i^{(k)}$. Now we illustrate our L-stable DIRK schemes as follows: (A schematic for DIRK2 based scheme is given in Fig \ref{RK2 schematic}.)
	\begin{figure}[!t]
		\centering
		\begin{subfigure}[b]{0.9\linewidth}
			\includegraphics[width=0.8\linewidth]{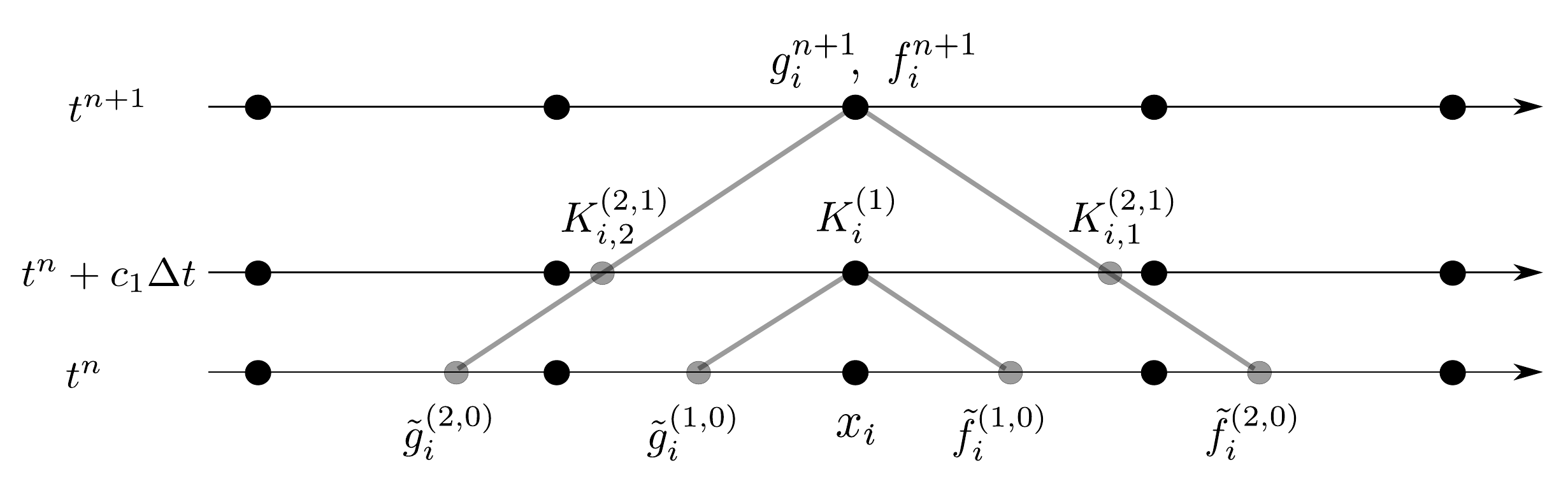}
		\end{subfigure}	
		\caption{Schematic of DIRK2 based SL method for Xin-Jin model. Gray circles are points where reconstruction is required.}\label{RK2 schematic}  	
	\end{figure}

	\subsubsection{Algorithm of $s$-stage L-stable DIRK method}\label{algorithm}
	
	For $k=1,\dots,s$.
	\begin{enumerate}
		\item Interpolate $\tilde{f}_{i}^{(k,0)}$ and $\tilde{g}_{i}^{(k,0)}$ on $x_{i}+c_k\Delta t$ and $x_{i}-c_k\Delta t$ from $\{f_{i}^{n}\}$ and $\{g_{i}^{n}\}$, respectively.
		\item
		Compute $u_i^{(k)}$ and $v_i^{(k)}$ from (\ref{hyper system1 lagrangian 2}).
		\item Compute:
		\begin{equation}\label{first order solution}
		f_i^{(k)}=u_i^{(k)}-v_i^{(k)}, \quad g_i^{(k)}=u_i^{(k)} + v_i^{(k)}
		\end{equation}
		
		\item If $k<s$, compute
		\begin{align*}
		K_{i,1}^{(k)}&=F(u_i^{(k)})-v_i^{(k)}, \quad K_{i,2}^{(k)}=-K_{i,1}^{(k)}
		\end{align*}
		and, for $\ell = k+1, \cdots, s$, interpolate 
		\begin{align*}
		&\text{$K_{i,1}^{(\ell,k)}$ on $x_{i}+(c_\ell-c_k)\Delta t$  from $\{K_{i,1}^{(k)}\}$},\cr
		&\text{$K_{i,2}^{(\ell,k)}$ on $x_{i}-(c_\ell-c_k)\Delta t$ from $\{K_{i,2}^{(k)}\}$}.
		\end{align*}
		\item Compute numerical solution: $f^{n+1}_{i} = f_{i}^{(s)}$ and $g^{n+1}_{i} = g_{i}^{(s)}$.
	\end{enumerate}
	For any term where reconstruction is required, we use the formula \eqref{scheme} based on the CWENO reconstructions.

	\begin{remark}
		In the Algorithm \ref{algorithm}, using the implicit Euler method for $s = 1$ and taking a limit $\kappa \to 0$ in (\ref{hyper system1 lagrangian 2}), we obtain 
		\begin{align}\label{ap}
		u_i^{n+1}&=\frac{\tilde{g}_i^{(1,0)} + \tilde{f}_i^{(1,0)}}{2}, \quad v_i^{n+1}= F(u_i^{n+1}),
		\end{align}
		for all $n\geq 0$ regardless of initial data. 
		Now assume that $\Delta t = \Delta x$, and we combine \eqref{ap} with \eqref{tilde}, and \eqref{first order solution} for $k = 1$ obtaining
		\begin{align*}
		\begin{split}
		u_i^{n+1}&=\frac{1}{2}\left(u_{i+1}^n + u_{i-1}^n\right) - \frac{1}{2}\left(F\left(u_{i+1}^n\right) - F\left(u_{i-1}^n\right)\right).
		\end{split}
		\end{align*}
		This is the Lax–Friedrichs method of the conservation law in \eqref{limit}
		with $\Delta t= \Delta x$. 
	\end{remark}

	\subsection{Broadwell model}
	Next example is the Broadwell model of kinetic theory \cite{broadwell1964shock}:
	\begin{align}\label{Broadwell}
	\begin{split}
	\partial_t f+\partial_x f&= \frac{1}{\kappa}Q\cr
	\partial_t g-\partial_x g&= \frac{1}{\kappa}Q\cr
	\partial_t h&= -\frac{1}{\kappa}Q.
	\end{split}
	\end{align}	
	where $Q=h^2-fg$. Introducing the fluid dynamic moment variables 
	dentity $\rho$, momentum $m$, and velocity $u$ and an additional variable $z$ as follows:
	\begin{align}\label{rhomz def}
	\rho= f+2h+g, \quad m= f-g, \quad z= f+g,
	\end{align}
	the system \eqref{Broadwell} can be rewritten as
	\begin{align}\label{Broadwell 2}
	\begin{split}
	&\partial_t \rho +\partial_x m= 0\cr
	&\partial_t m+\partial_x z= 0\cr
	&\partial_t z +\partial_x m = \frac{1}{2\kappa}\left(\rho^2-2\rho z +m^2\right).
	\end{split}
	\end{align}	
	Note that the original variables can be recovered by 
	\[
	f= \frac{z+m}{2},\quad g= \frac{z-m}{2},\quad h= \frac{\rho -z}{2}.
	\]
	As $\kappa \rightarrow 0$, one can see that $z$ goes to a local equibrium
	\[
	z\rightarrow z_E(\rho,m):= \frac{1}{2\rho}\left(\rho^2 + m^2\right) = \frac{1}{2}\left(\rho + \rho u^2\right)
	\]
	and the system \eqref{Broadwell 2} becomes the Euler equations:
	\begin{align}\label{Broadwell 3}
	\begin{split}
	&\partial_t \rho +\partial_x m= 0\cr
	&\partial_t m+\partial_x \left(\frac{1}{2}\left(\rho + \rho u^2\right)\right)= 0.
	\end{split}
	\end{align}

	\subsubsection{Semi-Lagrangian scheme for the Broadwell model} 
	Here, we consider again DIRK methods based on Tables \eqref{Butcher_1}-\eqref{Butcher_2}. The schemes are also explicitly solvable with algebraic computations. (For BDF methods, we refer to \eqref{BDF methods for Broadwell model}.)
	
	Let us denote $k$-th stage values by $f_i^{(k)}$, $g_i^{(k)}$, $h_i^{(k)}$, $1\leq k\leq s$, and introduce the following notation:
	\begin{align*}
	Q_i^{(k)}&= (h_i^{(k)})^2-f_i^{(k)}g_i^{(k)}, \cr
	Q_{i,1}^{(k,\ell)}&\approx Q(x_i-(c_k-c_\ell)\Delta t,t^n + c_\ell \Delta t), \quad Q_{i,2}^{(k,\ell)} \approx Q(x_i+(c_k-c_\ell)\Delta t,t^n + c_\ell \Delta t),	\cr	
	f_i^{(k,\ell)}&\approx f(x_i-(c_k- c_\ell) \Delta t,t^n + c_\ell \Delta t), \quad g_i^{(k,\ell)}\approx g(x_i-(c_k- c_\ell) \Delta t,t^n + c_\ell \Delta t).
	\end{align*}
	Applying a $s$-stage DIRK method to \eqref{Broadwell}, we can write $k$-th stage values in a compact form: 
	%
	\begin{align}\label{s-DIRK rewrite}
	\begin{split}
	f_i^{(k)}&= F_i^{(k)} + \frac{a_{kk}\Delta t}{\kappa} Q_i^{(k)},\quad F_i^{(k)}:=f_i^{(k,0)} + \frac{\Delta t}{\kappa}\sum_{\ell=1}^{k-1} a_{k\ell}Q_{i,1}^{(k,\ell)}\cr
	g_i^{(k)}&= G_i^{(k)} + \frac{a_{kk}\Delta t}{\kappa}Q_i^{(k)},\quad G_i^{(k)}:=g_i^{(k,0)} + \frac{\Delta t}{\kappa}\sum_{\ell=1}^{k-1} a_{k\ell}Q_{i,2}^{(k,\ell)}\cr
	h_i^{(k)}&= H_i^{(k)} - \frac{a_{kk}\Delta t}{\kappa} Q_i^{(k)},\quad H_i^{(k)}:=h_i^{n} - \frac{\Delta t}{\kappa}\sum_{\ell=1}^{k-1} a_{k\ell}Q_{i}^{(\ell)},
	\end{split}
	\end{align}
	for $k=1,2,\dots, s$. Here the SA property also implies $f_i^{n+1} = f_i^{(s)}$, $g_i^{n+1} =g_i^{(s)}$ and $h_i^{n+1}= h_i^{(s)}$.
	
	Now, we describe the algorithm:
	\subsubsection{Algorithm of $s$-stage L-stable DIRK method}\label{algorithm b}
	For $k=1,\cdots,s$, iterate the following procedures:
	\begin{enumerate}
		\item Reconstruct $f_i^{(k,0)}$ and 
		$g_i^{(k,0)}$ on $x_{i}-c_k\Delta t$ and $x_{i}+c_k\Delta t$ from $\{f_{i}^{n}\}$ and $\{g_{i}^{n}\}$, respectively.
		\item Reconstruct $Q_{i,1}^{(k,\ell)}$ and 
		$Q_{i,2}^{(k,\ell)}$ for $\ell=1,\cdots,k-1$ from $\{Q_i^{(\ell)}\}$ (skip this if $k=1$).
		\item Compute $F_i^{(k)}$, $G_i^{(k)}$ and $H_i^{(k)}$ using \eqref{s-DIRK rewrite}.
		\item Solve
		\begin{align}\label{first stage broadwell}
		\begin{split}
		f_i^{(k)}= F_i^{(k)} + \frac{a_{kk}\Delta t}{\kappa} Q_i^{(k)},\quad 
		g_i^{(k)}= G_i^{(k)} + \frac{a_{kk}\Delta t}{\kappa}Q_i^{(k)},\quad
		h_i^{(k)}= H_i^{(k)} - \frac{a_{kk}\Delta t}{\kappa} Q_i^{(k)}
		\end{split}
		\end{align}
		for 
		\begin{align}\label{hfg}
		\begin{split}
		&h_i^{(k)} = \frac{a_{kk}\Delta t( H_i^{(k)}+ F_i^{(k)} )( H_i^{(k)} + G_i^{(k)})+ \kappa H_i^{(k)}}{ a_{kk}\Delta t \left(G_i^{(k)} + 2 H_i^{(k)} + F_i^{(k)}\right) + \kappa},\cr
		&f_i^{(k)} =  H_i^{(k)} + F_i^{(k)}-  h_i^{(k)},\qquad 
		g_i^{(k)}=  H_i^{(k)} + G_i^{(k)}  -  h_i^{(k)}.
		\end{split}
		\end{align}
	\end{enumerate}

	\begin{remark}
		In Algorithm \ref{algorithm b}, consider the case $s = 1$ for implicit Euler method. Under the assumption  $\Delta t = \Delta x$, the relaxation limit $\kappa \to 0$ in \eqref{hfg} gives 
		\begin{align}
		\begin{split}
		&h_i^{n+1} = \frac{( h_i^{n}+ f_{i-1}^n )( h_i^{n} + g_{i+1}^n)}{ g_{i+1}^n + 2 h_i^{n} + f_{i-1}^n },\cr
		&f_i^{n+1} =  h_i^{n} + f_i^{(1,0)} -  h_i^{n+1},\qquad
		g_i^{n+1}=  h_i^{n} + g_i^{(1,0)}  -  h_i^{n+1}.
		\end{split}
		\end{align}
		This limiting scheme coincides with the relaxation scheme in \cite{jin1995relaxation} applied to the Broadwell model. Also, using the relation \eqref{rhomz def}, we can rewrite it as follows:
		\begin{align*}
		\rho_i^{n+1} &= \rho_i^{n} - \frac{1}{2}\left(m_{i+1}^{n}-m_{i-1}^{n}\right)
		+ \frac{1}{2}\left(z_{i-1}^{n}-2 z_i^{n} + z_{i-1}^{n}\right),\cr
		m_i^{n+1} &= \frac{1}{2}\left(m_{i+1}^{n} + m_{i-1}^{n}\right) - \frac{1}{2}\left(z_{i+1}^{n} - z_{i-1}^{n}\right),\cr
		z_i^{n+1} &= \frac{(\rho_i^{n+1})^2+(m_i^{n+1})^2}{ 2\rho_i^{n+1}}.
		\end{align*}
		We note that the scheme projects numerical solutions to equilibrium after one time step. 
	\end{remark}

	\section{Numerical tests}\label{sec:Numerical test}
	Our main interest is to confirm the performance of the proposed reconstruction in one and two dimensions. For numerical experiments, we consider the reconstruction \eqref{scheme} and \eqref{scheme 2d} based on CWENO reconstructions. This section is divided into three parts: 1D Xin-Jin model \eqref{relaxation system}, 1D Broadwell model \eqref{Broadwell} and 2D Xin-Jin model \eqref{relaxation system}. For each system, we check the accuracy of the corresponding semi-Lagrangian schemes and consider the related shock problems which arise in the relaxation limit $\kappa \rightarrow 0$. 
	For numerical tests, we use the CFL number defined by CFL$=\frac{\Delta t}{\Delta x}$ using uniform grid points based on $\Delta x$ and $\Delta t$. For 2D, we use CFL$=\frac{\Delta t}{\Delta x}=\frac{\Delta t}{\Delta y}$.

	\subsection{1D case for Xin-Jin model}\label{1d xinjin test}
	Here tests are based on the numerical method in Algorithm \ref{algorithm}. Note that we adopt $F(u)=u^2/2$.
	\subsubsection{Accuracy test}\label{xinjin 1d accruacy}
	We take well-prepared initial data up to first order in $\kappa$ \cite{boscarino2009class}:
	\begin{align}\label{initial 1d xinjin accuracy}
	u_0(x)=0.7+0.2 \sin(\pi x), \quad v_0(x)= \frac{u_0^2(x)}{2} + \kappa \left( u_0^2(x) -1\right) \partial_x u_0(x),
	\end{align}
	where periodic boundary conditions are imposed on $ x \in [-1, 1]$. In the limit $\kappa \rightarrow 0$ with $F(u)=u^2/2$, system \eqref{relaxation system} becomes the Burgers equation where shock appears after the positive minimum time: $\displaystyle T_b:= \inf_{u_0'<0} \left\{-\frac{1}{u_0'(x)}\right\} $. In view of this, we take a final time as $T^f=1$ which is less than the breaking time $T_b= 5/\pi \approx 1.5915$. 
	In this test, we use several values of CFL=$\Delta t/\Delta x <1$.
	We remark that the subcharacteristic condition $\max_u|F'(u)|<1$ is always satisfied. In Fig. \ref{fig 2}, a DIRK2 based method attains its desired accuracy between 2 and 3. In the case of DIRK43 method, it attains its desired accuracy between 3 and 5 except for some order reductions which appear in the intermediate regimes. We remark that the spatial errors are dominant for small CFL numbers, which make it easy to observe the order of spatial reconstructions. 
	\begin{figure}[htbp]
		\centering
		\begin{subfigure}[b]{0.45\linewidth}
			\includegraphics[width=1\linewidth]{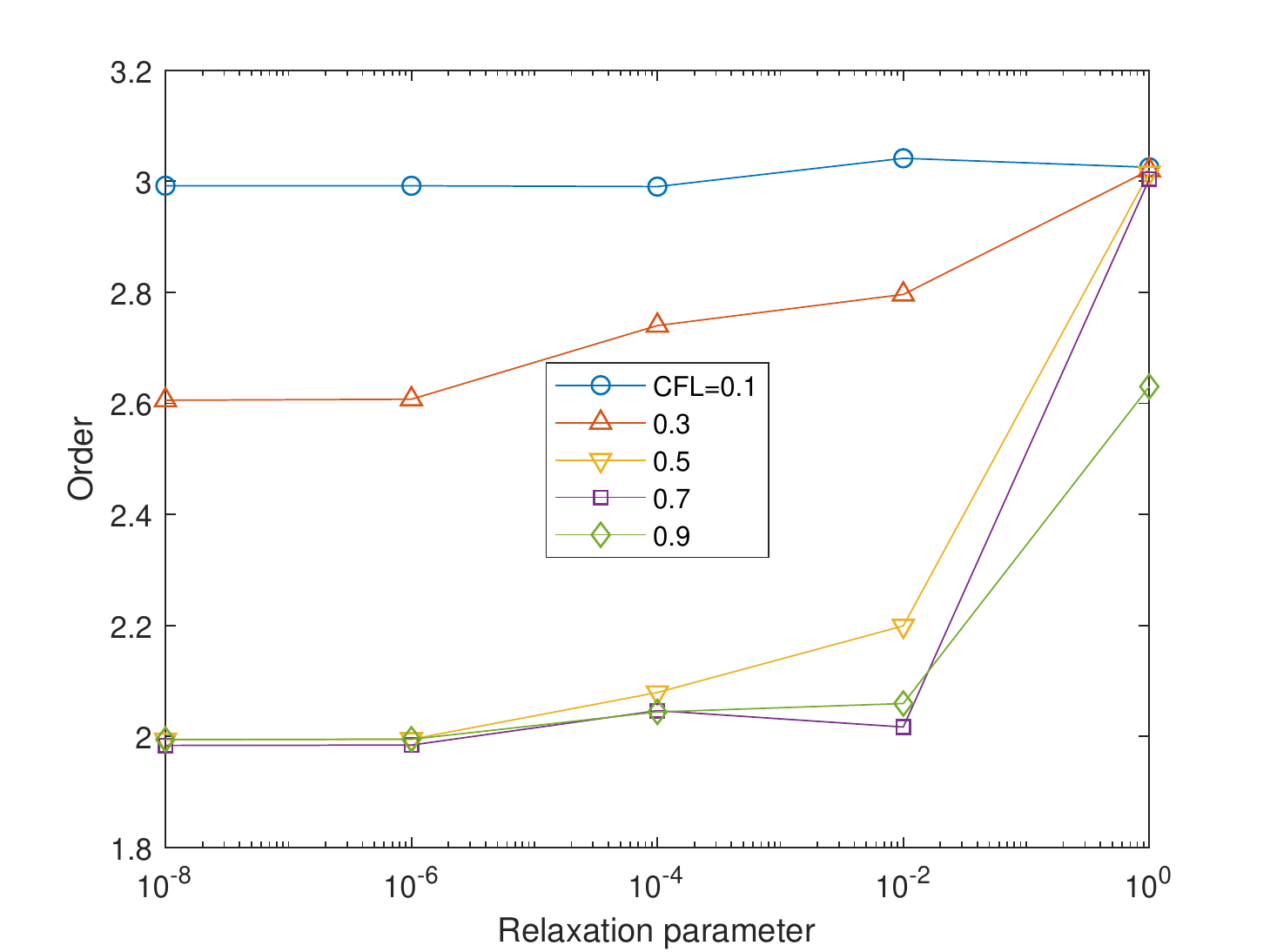}
			\subcaption{DIRK2 - Q-CWENO23}
		\end{subfigure}	
		\begin{subfigure}[b]{0.45\linewidth}
			\includegraphics[width=1\linewidth]{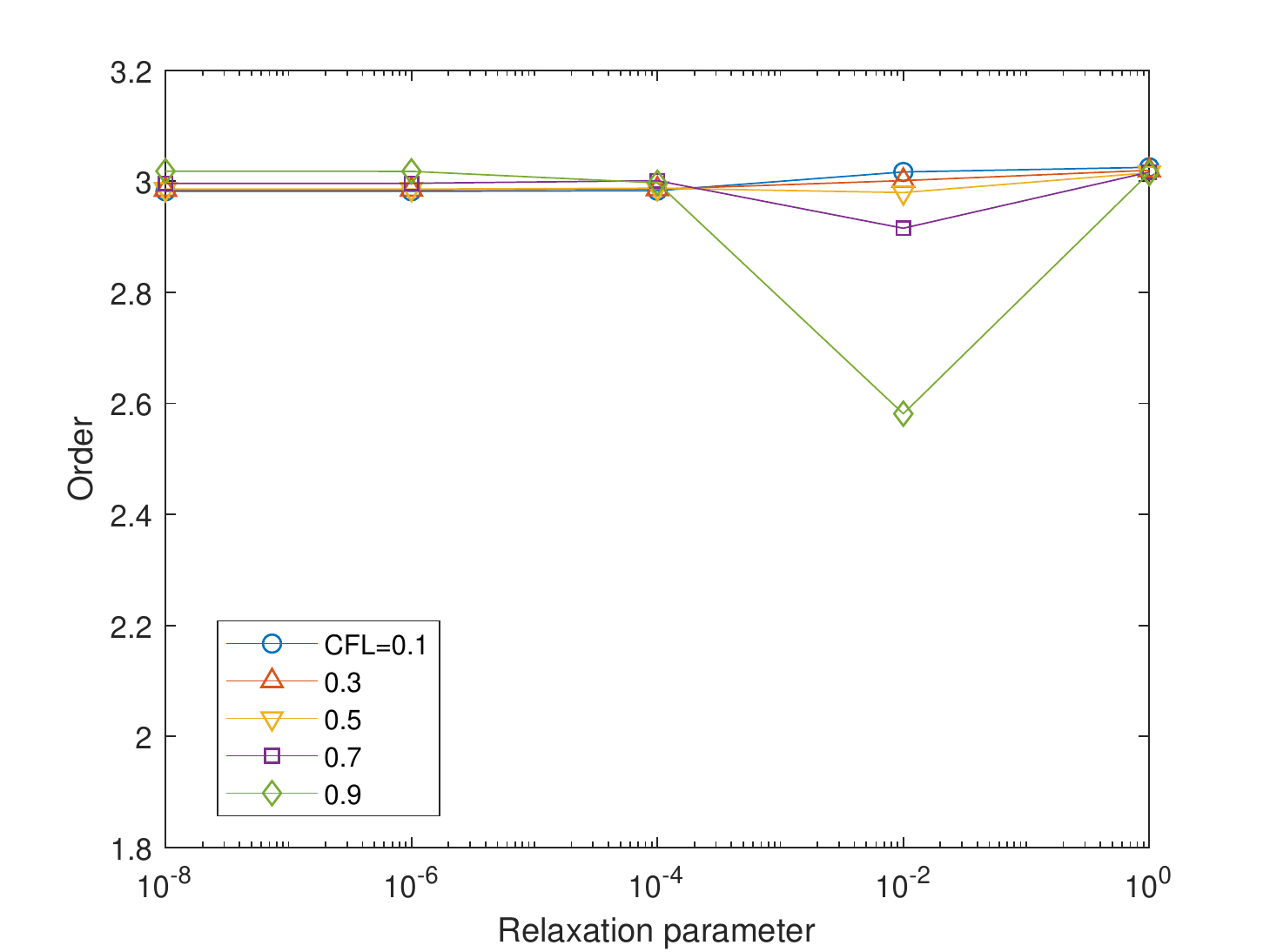}
			\subcaption{DIRK43 - Q-CWENO23}
		\end{subfigure}	
		
		\begin{subfigure}[b]{0.45\linewidth}
			\includegraphics[width=1\linewidth]{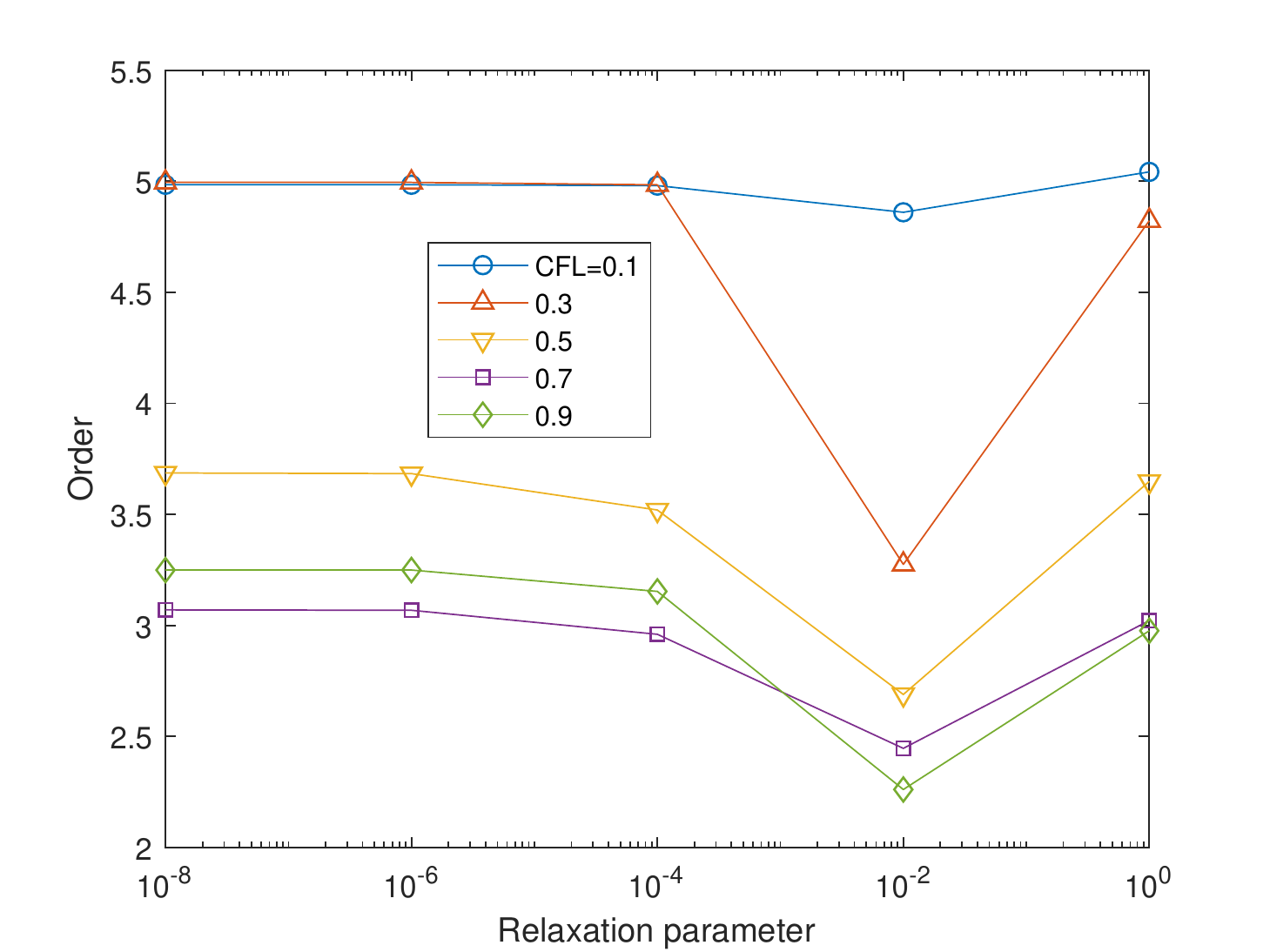}	
			\subcaption{DIRK43 - Q-CWENO35}
		\end{subfigure}	
		\caption{Accuracy tests for 1D Xin-Jin model. Initial data is associated to \eqref{initial 1d xinjin accuracy}. $x$-axis is for the relaxation parameter $\kappa$ and $y$-axis is for order of accuracy based on $N_x=160,320,640$.}\label{fig 2}  	
	\end{figure}

	\subsubsection{Shock tests}\label{shock 1d xinjin data}
	To confirm the conservation property of the proposed reconstruction in shock problems, we here compare numerical solutions obtained by conservative semi-Lagrangian schemes with non-conservative ones.
	
	{\bf $\bullet$ Smooth initial data.}
	We first take the an smooth initial data 
	\begin{align}\label{smooth initial xinjin 1d}
	u_0(x)=0.7+0.2\sin(\pi x), \quad v_0(x)= \frac{u_0^2(x)}{2},
	\end{align}
	where periodic boundary condition is imposed on $ x \in [-1, 1]$. We use grid points of $N_x=160$ up to final time $T^f=4$. Each time step is taken by $\Delta t= \text{CFL} \Delta x$. For each time $t=t^n$, we compute the conservation error using
	\[
	E_{con}^n:= \frac{\left|\sum_i u_i^{n} \Delta x - \sum_i u_i^{0} \Delta x\right|}{\sum_i u_i^{0} \Delta x}.
	\]
	In Fig. \ref{fig 3}, we compare the numerical solutions obtained from our reconstruction, linear interpolation (first order scheme), GWENO34 and GWENO46 \cite{CFR} with the reference solution in \cite{W}. We observe that the use of our reconstruction and linear interpolation leads to correct shock position. Also, the
	corresponding conservative errors show very small change as time flows. In contrast, conservation errors become bigger when we adopt GWENO34 and GWENO46 reconstructions after time $t=1$, which give wrong shock positions. (See Fig. \ref{fig 3}) \\
	
	\begin{figure}[htbp]
		\centering
		\begin{subfigure}[b]{0.45\linewidth}
			\includegraphics[width=1\linewidth]{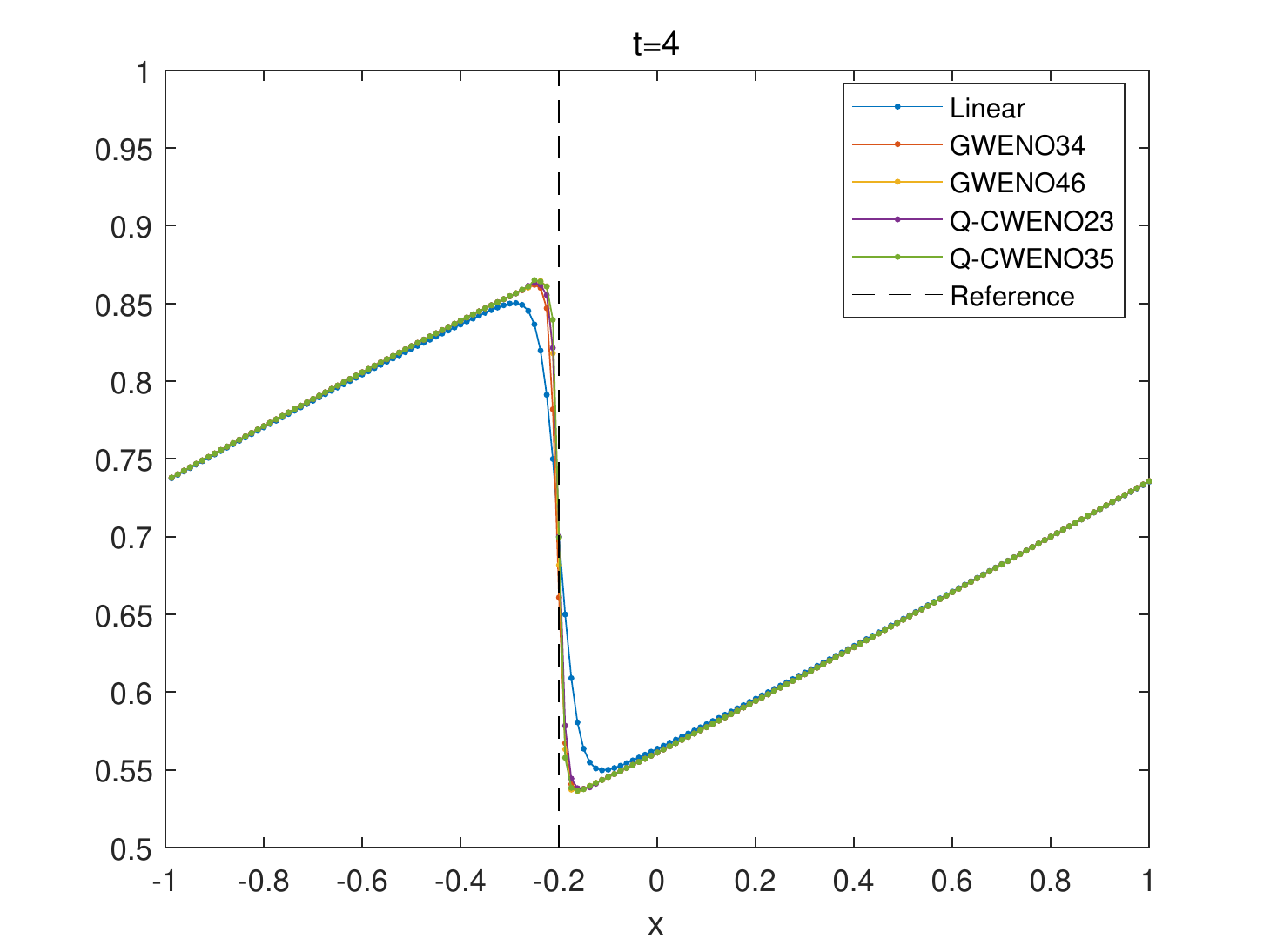}
			\subcaption{Comparison of numerical solutions w.r.t. reconstruction at $x \in[-1,1]$}
		\end{subfigure}	
		\begin{subfigure}[b]{0.45\linewidth}
			\includegraphics[width=1\linewidth]{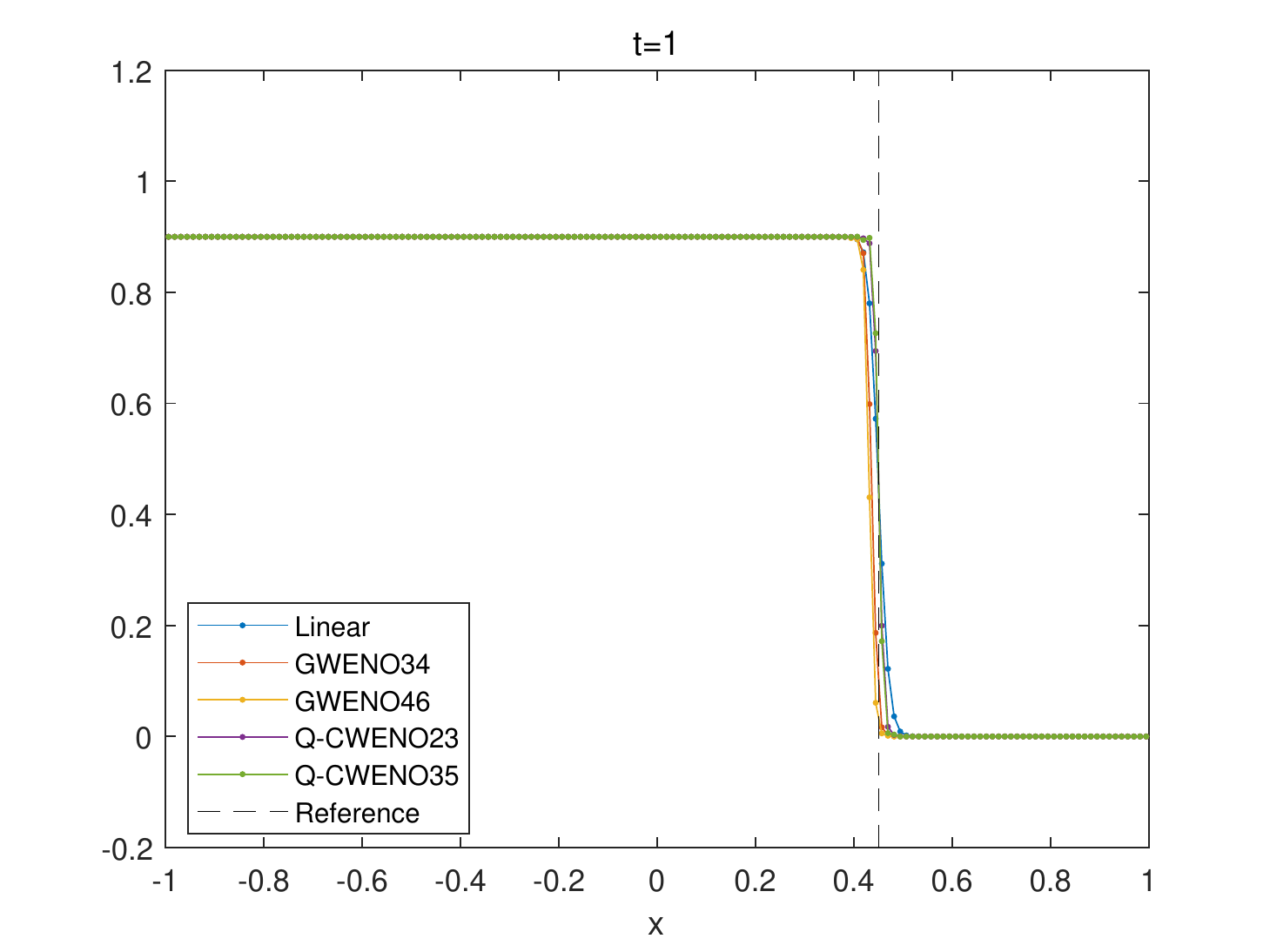}
			\subcaption{Comparison of numerical solutions w.r.t. reconstruction at $x \in[-1,1]$}
		\end{subfigure}		
		\begin{subfigure}[b]{0.45\linewidth}
			\includegraphics[width=1\linewidth]{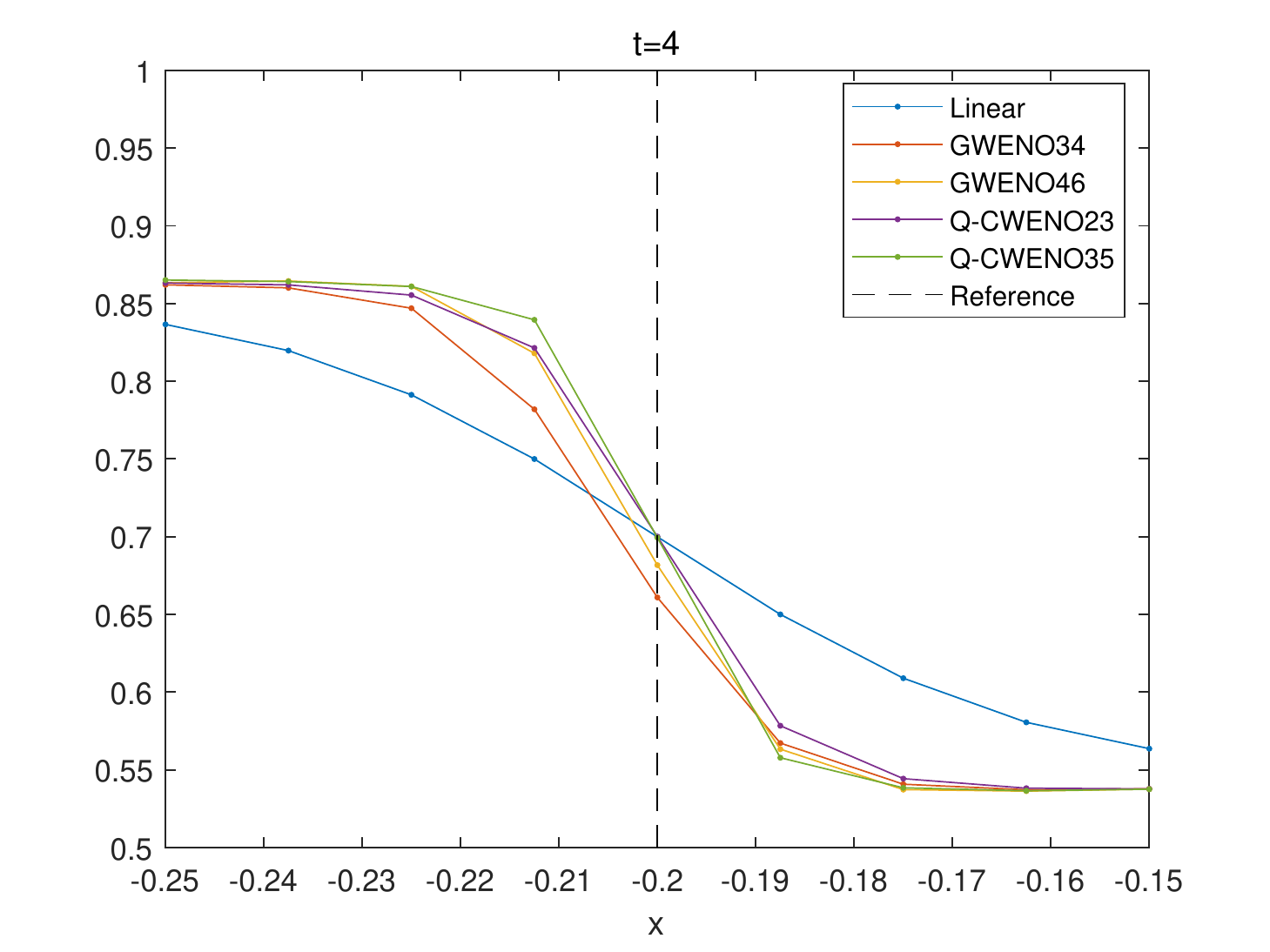}
			\subcaption{Numerical solutions w.r.t. reconstruction at $x \in[-0.25,-0.15]$}
		\end{subfigure}	
		\begin{subfigure}[b]{0.45\linewidth}
			\includegraphics[width=1\linewidth]{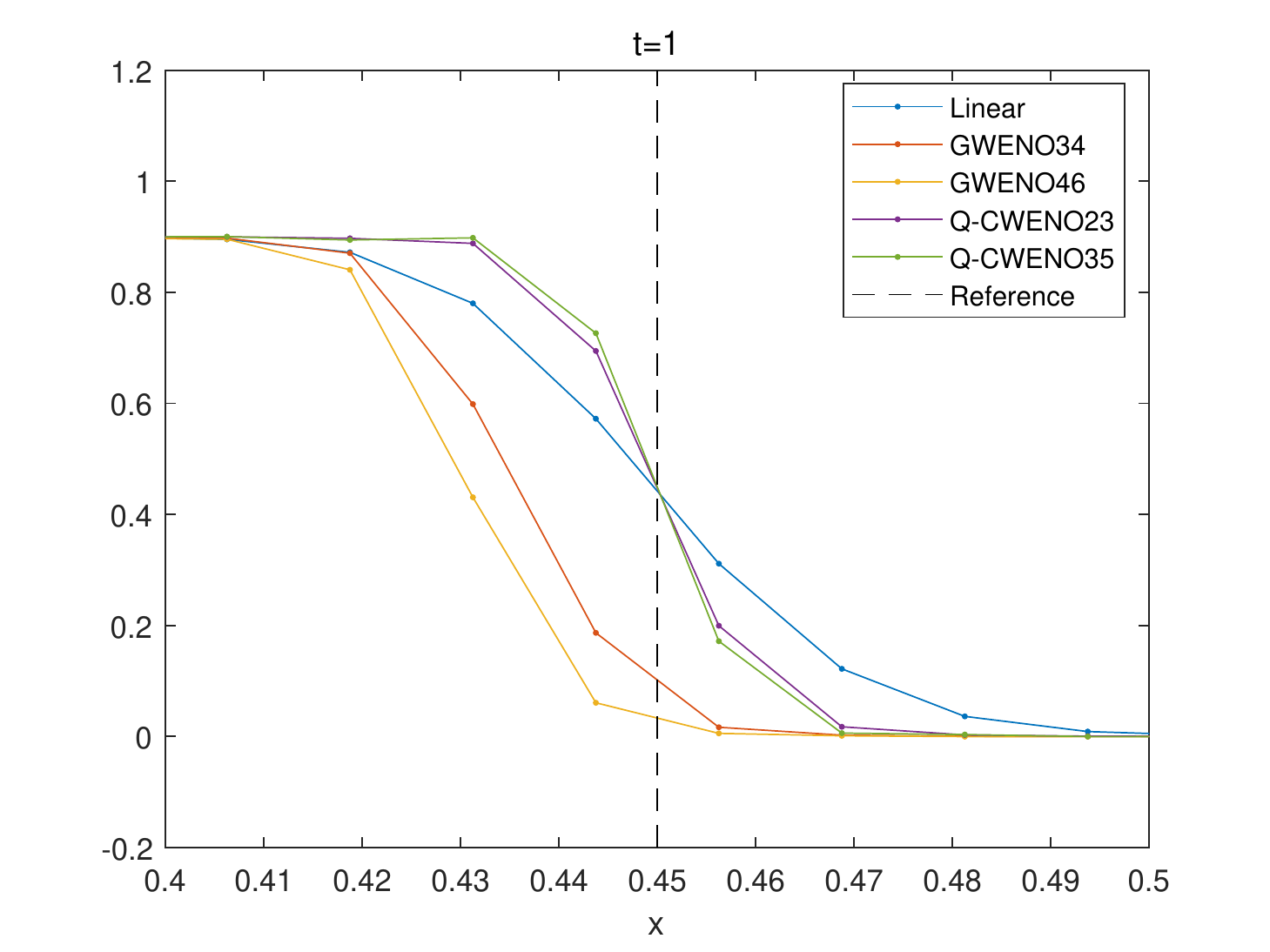}
			\subcaption{Numerical solutions w.r.t. reconstruction at $x \in[0.4,0.6]$}
		\end{subfigure}	
		\begin{subfigure}[b]{0.45\linewidth}
			\includegraphics[width=1\linewidth]{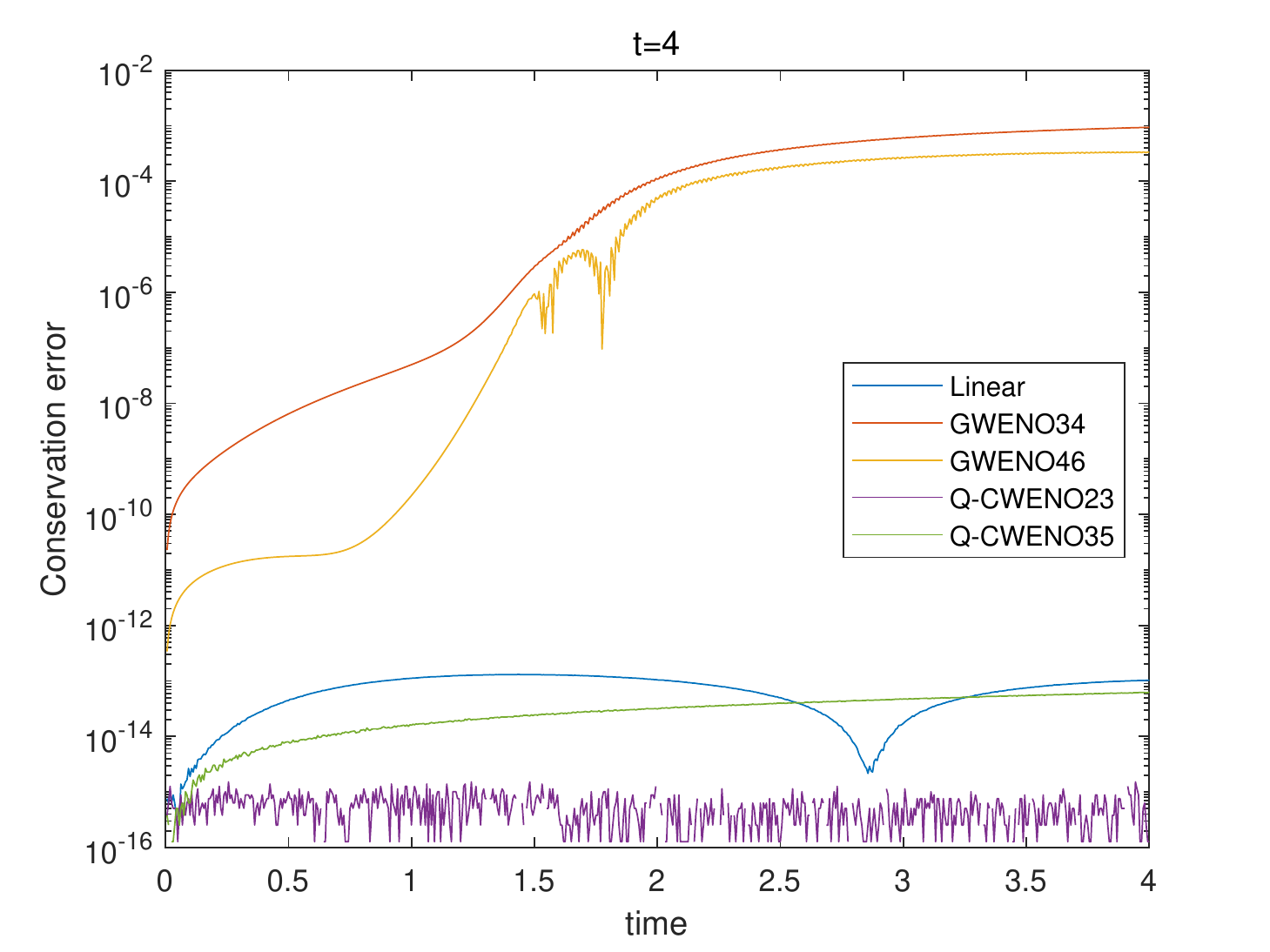}	
			\subcaption{Conservation errors w.r.t. time}
		\end{subfigure}	
		\begin{subfigure}[b]{0.45\linewidth}
			\includegraphics[width=1\linewidth]{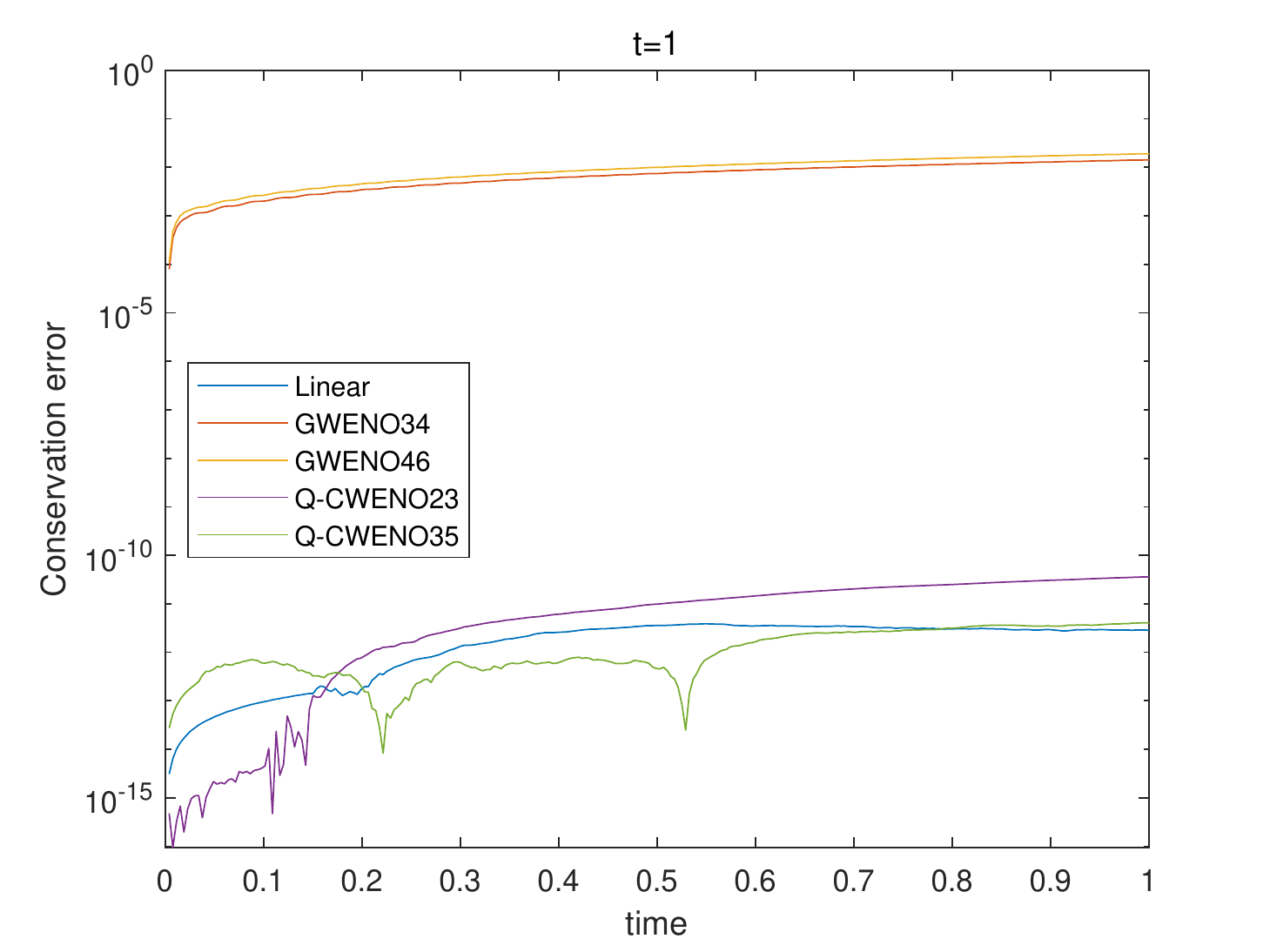}	
			\subcaption{Conservation errors w.r.t. time}
		\end{subfigure}	
		
		\caption{Shock tests for 1D Xin-Jin model. Left: initial data \eqref{smooth initial xinjin 1d} with CFL$=0.5$ Right: initial data \eqref{xin jin shock step initial} with CFL$=0.3$. The results are obtained by DIRK43 based SL methods for $\kappa=10^{-8}$ with various reconstructions.}\label{fig 3}  	
	\end{figure}
	{\bf $\bullet$ Discontinuous initial data.}
	In this test, we again solve the system \eqref{relaxation system} with initial data
	\begin{align}\label{xin jin shock step initial}
	u_0(x)=\begin{cases}
	0.9, \quad x \leq 0\\
	0, \quad x > 0
	\end{cases}, \quad v_0(x)= \frac{u_0^2(x)}{2}
	\end{align}
	under freeflow boundary condition on $ x \in [-1, 1]$ with grid points $N_x=160$ up to final time $T^f=1$. In this test, we compute the conservation error using
	\[
	E_{con}^n:= \frac{\sum_i u_i^{n} \Delta x - (\sum_i u_i^{0}\Delta x + s t^n)}{\sum_i u_i^{0} \Delta x},
	\]
	where $s$ is the speed of shock, which is given by $s=0.45$. 
	
	We show our reconstruction can be more effective in capturing shock position. In Fig. \ref{fig 3}, we again compare the numerical solutions for different reconstructions. As in the previous shock test, our reconstruction and linear interpolation show better performance in capturing shock position compared to GWENO34 and GWENO46 reconstructions.
	

	\begin{remark}\label{rmk 5.1}
		\begin{figure}[htbp]
			\centering		
			\begin{subfigure}[b]{0.45\linewidth}
				\includegraphics[width=1\linewidth]{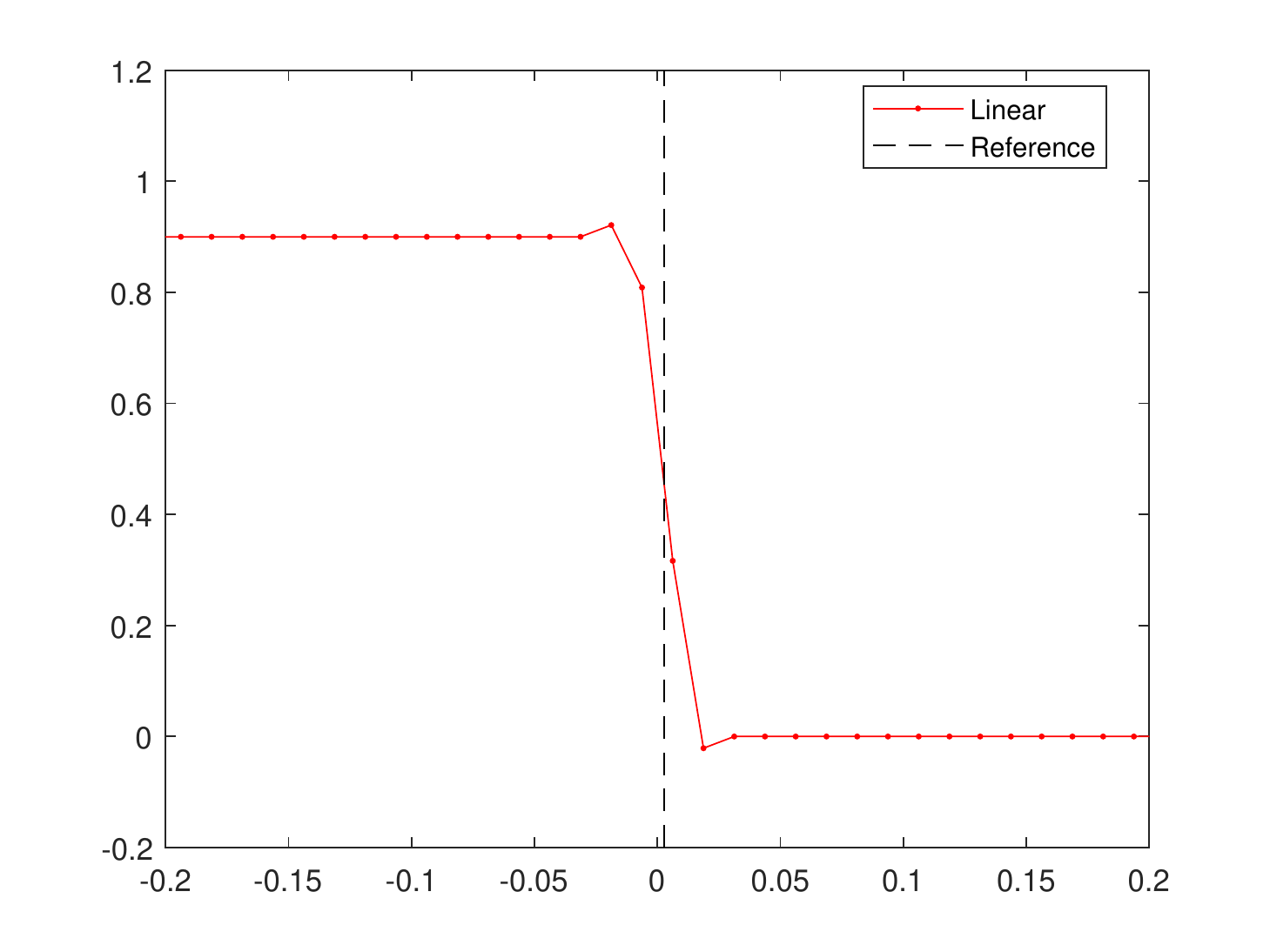}
			\end{subfigure}	
			\caption{Shock test associated to Remark \ref{rmk 5.1}. For $\kappa=10^{-8}$, DIRK2 based SL scheme is implemented with linear interpolation. Note that oscillation appears at $T^f=\Delta t=0.00625$. 
			}\label{fig oscillation}  	
		\end{figure}
		In the section \ref{1d xinjin test}, we confirmed that high-order DIRK based SL schemes of Xin-Jin model works for all ranges of relaxation parameters. We also observed that, in the limit $\kappa \rightarrow 0$, oscillations appear near discontinuities for all high-order RK and BDF based SL schemes.
		To understand this phenomena, as a simple case, consider $F(u)=bu$ for $|b|<1$. We will show that oscillation appears even after one step $t=t^1$ for arbitrary second order DIRK based SL schemes with linear interpolation (see Fig \ref{fig oscillation}). We use the Butcher's table given by 
		\begin{align*}
		\begin{array}{c|c c}
		\alpha_1 & \alpha_1 & 0 \\
		1 & 1-\alpha_2 & \alpha_2 \\
		\hline
		& 1-\alpha_2 & \alpha_2
		\end{array}, \quad \alpha_2= \frac{\frac{1}{2}-\alpha_1}{1-\alpha_1}.
		\end{align*}	
		Then, with the initial conditions \eqref{xin jin shock step initial}, the following calculation verifies our remark.
		$\bullet$ Assume CFL$=\frac{\Delta t}{\Delta x}\leq 1$, and $(u_{i-2}^0,u_{i-1}^0,u_i^0,u_{i+1}^0,u_{i+2}^0)=(0.9,0.9,0.9,0.9,0)$. Then, in the limit $\kappa \rightarrow 0$, we have 
		\begin{align*}
		u_i^1&= \left(1-\frac{\Delta t}{\Delta x}\right)u_{i}^0 + \frac{\Delta t}{2\Delta x}(u_{i-1}^0+u_{i+1}^0)- \frac{b\Delta t}{2\Delta x}(u_{i-1}^0-u_{i+1}^0)+\frac{(b^2-1)(\Delta t)^2}{8(\Delta x)^2} (u_{i-2}^0 -2u_{i}^0 + u_{i+2}^0)\cr
		&= 0.9 \left(1+ \frac{(1-b^2)(\Delta t)^2}{8(\Delta x)^2}\right)>0.9,
		\end{align*}
		for any $\alpha_1 \ne 0,1$.
	\end{remark}

	\subsection{1D Broadwell model}
	Now, we move on to the semi-Lagrangian schemes for 1D Broadwell model \eqref{Broadwell}.
	\subsubsection{Accuracy test}
	To check the accuracy of the proposed schemes, we consider well-prepared data \cite{pareschi2005implicit}:
	\begin{align}\label{initial acc 1d broadwell}
	\begin{split}
	\rho_0(x)&= 1+ a_\rho \sin{\frac{2\pi}{L} x}, \quad u_0(x)= \frac{1}{2}+ a_u \sin{\frac{2\pi}{L} x}, \cr
	z_0(x)&=z_E(\rho_0(x),u_0(x)) + \kappa z_1(\rho_0(x),u_0(x))
	\end{split}
	\end{align}
	where $a_\rho=0.3$, $a_u=0.1$, $L=20$, $T^f=30$, and 
	\begin{align*}
	z_E(\rho_0,m_0)&= \frac{1}{2\rho_0}\left(\rho_0^2 + m_0^2\right), \quad 
	z_1(\rho_0,m_0)=-\frac{H(\rho_0,m_0)}{\rho_0},\cr
	H(\rho_0,m_0)&= \left(1-\partial_\rho z_E +(\partial_m z_E)^2\right)\partial_x m_0 + (\partial_\rho z_E \partial_m z_E)\partial_x\rho_0.
	\end{align*}
	The periodic condition is imposed on $[-20,20]$ upto final time $T^f=30$.
	We take different CFL numbers less than $1$. The order of convergence is based on the grid points $N_x=160,320,640$. Here the desired accuracy for DIRK2 is between 2 and 3, while for DIRK43, it is between 3 and 5.
	
	In Fig. \ref{fig Broadwell accuracy2}, one can see that the DIRK2 based method attains the desired accuracy for all ranges of $\kappa$. On the other hand, in the limit $\kappa \rightarrow 0$, the DIRK43 based method shows order reduction, which could be prevented by adopting the BDF3 based method. For small CFL numbers, space errors dominate so the order of accuracy comes from spatial reconstruction, while for large CFL time discretization errors dominate so the order of accuracy comes from time integration.  
	
	\begin{figure}[!]
		\centering
		\begin{subfigure}[b]{0.45\linewidth}
			\includegraphics[width=1\linewidth]{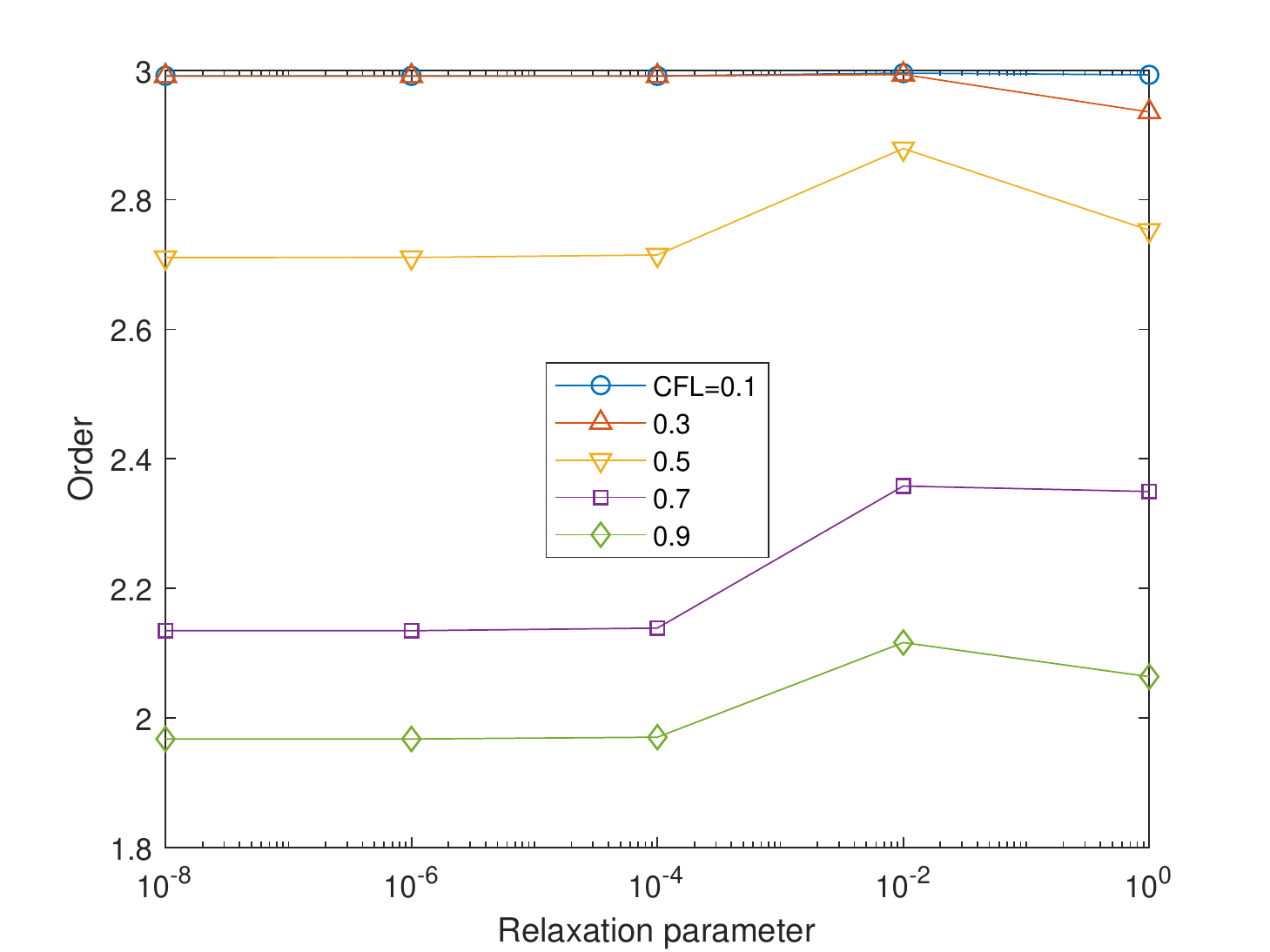}
			\subcaption{DIRK2, Q-CWENO23}
		\end{subfigure}		
		\begin{subfigure}[b]{0.45\linewidth}
			\includegraphics[width=1\linewidth]{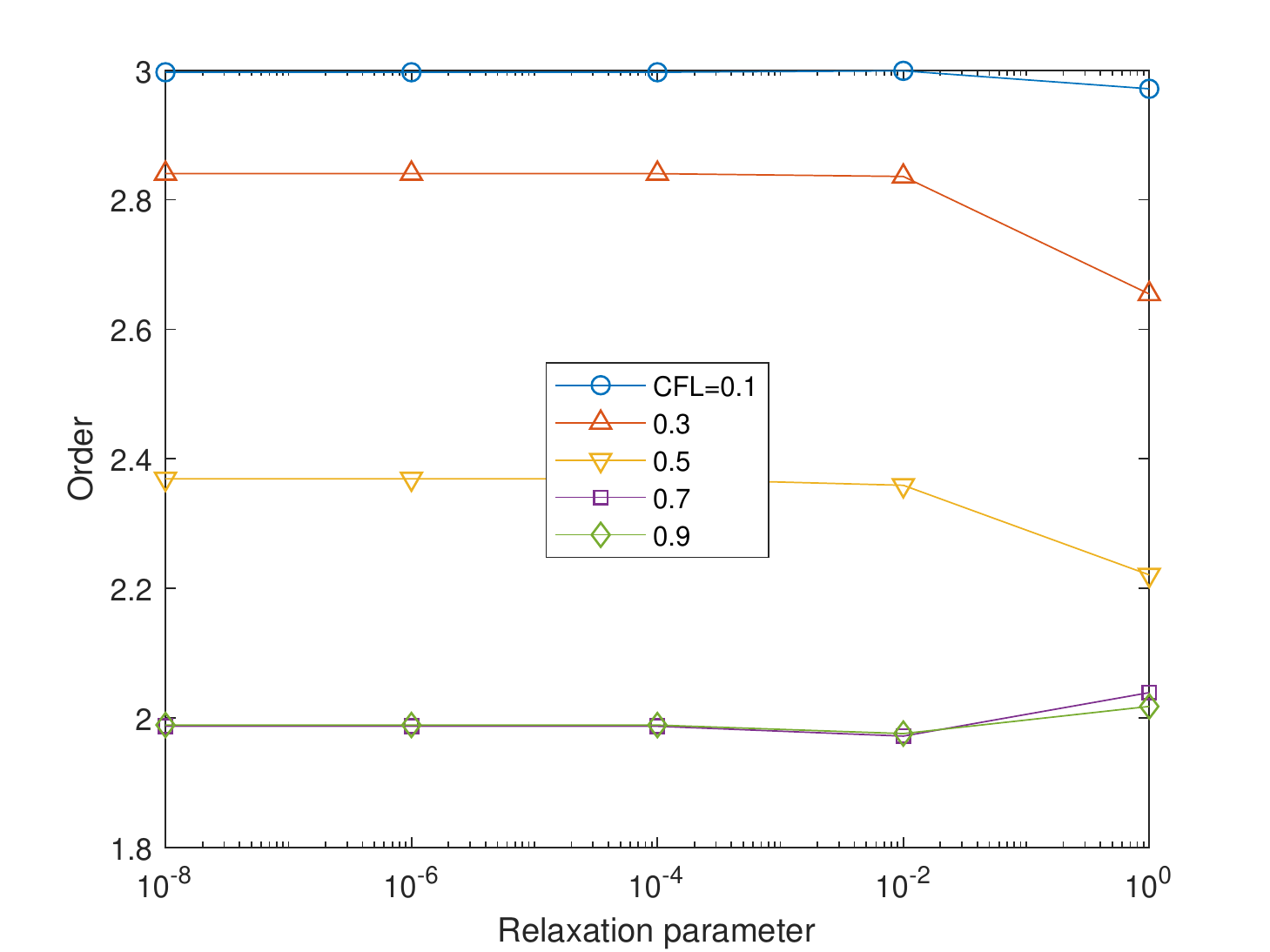}
			\subcaption{BDF2, Q-CWENO23}
		\end{subfigure}		
		\begin{subfigure}[b]{0.45\linewidth}
			\includegraphics[width=1\linewidth]{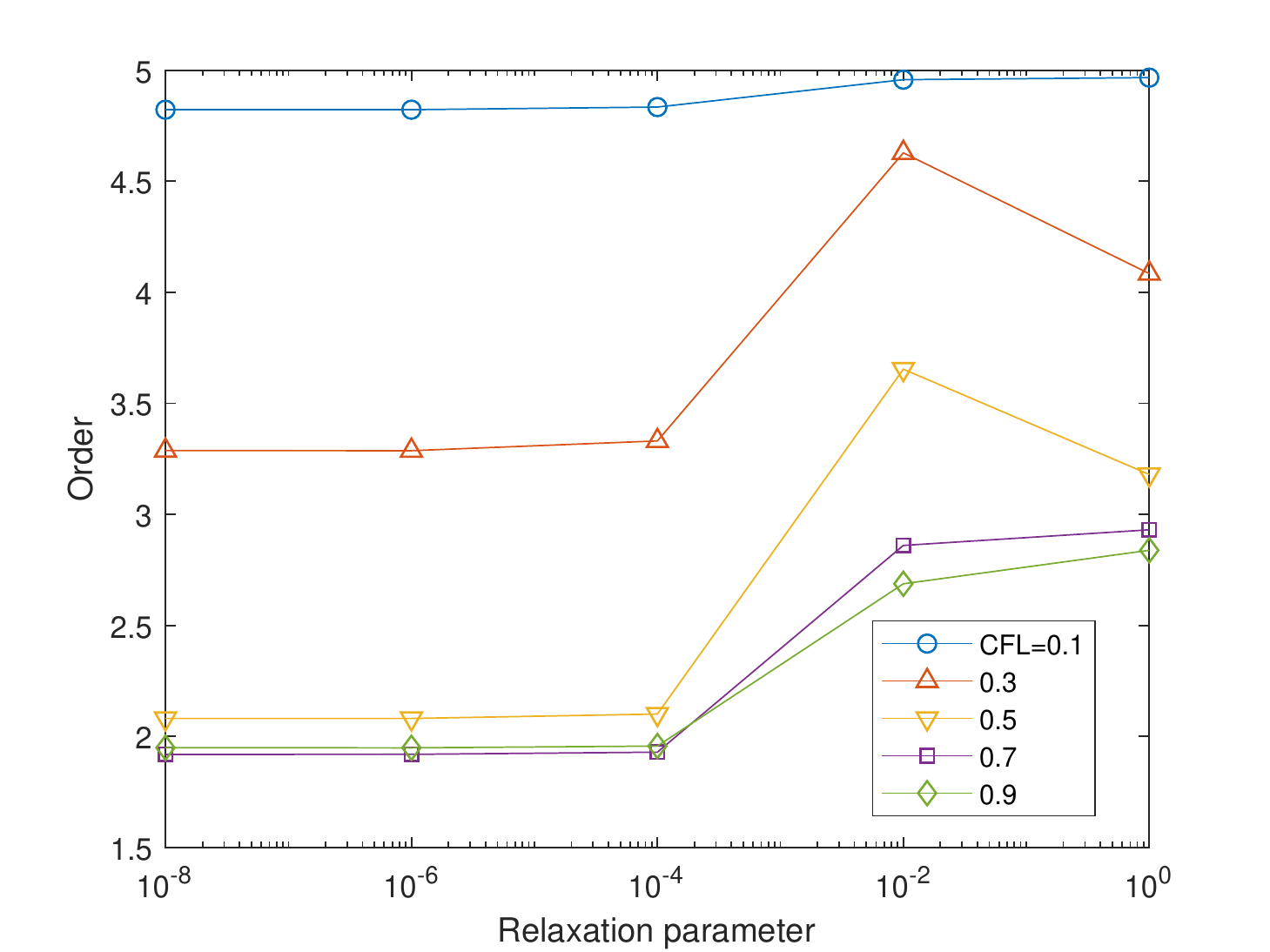}
			\subcaption{DIRK43, Q-CWENO35}
		\end{subfigure}		
		\begin{subfigure}[b]{0.45\linewidth}
			\includegraphics[width=1\linewidth]{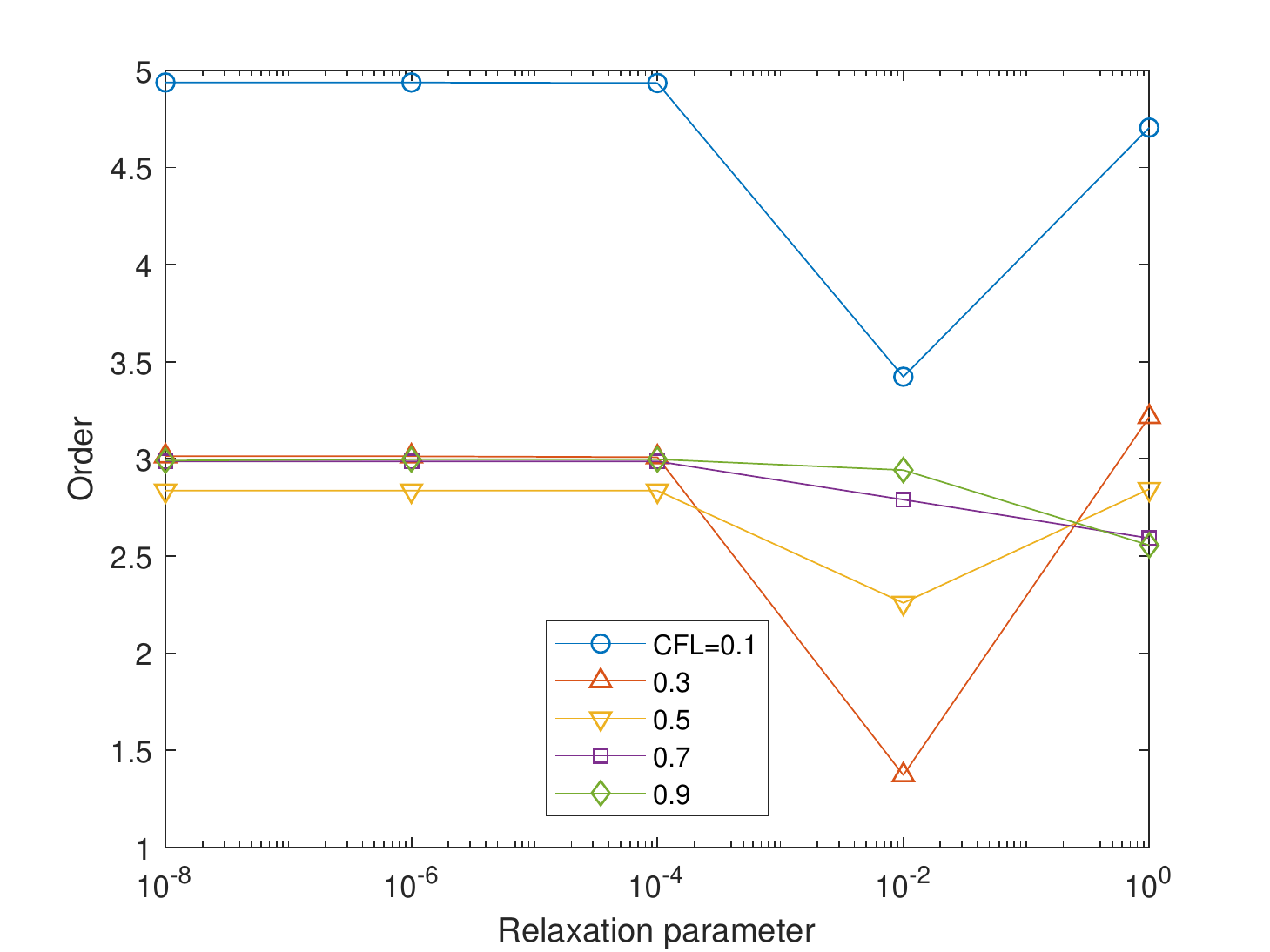}
			\subcaption{BDF3, Q-CWENO35}
		\end{subfigure}				
		\begin{subfigure}[b]{0.45\linewidth}
			\includegraphics[width=1\linewidth]{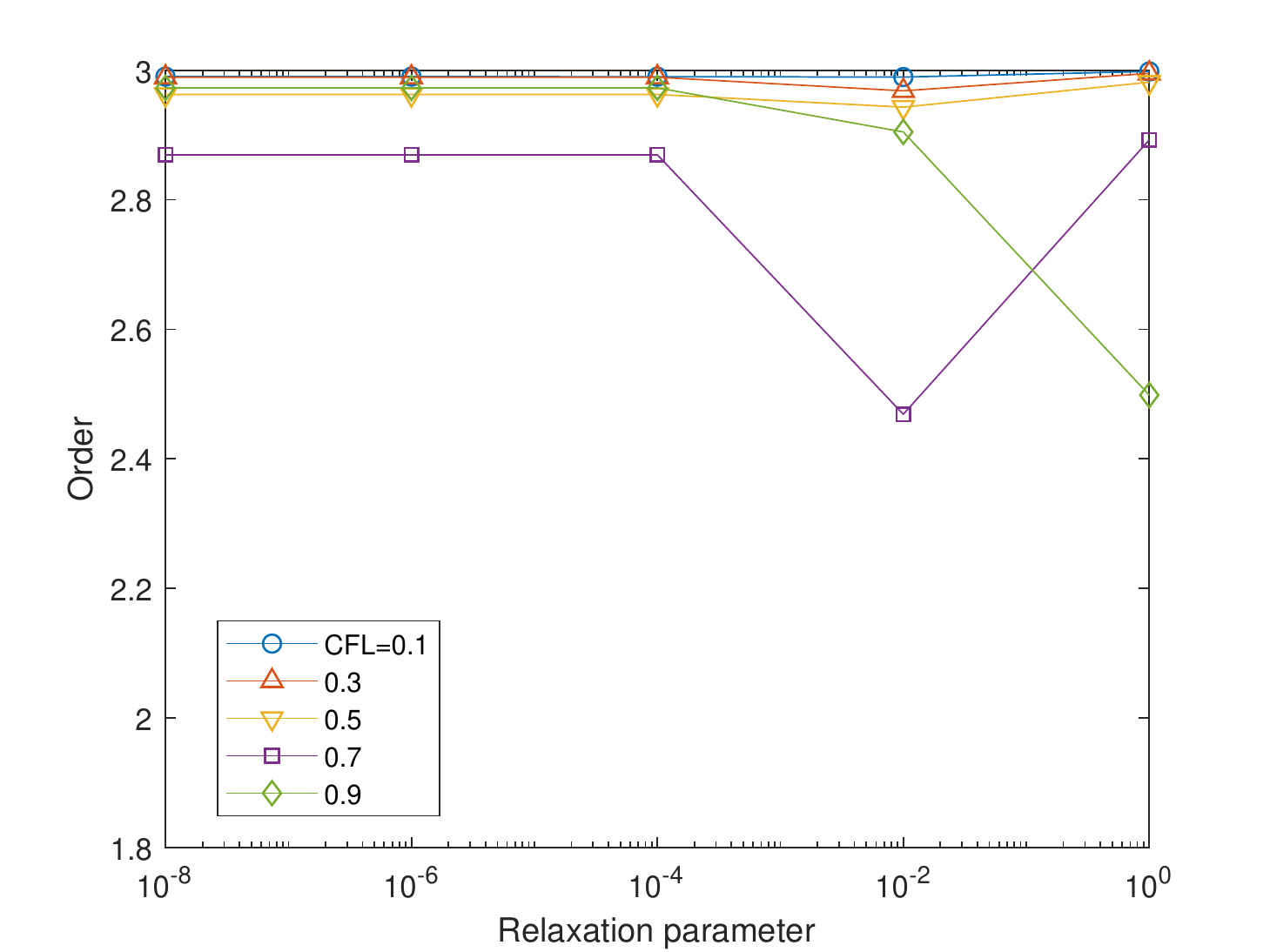}
			\subcaption{BDF3, Q-CWENO23}
		\end{subfigure}				
		
		\caption{Accuracy tests for 1D Broadwell model. Initial data is associated to \eqref{initial acc 1d broadwell}. $x$-axis is for the relaxation parameter $\kappa$ and $y$-axis is for order of accuracy based on $N_x=160,320,640$.}\label{fig Broadwell accuracy2}  	
	\end{figure}
	
	%
	
	\subsection{Shock tests}
	We consider the following two cases in \cite{CJR}:
	\begin{align}\label{shock cases}
	\begin{split}
	\text{Case 1.}\, (\rho,\, m,\, z)=
	&\begin{cases}
	(2,1,1) \quad x<0.2 \\ 
	(1,0.13962,1) \quad x>0.2 \\ 
	\end{cases}, \, x\in [-1,1],\, T^f=0.25,\, \kappa= 1,\\
	\text{Case 2.}\, (\rho,\, m,\, z)=		
	&\begin{cases}
	(1,0,1) \quad x<0.5 \\ 
	(0.2,0,1) \quad x>0.5  
	\\ 
	\end{cases}, \, x\in [0,1],\, T^f=0.25,\, \kappa= 10^{-8}.
	\end{split}
	\end{align}
	For each case, we take $N_x=200$. In Fig. \ref{fig Broadwell case1 shock} we observe that the proposed schemes allows large CFL$>1$ with the choice of $\kappa=1$. In case of $\kappa=10^{-8}$, some oscillations appear near the discontinuity for CFL$> 0.8$. For CFL$\leq 0.8$, we obtain solutions which reproduce the numerical results in \cite{CJR}.
	\begin{figure}[htbp]
		\centering
		\begin{subfigure}[b]{0.45\linewidth}
			\includegraphics[width=1\linewidth]{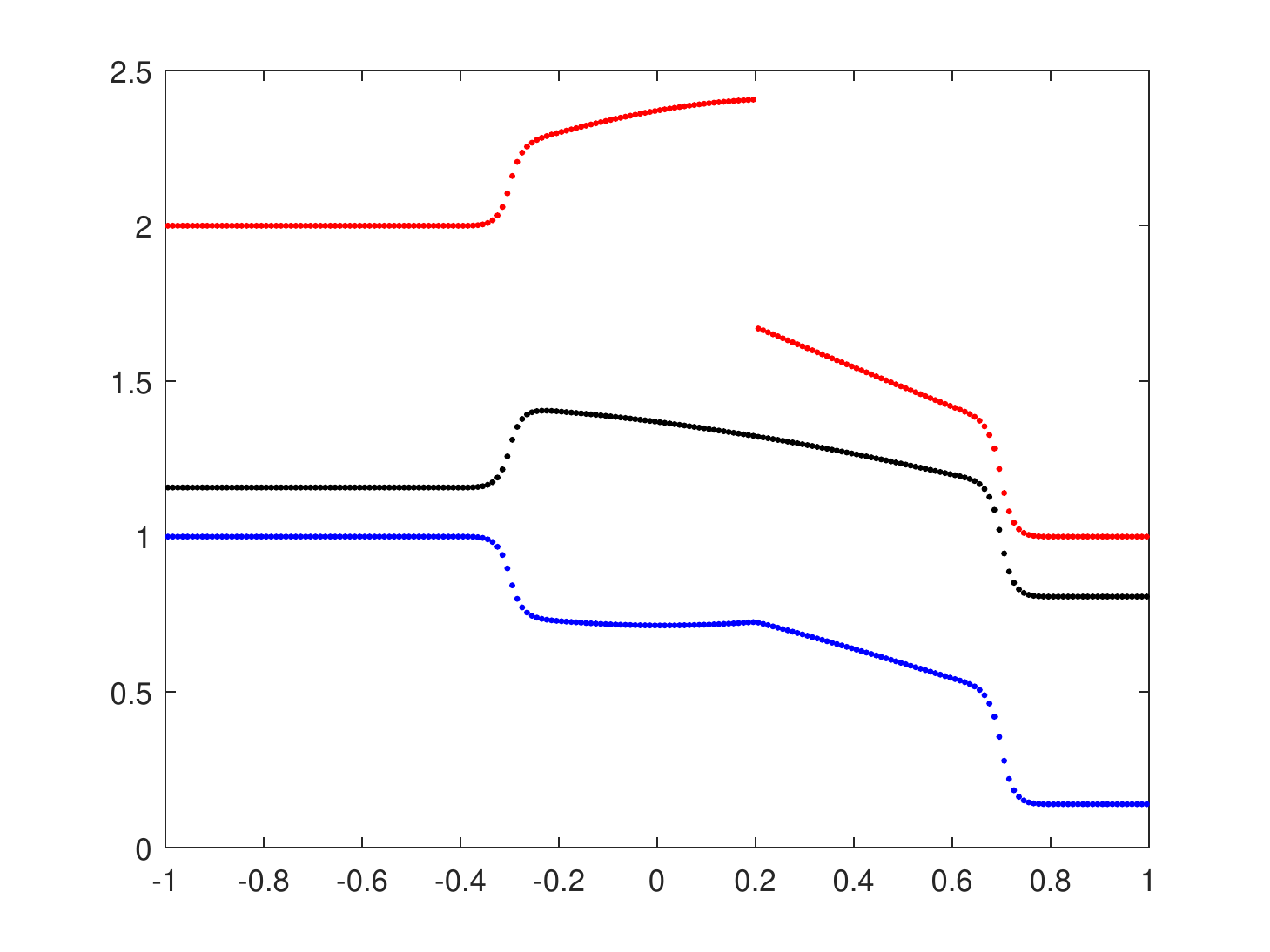}
			\subcaption{DIRK2, Q-CWENO23, CFL$=0.5$}
		\end{subfigure}	
		\begin{subfigure}[b]{0.45\linewidth}
			\includegraphics[width=1\linewidth]{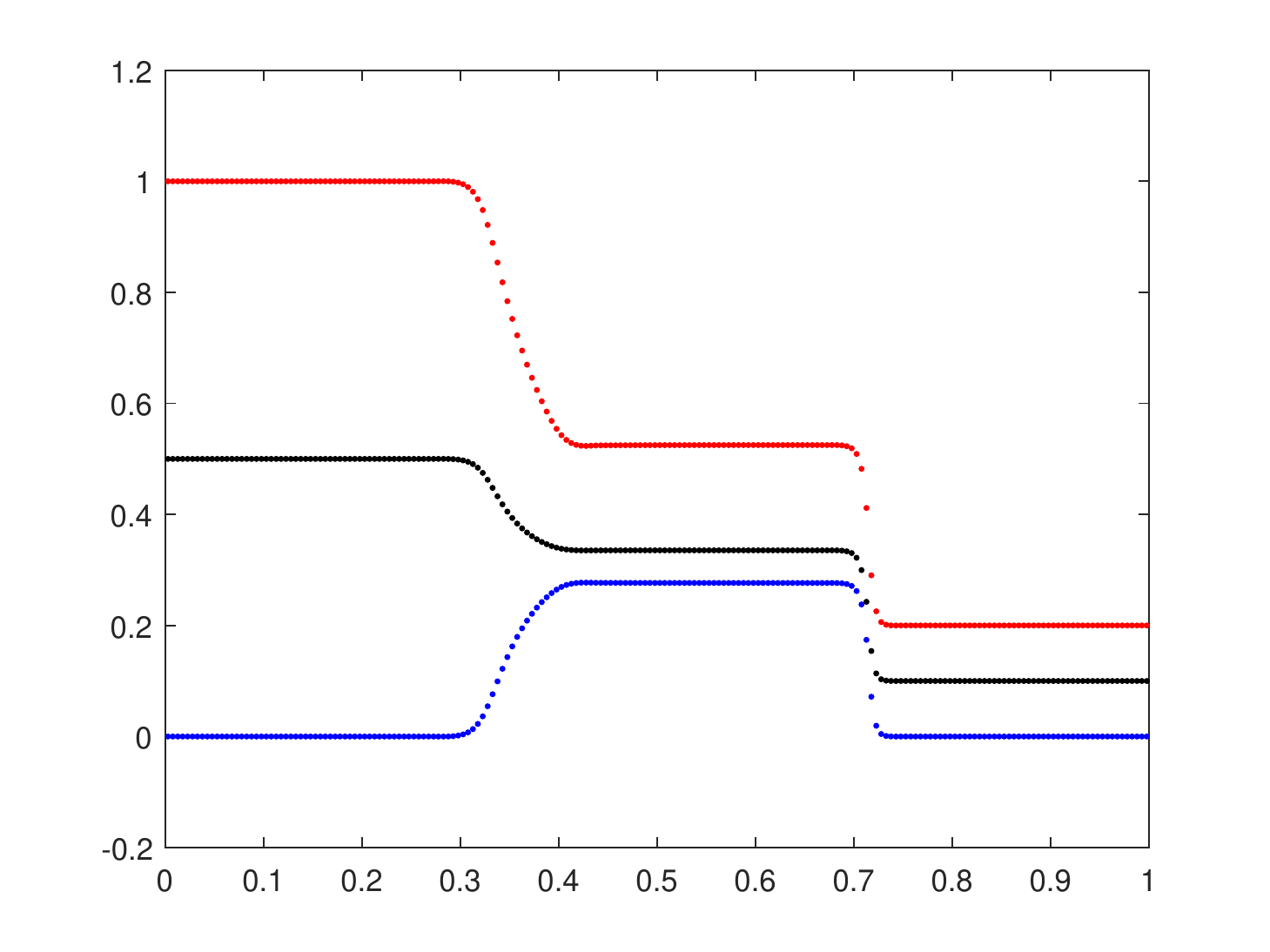}
			\subcaption{DIRK2, Q-CWENO23, CFL$=0.5$}
		\end{subfigure}	
		\begin{subfigure}[b]{0.45\linewidth}
			\includegraphics[width=1\linewidth]{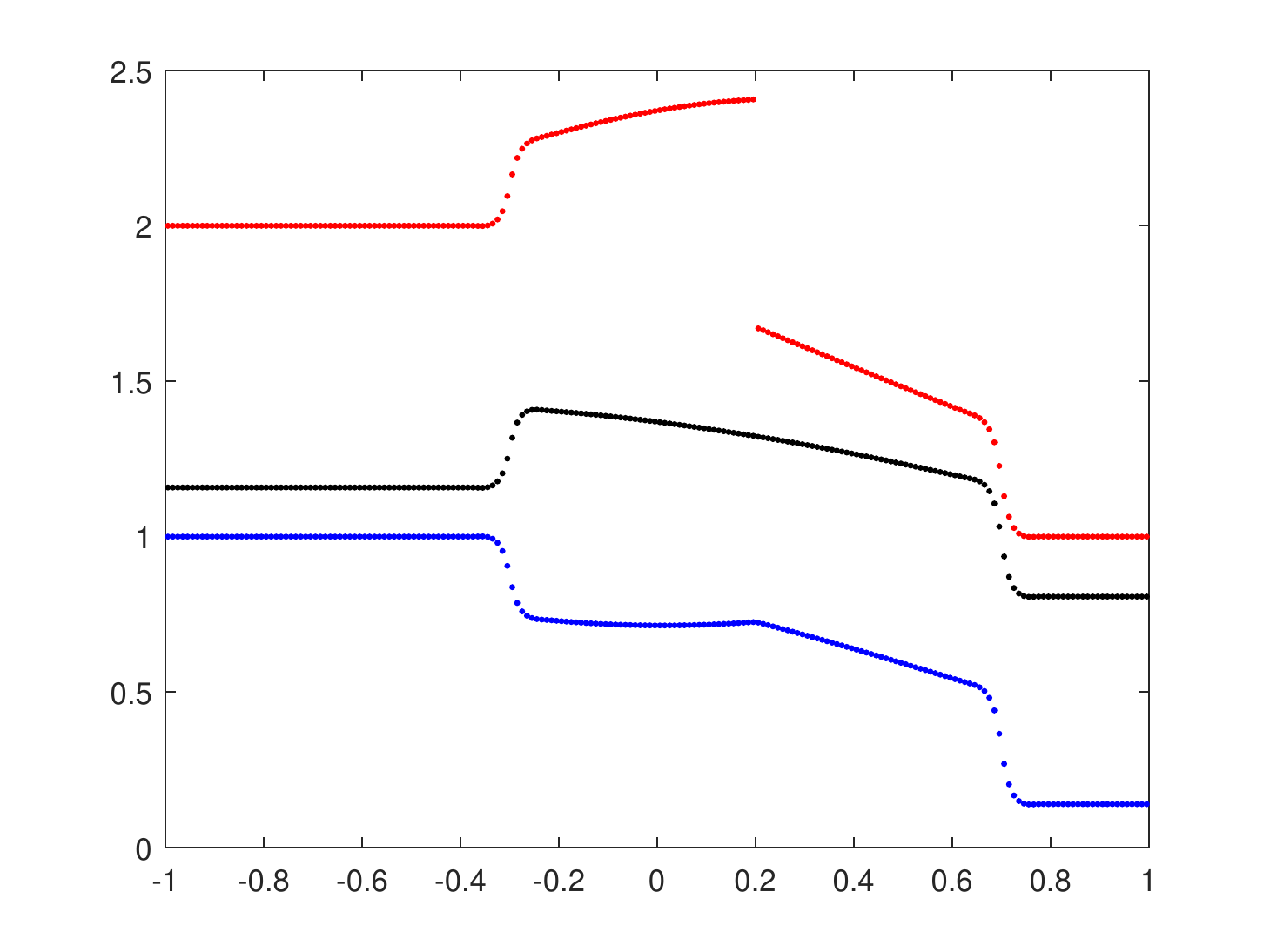}
			\subcaption{DIRK2, Q-CWENO23, CFL$=0.8$}
		\end{subfigure}	
		\begin{subfigure}[b]{0.45\linewidth}
			\includegraphics[width=1\linewidth]{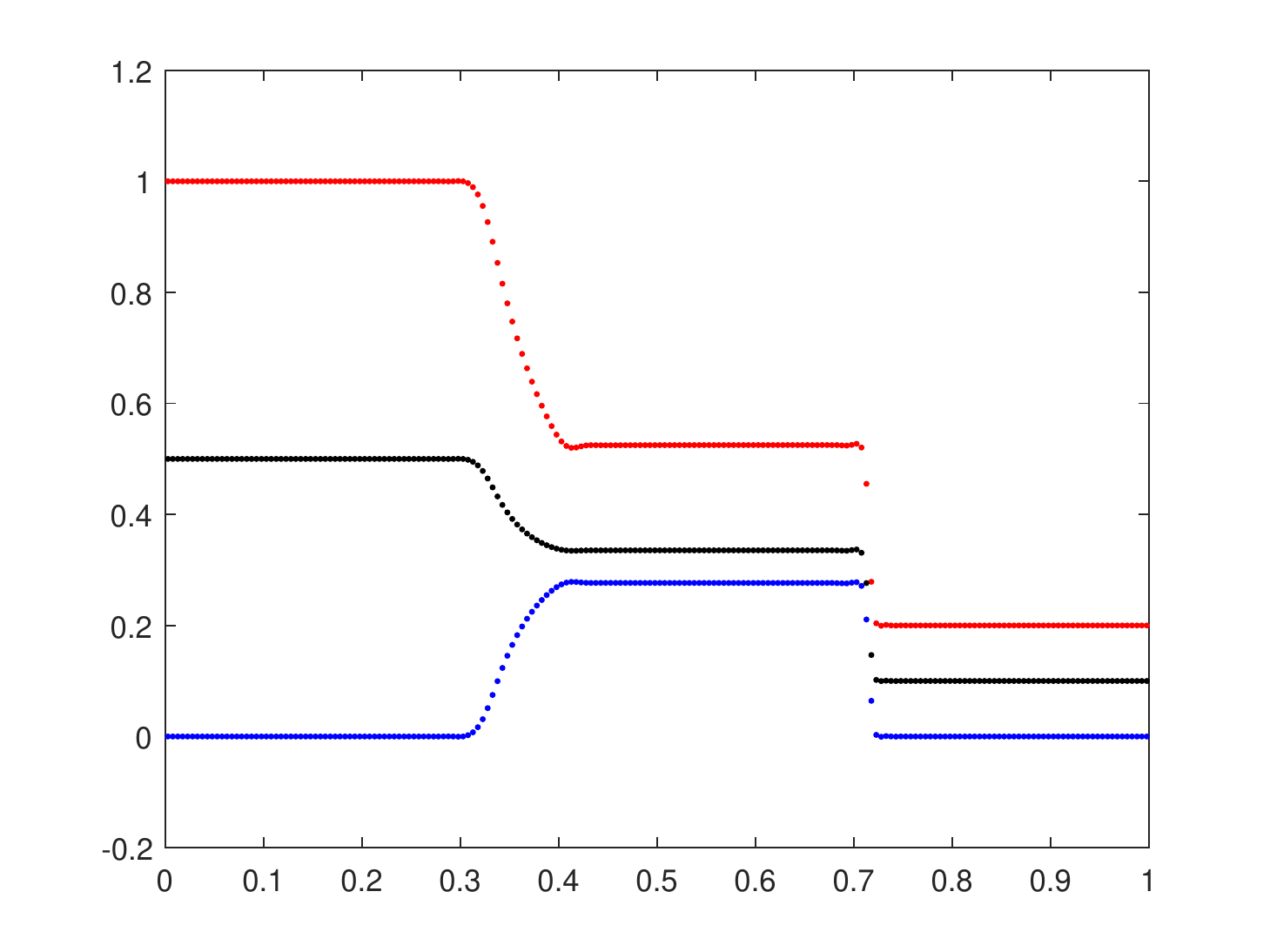}
			\subcaption{DIRK2, Q-CWENO23, CFL$=0.8$}
		\end{subfigure}	
		\begin{subfigure}[b]{0.45\linewidth}
			\includegraphics[width=1\linewidth]{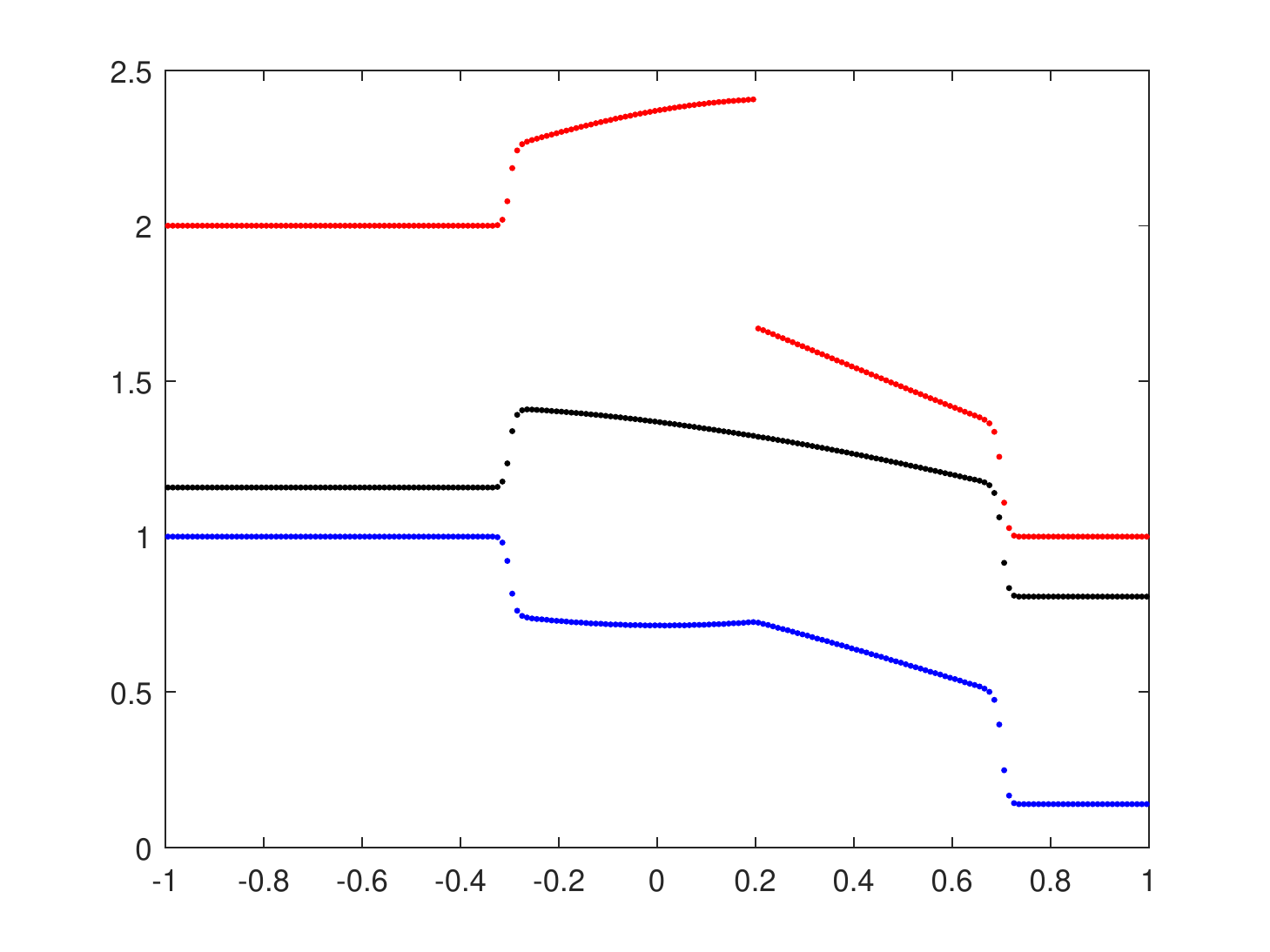}	
			\subcaption{DIRK2, Q-CWENO23, CFL$=1.9$}
		\end{subfigure}	
		\begin{subfigure}[b]{0.45\linewidth}
			\includegraphics[width=1\linewidth]{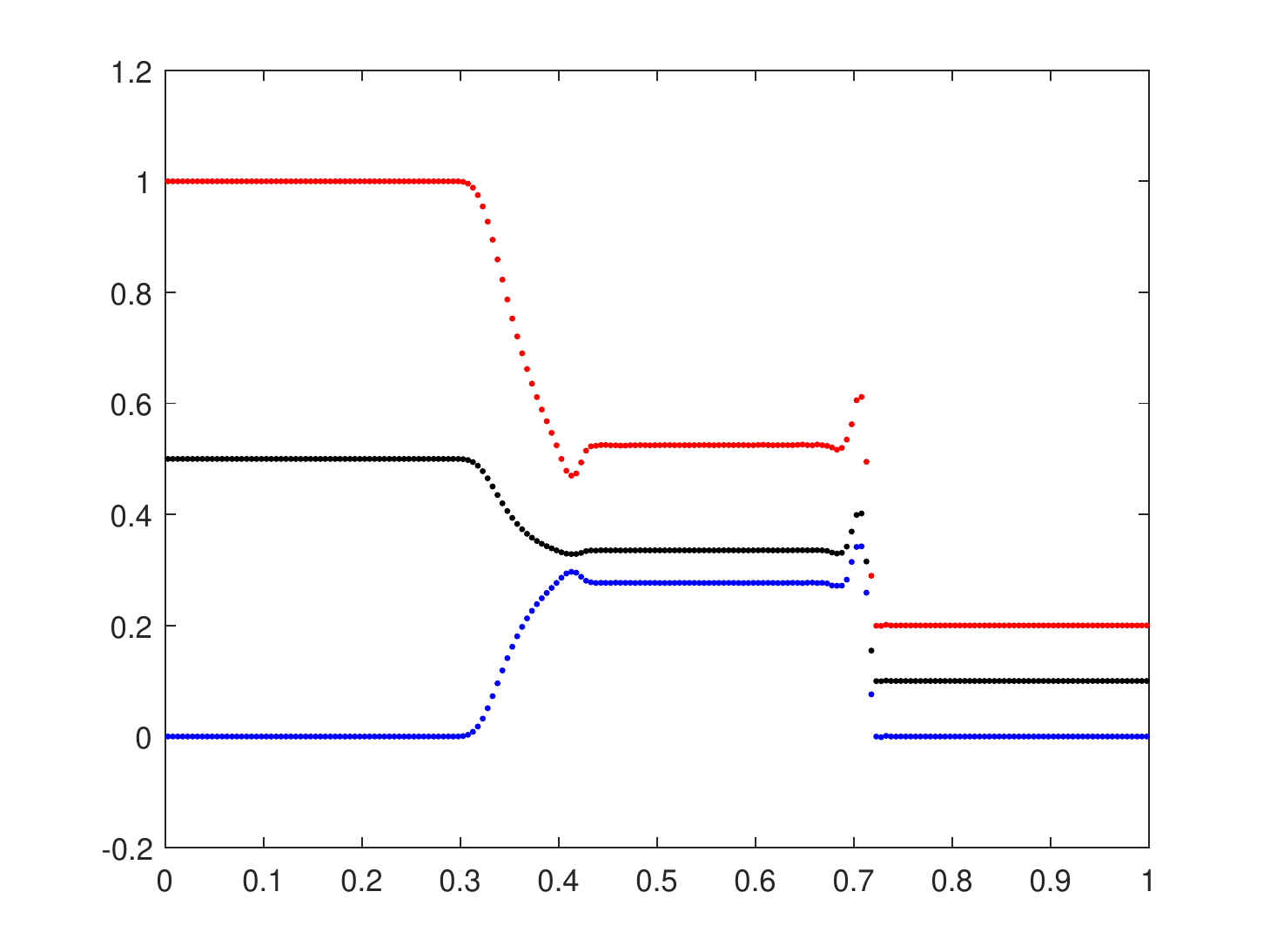}	
			\subcaption{DIRK2, Q-CWENO23, CFL$=1.9$}
		\end{subfigure}		
		%
		\caption{ Shock tests for 1D Broadwell model. Macroscopic variables $\rho$ (red), $m$ (blue) and $z$ (black). Left: Case 1 in \eqref{shock cases}, Right: Case 2 in \eqref{shock cases}.}\label{fig Broadwell case1 shock}  	
	\end{figure}

	\subsection{2D simplified Xin-Jin model}
	For 2D tests, we here consider the DIRK2 based method. 
	\subsubsection{Accuracy test}
	Here, we use well-prepared initial data:
	\begin{align}\label{initial acc 2d xinjin}
	u_0(x,y)=0.8\sin^2(\pi x)\sin^2(\pi y), \quad v_0(x,y)= \frac{u_0^2(x,y)}{2} + \kappa \left( u_0^2(x,y) -1\right)   (\partial_x{u_0}(x,y) + \partial_y{u_0}(x,y)).
	\end{align}
	The computation is performed in $ (x,y) \in [0, 1]^2$ with the periodic boundary condition with $N_x=N_y$. 
	In this problem, 
	the breaking time is $\displaystyle T_b = \frac{1}{0.6\pi\sqrt{3}}\approx 0.3063$, we take a final time as $\displaystyle T^f=0.15$. Since $|u_0| < 1$, the subcharacteristic condition is satisfied. We restrict the ratio to satisfy $\frac{\Delta t}{\Delta x}\leq 1$. In Fig. \ref{fig 2 2d}, we confirm that SL schemes based on DIRK2 and BDF2 attains desired accuracy between 2 and 3 for all ranges of the relaxation parameter $\kappa$.
	
	\begin{figure}[htbp]
		\centering
		\begin{subfigure}[b]{0.45\linewidth}
			\includegraphics[width=1\linewidth]{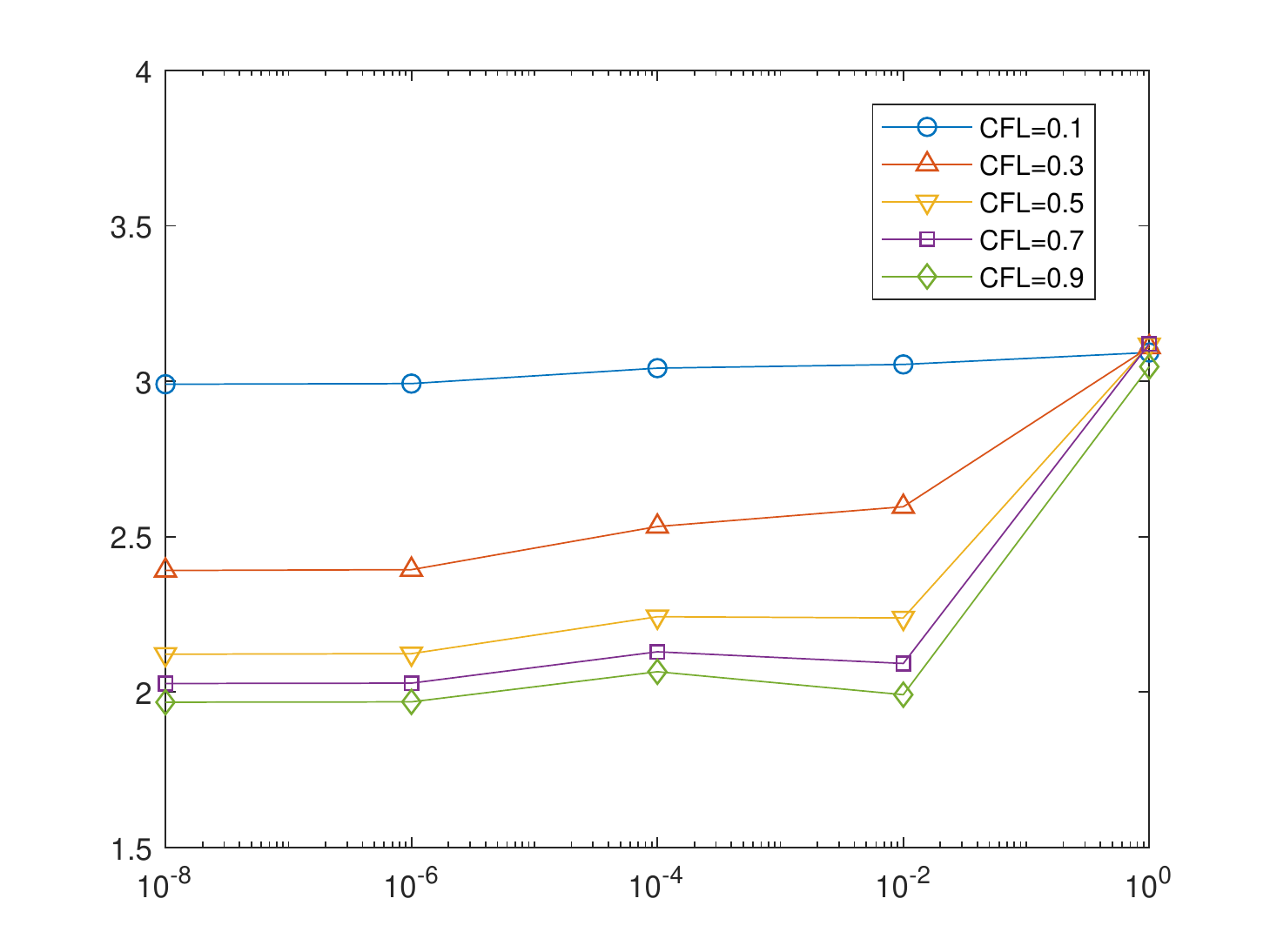}
			\subcaption{DIRK2 - Q-CWENO23}
		\end{subfigure}	
		\begin{subfigure}[b]{0.45\linewidth}
			\includegraphics[width=1\linewidth]{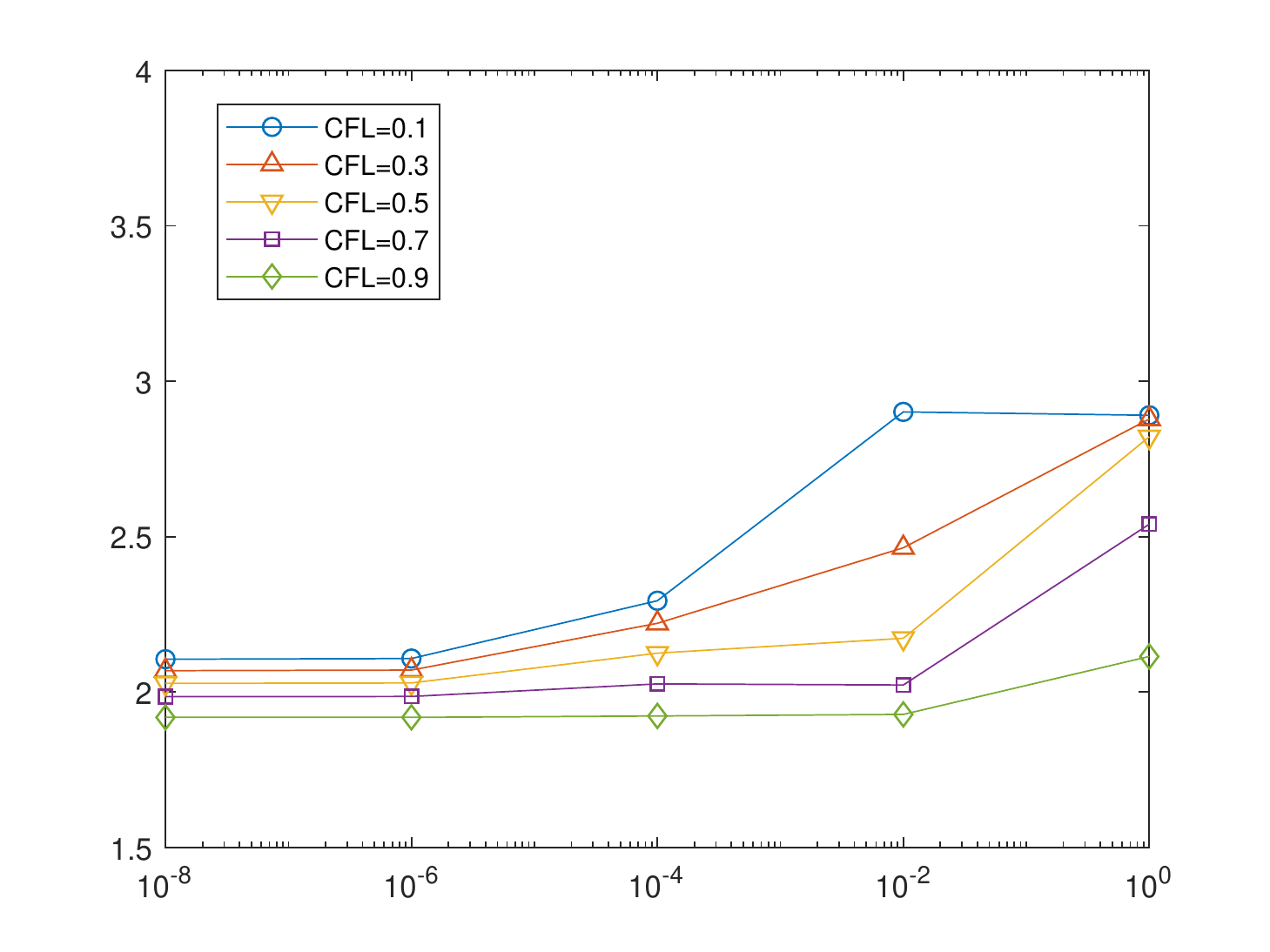}	
			\subcaption{BDF2 - Q-CWENO23}
		\end{subfigure}	
		\caption{Accuracy tests for 2D Xin-Jin model. Initial data is associated to \eqref{initial acc 2d xinjin}. $x$-axis is for the relaxation parameter $\kappa$ and $y$-axis is for order of accuracy based on $N_x^2=N_y^2=160^2, 320^2, 640^2$.}\label{fig 2 2d}  	
	\end{figure}
	

	\subsubsection{Shock tests}
	Now, we move on to 2D shock tests for \eqref{relaxation system}.
	\\
	{\bf $\bullet$ Smooth initial data.}
	Here, we solve the relaxation system \eqref{relaxation system} to capture the profile of the shock in Burgers equation. For this, we consider the following initial data:
	\begin{align}\label{initial shock smmoth 2d xinjin}
	u_0(x,y)=0.8\sin^2(\pi x)\sin^2(\pi y), \quad v_0(x,y)= \frac{u_0^2(x,y)}{2}.
	\end{align}
	on the periodic domain $ (x,y) \in [0, 1]^2$ with grid points $N_x=400$ and mesh ratio $\frac{\Delta t}{\Delta x}=\frac{\Delta t}{\Delta y}=0.2$. In Fig. \ref{fig 10}, results are reported for $t=1,2,3$. Here, we only present result using 2D SL methods based on DIRK2 and Q-CWENO23.\\
	\begin{figure}[htbp]
		\centering	
		\begin{subfigure}[b]{0.45\linewidth}
			\includegraphics[width=1\linewidth]{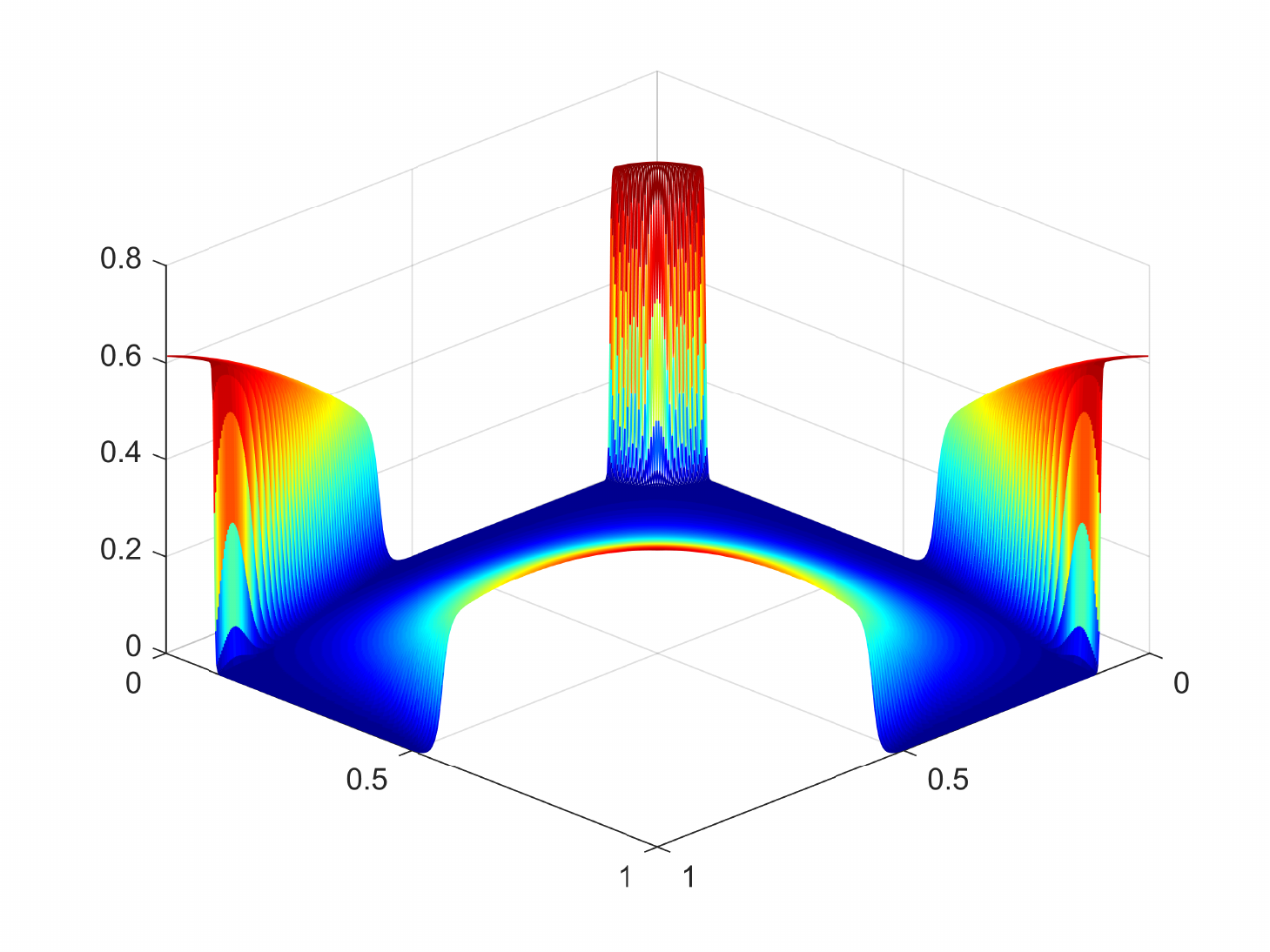}
			\subcaption{$t=1$}
		\end{subfigure}		
		\begin{subfigure}[b]{0.45\linewidth}
			\includegraphics[width=1\linewidth]{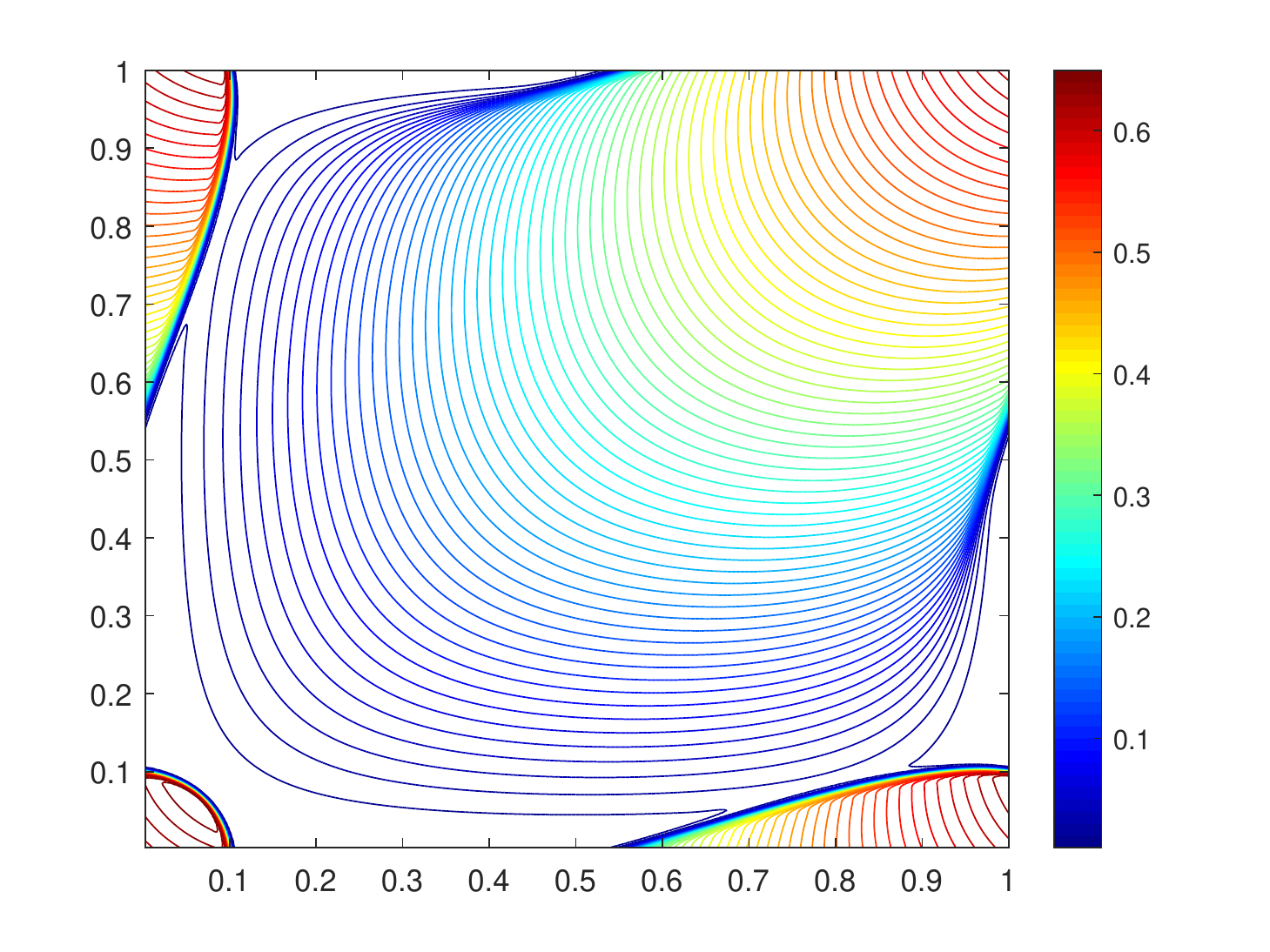}
			\subcaption{$t=1$}
		\end{subfigure}	
		\begin{subfigure}[b]{0.45\linewidth}
			\includegraphics[width=1\linewidth]{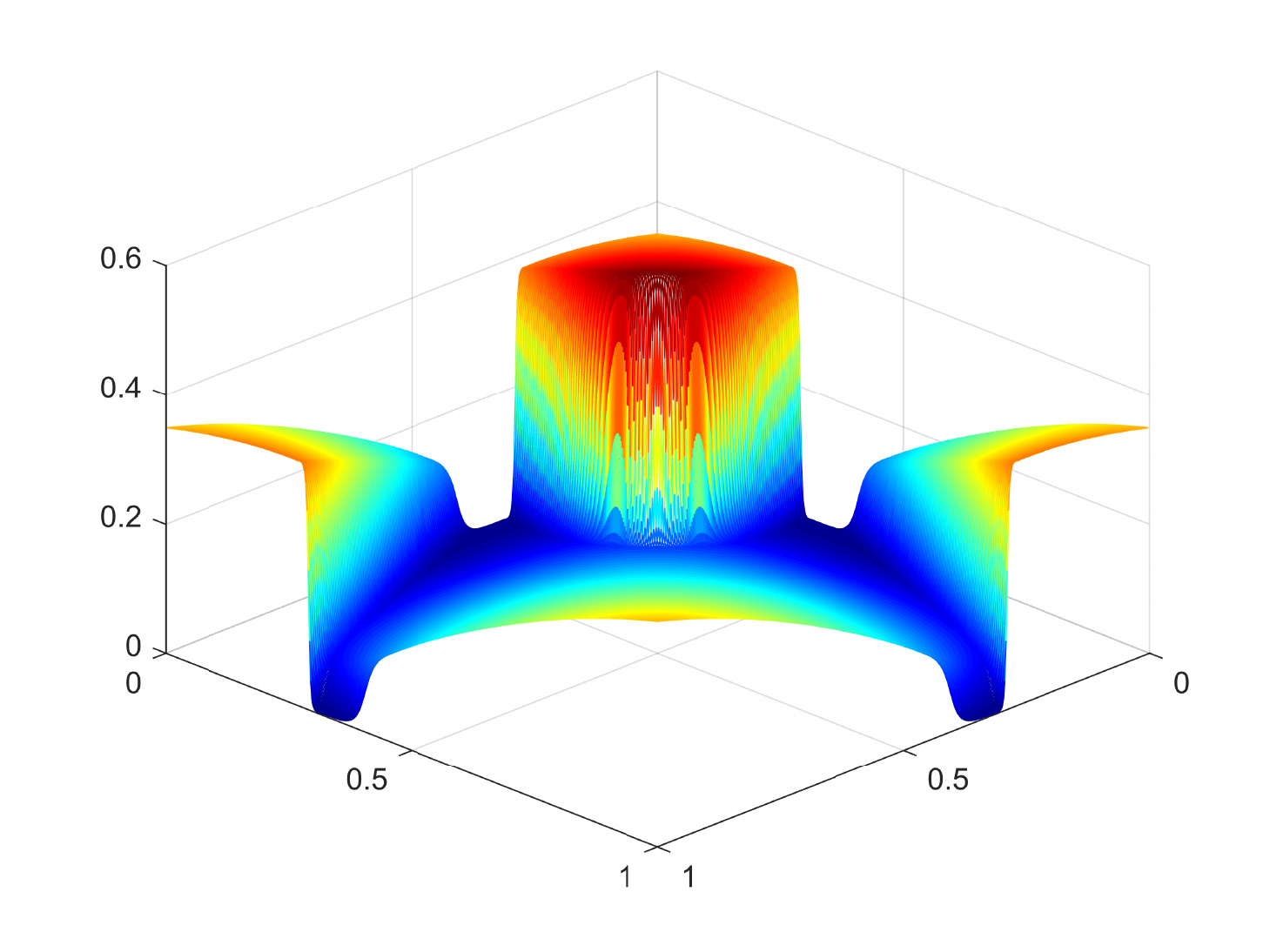}
			\subcaption{$t=2$}
		\end{subfigure}	
		\begin{subfigure}[b]{0.45\linewidth}
			\includegraphics[width=1\linewidth]{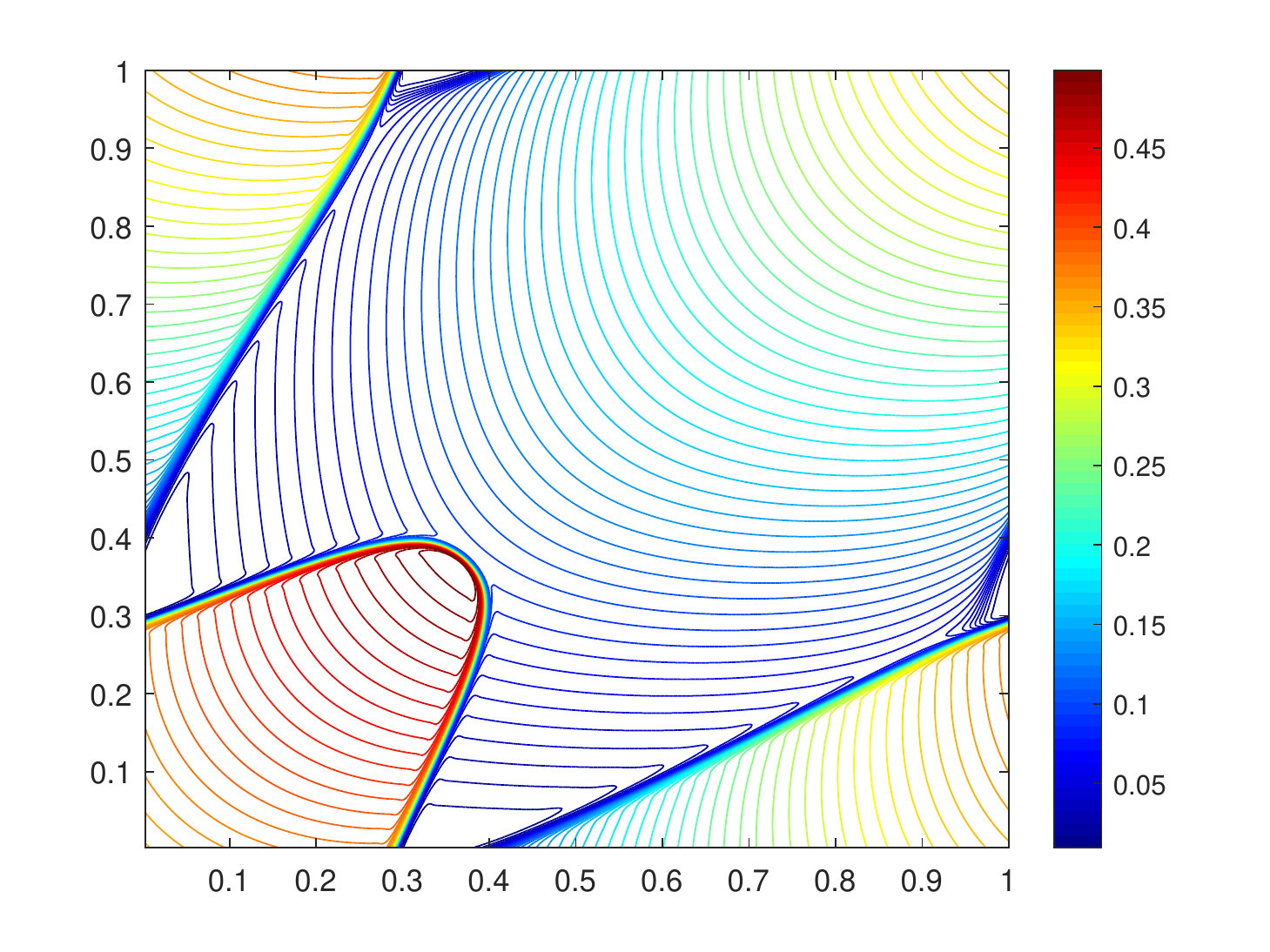}
			\subcaption{$t=2$}
		\end{subfigure}	
		\begin{subfigure}[b]{0.45\linewidth}
			\includegraphics[width=1\linewidth]{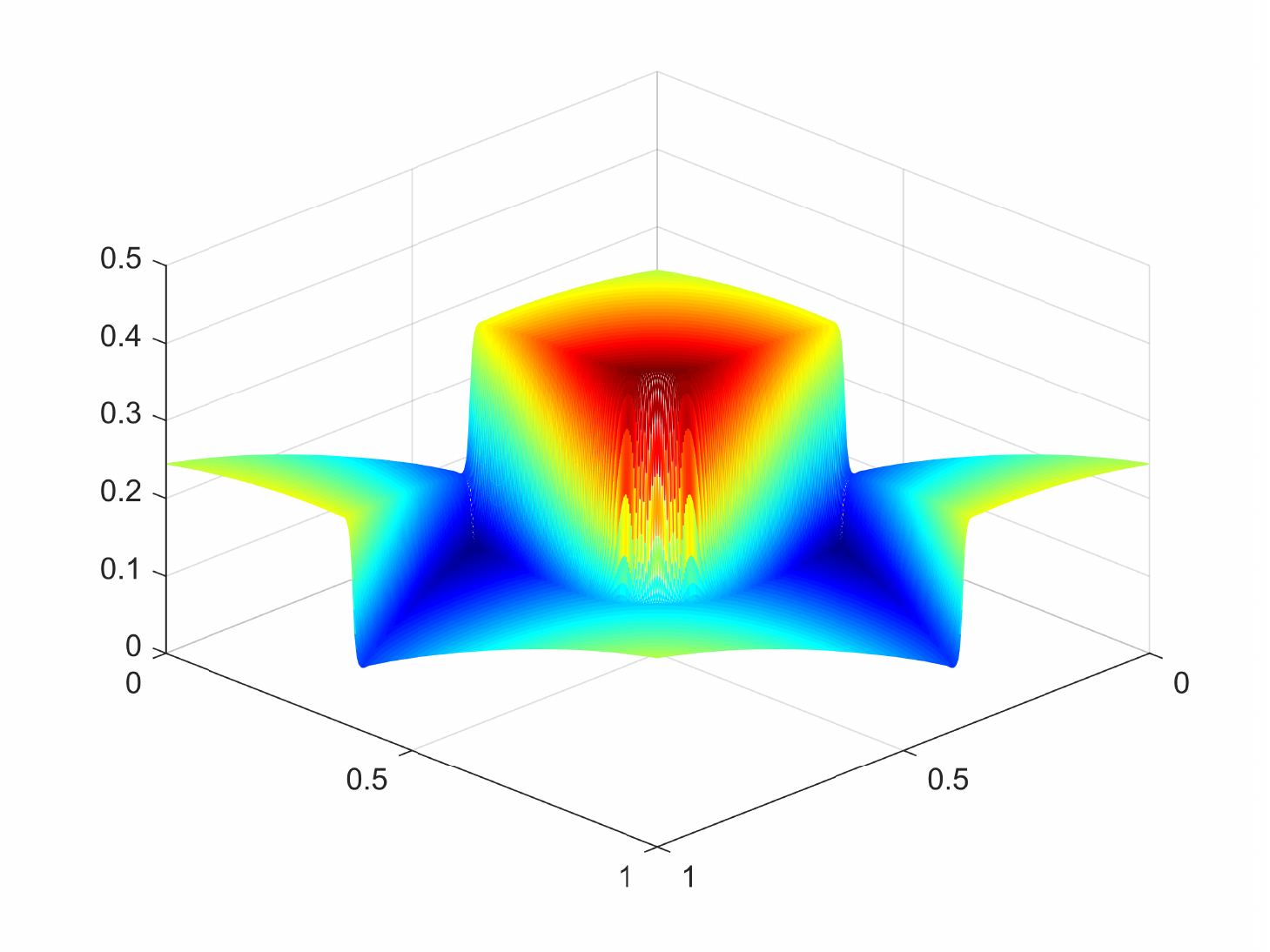}	
			\subcaption{$t=3$}
		\end{subfigure}	
		\begin{subfigure}[b]{0.45\linewidth}
			\includegraphics[width=1\linewidth]{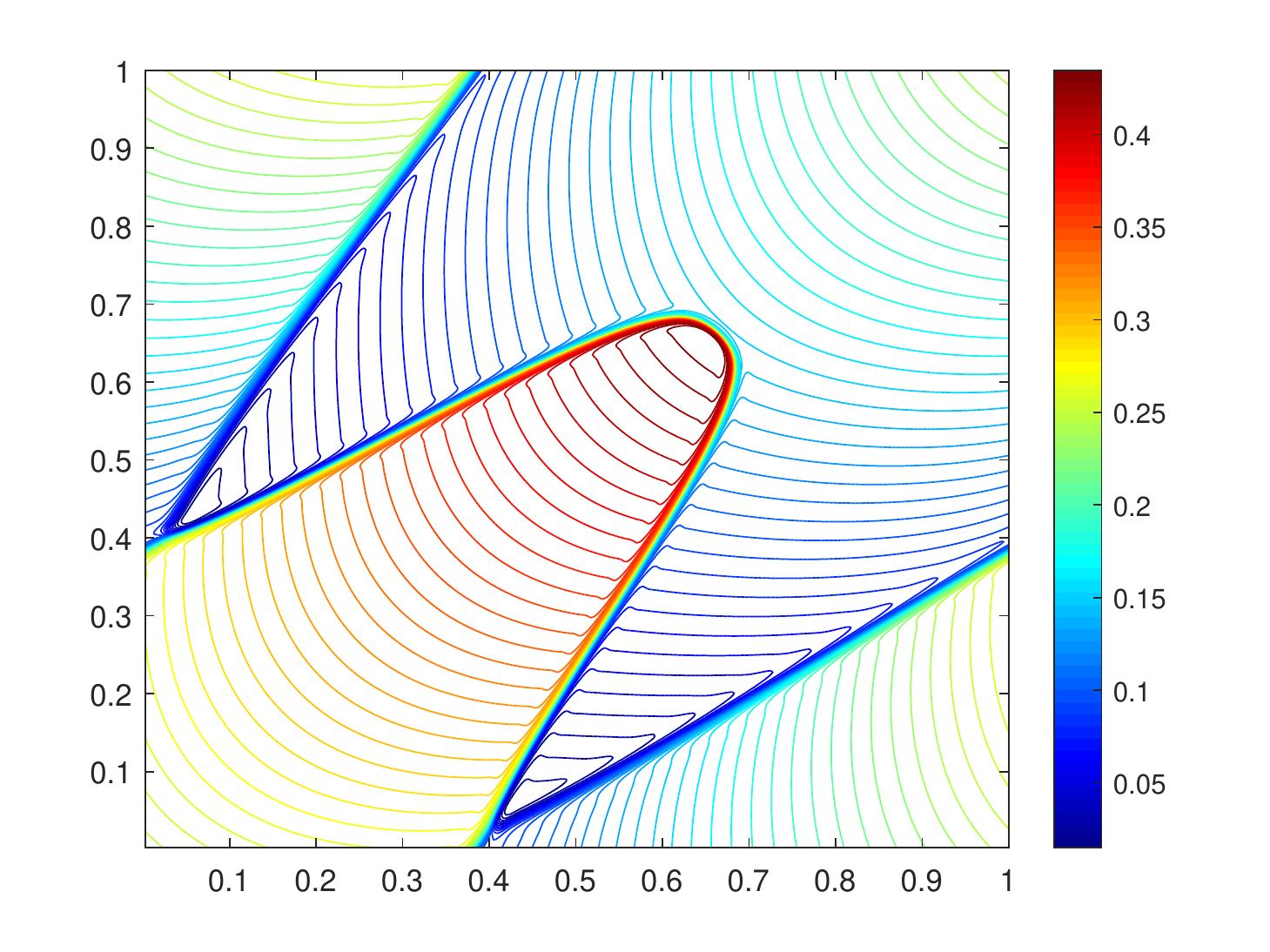}	
			\subcaption{$t=3$}
		\end{subfigure}	
		\caption{Shock test for 2D Xin-Jin model. Initial data is associated to \eqref{initial shock smmoth 2d xinjin}. Numerical solutions are obtained by DIRK2 based SL scheme for $\kappa=10^{-4}$. Mesh plot (left) and contour plot (right) of the solution $u$ at various times $t=1,2,3$. }\label{fig 10}  	
	\end{figure}

	
	{\bf $\bullet$ Discontinuous initial data.}
	This test has been solved by solving a viscous Burgers equation in \cite{DA}. Here, we instead solve the relaxation system \eqref{relaxation system} to capture the correct shock position of Burgers equation. Initial data is given by
	\begin{align}\label{initial shock dis 2d xinjin}
	u_0(x)=\begin{cases}
	-0.5, \quad x \leq 0, y \leq 0\\
	\, 0.25, \quad x \leq 0, y > 0\\
	\, 0.25, \quad x > 0, y \leq 0\\
	\, 0.5, \quad x > 0, y > 0
	\end{cases}, \quad v_0(x,y)= \frac{u_0^2(x,y)}{2}
	\end{align}
	with freeflow boundary condition $ (x,y) \in [-1, 1].^2$ with grid points $N_x=400$ and mesh ration $\frac{\Delta t}{\Delta x}=\frac{\Delta t}{\Delta y}=0.2$. In Fig. \ref{fig 11}, we plot the results for $t=1,2,3$. We only present result using 2D SL methods based on DIRK2 and Q-CWENO23.
	
	\begin{figure}[htbp]
		\centering	
		\begin{subfigure}[b]{0.45\linewidth}
			\includegraphics[width=1\linewidth]{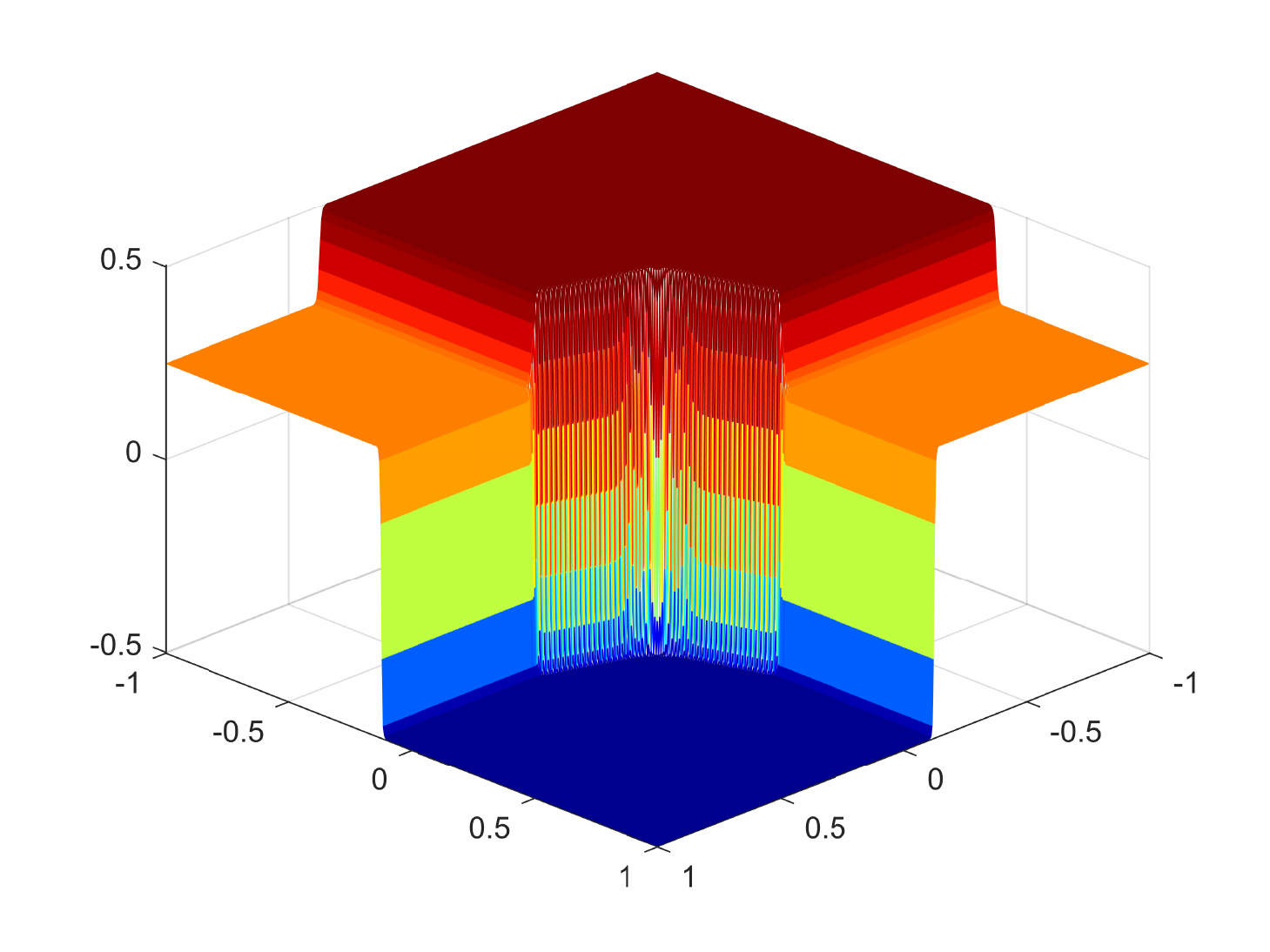}
			\subcaption{$t=1$}
		\end{subfigure}		
		\begin{subfigure}[b]{0.45\linewidth}
			\includegraphics[width=1\linewidth]{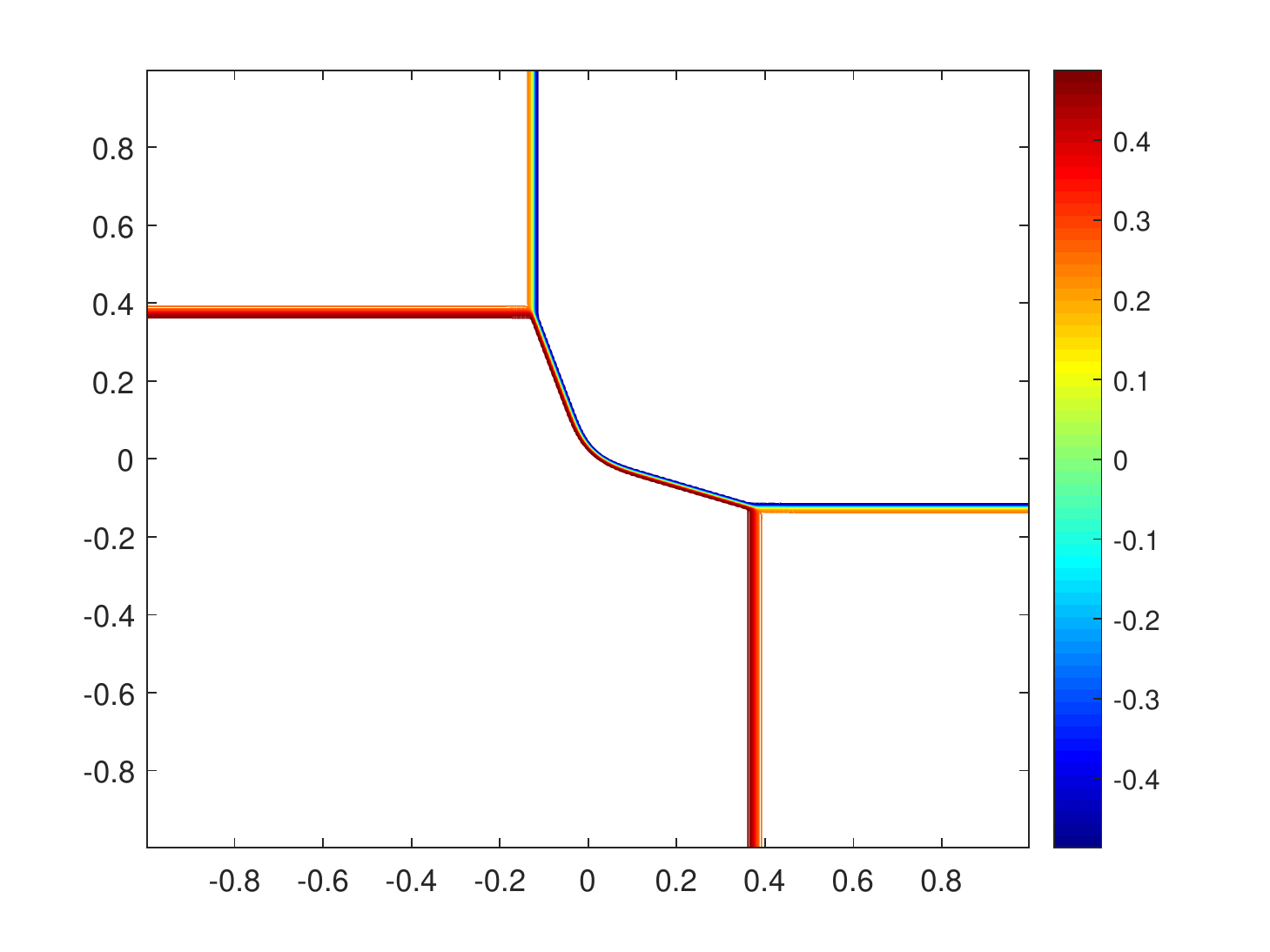}
			\subcaption{$t=1$}
		\end{subfigure}	
		\begin{subfigure}[b]{0.45\linewidth}
			\includegraphics[width=1\linewidth]{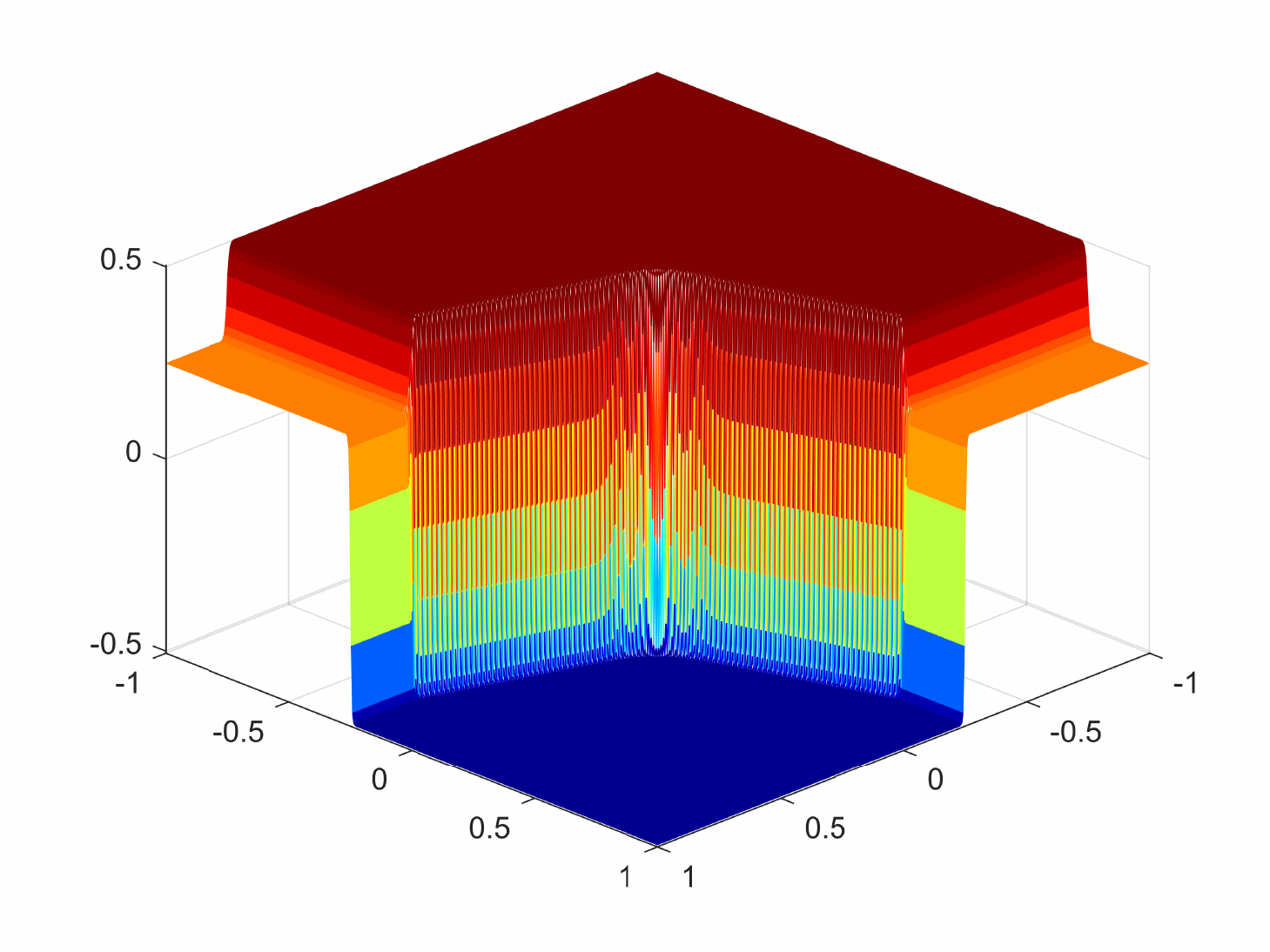}
			\subcaption{$t=2$}
		\end{subfigure}	
		\begin{subfigure}[b]{0.45\linewidth}
			\includegraphics[width=1\linewidth]{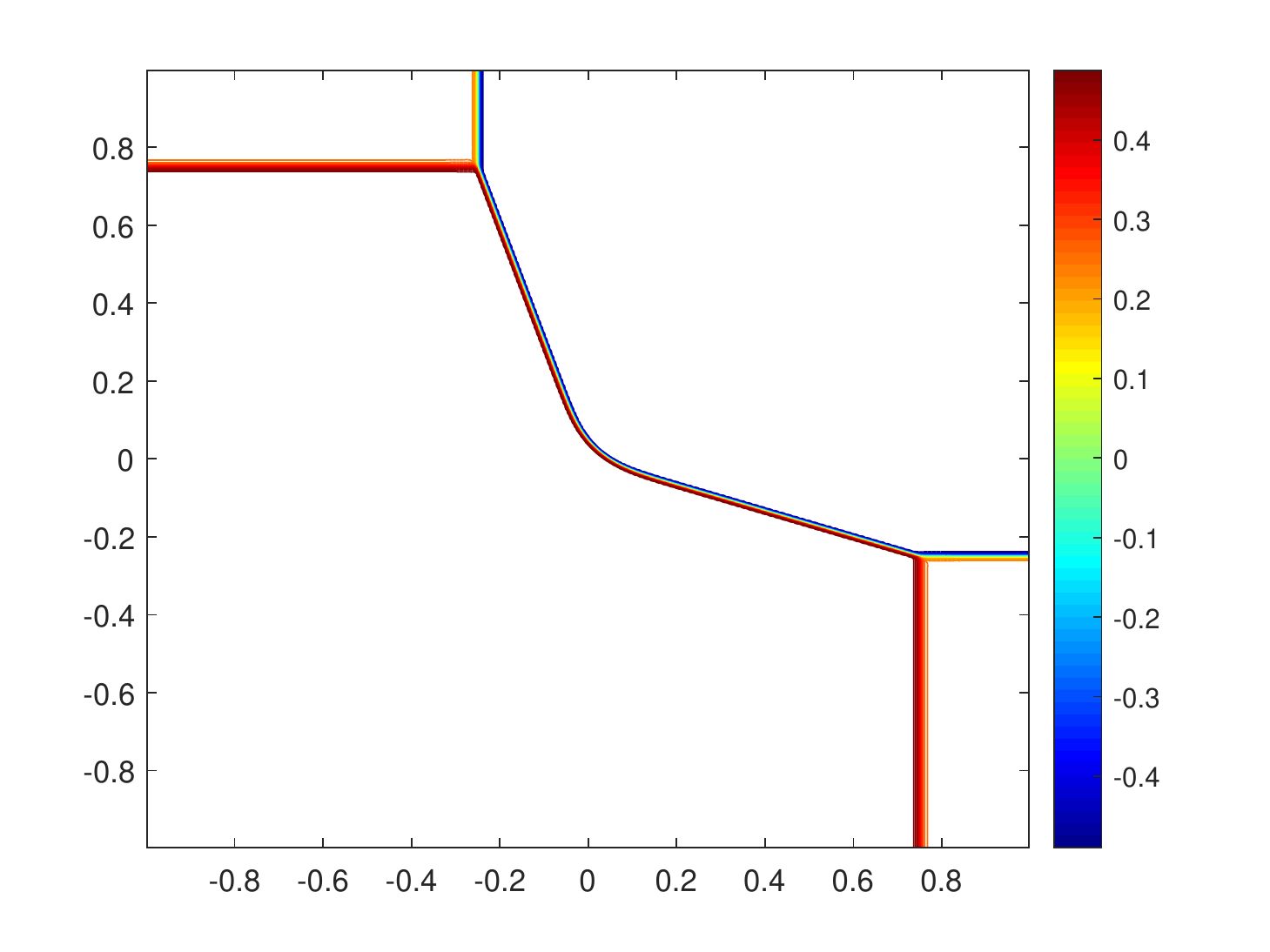}
			\subcaption{$t=2$}
		\end{subfigure}	
		\begin{subfigure}[b]{0.45\linewidth}
			\includegraphics[width=1\linewidth]{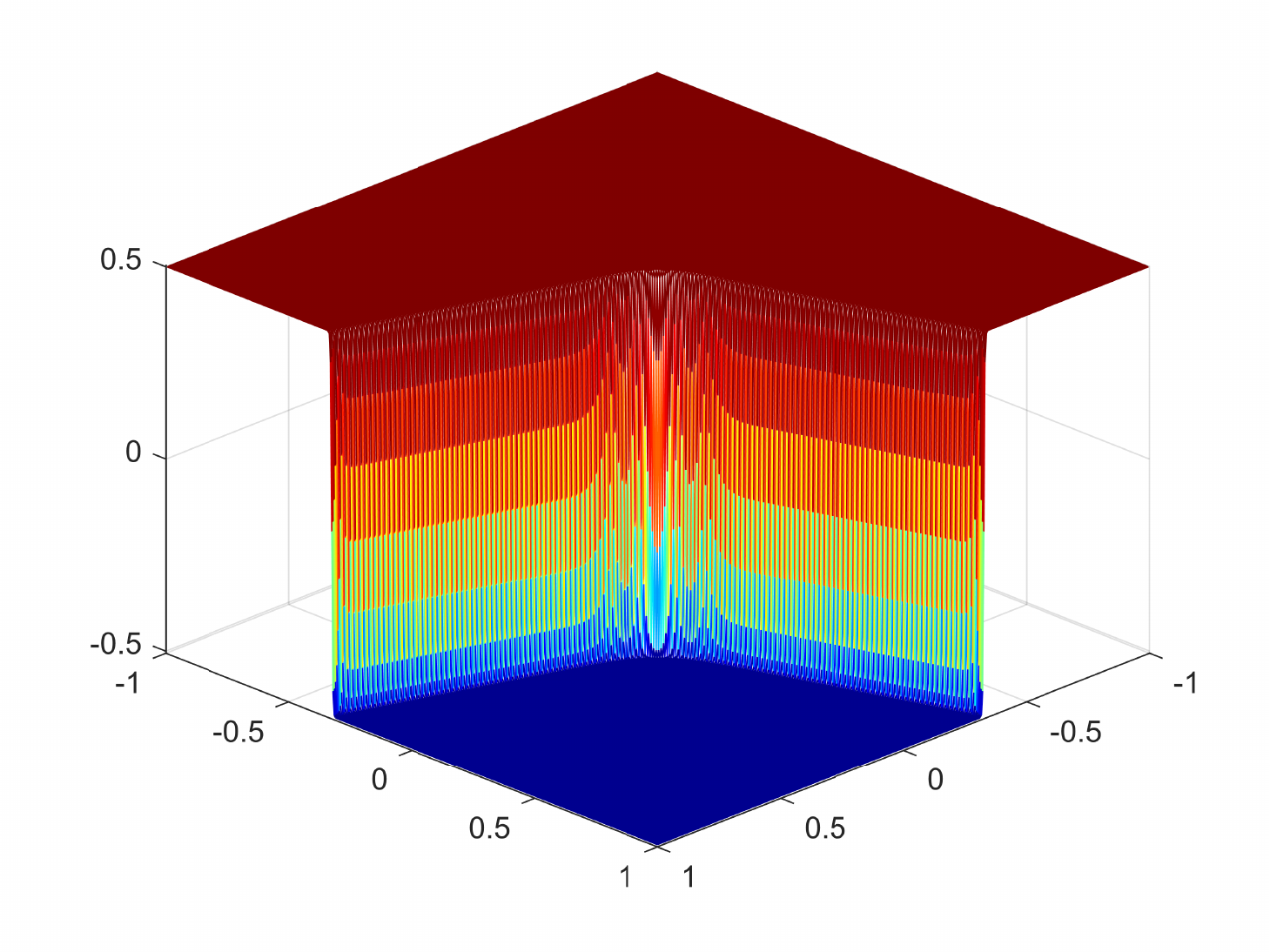}	
			\subcaption{$t=3$}
		\end{subfigure}	
		\begin{subfigure}[b]{0.45\linewidth}
			\includegraphics[width=1\linewidth]{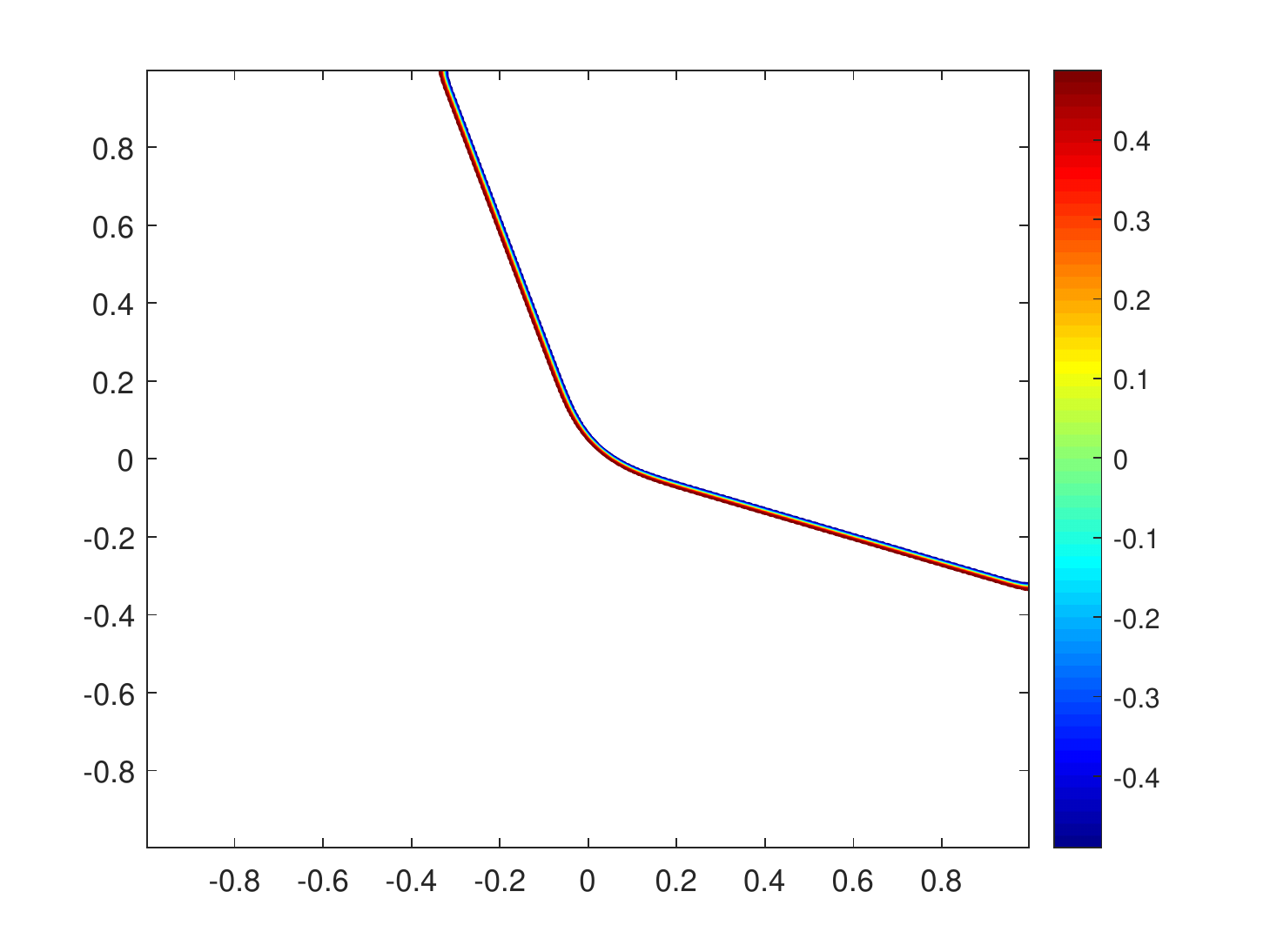}	
			\subcaption{$t=3$}
		\end{subfigure}	
		\caption{Shock test for 2D Xin-Jin model. Initial data is associated to \eqref{initial shock dis 2d xinjin}. Numerical solutions are obtained by DIRK2 based SL scheme for $\kappa=10^{-4}$. Mesh plot (left) and contour plot (right) of the solution $u$ at various times $t=1,2,3$.}\label{fig 11}  	
	\end{figure}

	\section{Conclusions}
	\label{sec:conclusions}
	We propose a simple technique to restore conservation in semi-Lagrangian schemes when non-linear reconstructions are adopted to avoid spurious oscillation or to preserve the positivity of the solution. The reconstruction is obtained by taking the sliding average of a basic non-oscillatory (positive-preserving) cell-average to point-wise reconstruction $R$, thus it inherits the non-oscillatory (positivity-preserving) property of $R$. A detailed analysis is performed of the proposed reconstruction, proving its accuracy and conservation properties, and its consistency with Lagrange interpolation in the case of linear basic reconstruction. Two dimensional extension is also considered and analyzed. The technique is then tested on the Xin-Jin relaxation system in one and two space dimensions, and on the 1D Broadwell model. Applications to BGK model and Vlasov-Poisson system will be presented in the second part of the paper.

	
	\appendix

	\section{Proof of Proposition \ref{Consistency}}\label{Appendix accuracy proof}
	\begin{proof}
		We first write \eqref{scheme} as
		\begin{align}\label{Q explicit 1d}
		\begin{split}
		Q_{i+\theta}&= \sum_{\ell=\text{even}}^k (\Delta x)^{\ell} \left(\alpha_{\ell}(\theta)R_{i}^{(\ell)} +\left(\frac{1}{(\ell+1)!}\left(\frac{1}{2}\right)^{\ell} - \alpha_{\ell}(\theta)\right)R_{i+1}^{(\ell)}\right)
		+ \sum_{\ell=\text{odd}}^k (\Delta x)^{\ell} \alpha_{\ell}(\theta)\left(R_{i}^{(\ell)}-R_{i+1}^{(\ell)}\right).
		\end{split}
		\end{align}	 
		This, together with the assumption \eqref{consistency condition}, gives
		\begin{align*}
		Q_{i+\theta}&= \sum_{\ell=\text{even}}^k (\Delta x)^{\ell} \left(\alpha_{\ell}(\theta)u_{i}^{(\ell)} +\left(\frac{1}{(\ell+1)!}\left(\frac{1}{2}\right)^{\ell} - \alpha_{\ell}(\theta)\right)u_{i+1}^{(\ell)}\right)
		\cr &\quad 
		+ \sum_{\ell=\text{odd}}^k (\Delta x)^{\ell} \alpha_{\ell}(\theta)\left(u_{i}^{(\ell)}-u_{i+1}^{(\ell)}\right) + (\Delta x)^{k+2}\cr
		&=  \sum_{\ell=0}^k (\Delta x)^{\ell} \left(\alpha_{\ell}(\theta)u_{i}^{(\ell)} +\beta_\ell(\theta)u_{i+1}^{(\ell)}\right)+ (\Delta x)^{k+2}.
		\end{align*}
		Using Taylor's expansion $u_{i+1}^{(\ell)} =  u_i^{(\ell)}+ u_i^{(\ell+1)}\Delta x+ \frac{1}{2}u_i^{(\ell+2)} (\Delta x)^2 + \frac{1}{6}u_i^{(\ell+3)}(\Delta x)^3  + \cdots$,
		we obtain
		\begin{align*}
		Q_{i+\theta} &=  \sum_{\ell=0}^k (\Delta x)^{\ell} \left(\alpha_{\ell}(\theta)u_{i}^{(\ell)} +\beta_\ell(\theta) \sum_{m=0}^{k+1-\ell} \frac{u_i^{(\ell+m)}}{m !}(\Delta x)^m \right)+ \mathcal{O}\left((\Delta x)^{k+2}\right)\cr
		&=  \sum_{\ell=0}^k (\Delta x)^{\ell} \alpha_{\ell}(\theta)u_{i}^{(\ell)} + \sum_{\ell=0}^k \sum_{m=0}^{k-\ell}(\Delta x)^{\ell+m} \beta_\ell(\theta)  \frac{u_i^{(\ell+m)}}{m !} + \sum_{\ell=0}^k (\Delta x)^{k+1} \beta_\ell(\theta)  \frac{u_i^{(k+1)}}{(k+1-\ell) !} + \mathcal{O}\left((\Delta x)^{k+2}\right)\cr
		&=:  \sum_{\ell=0}^k (\Delta x)^{\ell} \lambda_\ell(\theta) u_{i}^{(\ell)}+ \sum_{\ell=0}^k (\Delta x)^{k+1} \beta_\ell(\theta)  \frac{u_i^{(k+1)}}{(k+1-\ell) !} + \mathcal{O}\left((\Delta x)^{k+2}\right),
		\end{align*}
		where $\displaystyle \lambda_\ell(\theta)=\alpha_\ell(\theta)  + \sum_{m=0}^{\ell}\beta_m(\theta) \frac{1}{(\ell-m)!}$. Note that
		\begin{align*}
		\sum_{\ell=0}^k \beta_\ell(\theta)  \frac{1}{(k+1-\ell) !} 	
		&=\alpha_{k+1}(\theta)  + \sum_{m=0}^{k+1}\beta_{m}(\theta) \frac{1}{(k+1-m)!}=\lambda_{k+1}(\theta).
		\end{align*}
		The first equality follows from $\alpha_{k+1}(\theta) + \beta_{k+1}(\theta)=0$, which holds due to \eqref{ab odd} for an even integer $k$. To sum up,
		\begin{align*}
		Q_{i+\theta}&= \sum_{\ell=0}^{k+1} (\Delta x)^{\ell} \lambda_\ell(\theta) u_{i}^{(\ell)}+ \mathcal{O}\left((\Delta x)^{k+2}\right),
		\end{align*}
		and this can be written explicitly as follows:
		\begin{align*}
		Q_{i+\theta} 
		&= u_{i}^{(0)} + \theta u_{i}^{(1)}\Delta x + \left(\frac{\theta^2}{2}+ \frac{1}{24}\right) u_{i}^{(2)}(\Delta x)^2 + \left(\frac{\theta^3}{6}+ \frac{\theta}{24}\right) u_{i}^{(3)}(\Delta x)^3 \cr
		&\quad+ \left(\frac{\theta^4}{24} + \frac{\theta^2}{48}+ \frac{1}{1920}\right) u_{i}^{(4)}(\Delta x)^4 + \left( \frac{\theta^5}{120} + \frac{\theta^3}{144}+ \frac{\theta}{1920}\right) u_{i}^{(5)}(\Delta x)^5\cr 
		&+ \left(\frac{\theta^6}{720} + \frac{\theta^4}{576} + \frac{\theta^2}{3840} + \frac{1}{322560}\right) u_{i}^{(6)}(\Delta x)^6 + \cdots + \mathcal{O}\left((\Delta x)^{k+2}\right)\cr
		&= \sum_{\ell=0}^{k+1} \frac{\theta^{\ell}}{\ell !}u_{i}^{(\ell)}(\Delta x)^{\ell} + \frac{(\Delta x)^{2}}{24}\sum_{\ell=0}^{k-1} \frac{\theta^{\ell}}{\ell !}u_{i}^{(\ell+2)}(\Delta x)^{\ell} +\frac{(\Delta x)^{4}}{1920}\sum_{\ell=0}^{k-3} \frac{\theta^{\ell}}{\ell !}u_{i}^{(\ell+4)}(\Delta x)^{\ell}\cr
		&\quad+ \frac{(\Delta x)^6}{322560}\sum_{\ell=0}^{k-5} \frac{\theta^{\ell}}{\ell !}u_{i}^{(\ell+4)}(\Delta x)^{\ell} +\cdots + \mathcal{O}\left((\Delta x)^{k+2}\right).
		\end{align*}
		Consequently, we can derive
		\begin{align*}
		Q_{i+\theta} &= u(x_{i+\theta}) + \frac{(\Delta x)^{2}}{24}u^{(2)}(x_{i+\theta}) + \frac{(\Delta x)^{4}}{1920}u^{(4)}(x_{i+\theta}) + \cdots + 
		\mathcal{O}\left((\Delta x)^{k+2}\right)\cr
		&=\sum_{\ell=\text{even}}^k \left(\Delta x\right)^{\ell} u^{(\ell)}(x_{i+\theta}) 
		\frac{1}{(\ell+1)!}\left(\frac{1}{2}\right)^{\ell} + \mathcal{O}\left((\Delta x)^{k+2}\right)\cr
		&=\bar{u}(x_{i+\theta}) + \mathcal{O}\left((\Delta x)^{k+2}\right).
		\end{align*}
	\end{proof}

	\section{Proof of Remark \ref{rmk 2.1}}\label{appendix remark proof}
	Consider any polynomial reconstruction $R_i(x), R_{i+1}(x)\in \mathbb{P}^k$ of the form \eqref{R 1d taylor} such that
	\begin{align*}
	u^{(\ell)}(x_i) - R_i^{(\ell)}= \mathcal{O}\left((\Delta x)^{k+1-\ell}\right), \quad u^{(\ell)}(x_{i+1}) - R_{i+1}^{(\ell)}= \mathcal{O}\left((\Delta x)^{k+1-\ell}\right).
	\end{align*} 
	From the assumption that $R_i^{(\ell)}$ is represented by Lipschitz functions $F_\ell$ of $\{\bar{u}_{i-r},\cdots,\bar{u}_{i+s}\}$, we can write it as
	\begin{align*}
	R_i^{(\ell)} = u^{(\ell)}(x_i) + F_\ell\left(\bar{u}_{i-r},\cdots,\bar{u}_{i+s}\right)- u^{(\ell)}(x_i),
	\end{align*}
	where $F_\ell\left(\bar{u}_{i-r},\cdots,\bar{u}_{i+s}\right)- u^{(\ell)}(x_i) = \mathcal{O}\left((\Delta x)^{k+1-\ell}\right).$
	Also, in \eqref{p derivative}, one can see that the function $p_{i}^{(\ell)} \in \mathbb{P}^{k-\ell}$ is written with a Lipschitz function $G_\ell$ of $\{\bar{u}_{i-r},\cdots,\bar{u}_{i+s}\}$ such that
	\begin{align*}
	p_{i}^{(\ell)} = u^{(\ell)}(x_i) + G_\ell\left(\bar{u}_{i-r},\cdots,\bar{u}_{i+s}\right) - u^{(\ell)}(x_i),\quad G_\ell\left(\bar{u}_{i-r},\cdots,\bar{u}_{i+s}\right) - u^{(\ell)}(x_i) &= \mathcal{O}\left((\Delta x)^{k+1-\ell}\right).
	\end{align*}
	Now, let us define $H_\ell\left(\{\bar{u}_{i-r},\cdots,\bar{u}_{i+s}\}\right)$ by 
	$$H_\ell \left(\bar{u}_{i-r},\cdots,\bar{u}_{i+s}\right):=G_\ell\left(\bar{u}_{i-r},\cdots,\bar{u}_{i+s}\right) - F_\ell\left(\bar{u}_{i-r},\cdots,\bar{u}_{i+s}\right),$$ 
	then it is Lipschitz continuous w.r.t. $\{\bar{u}_{i-r},\cdots,\bar{u}_{i+s}\}$ and $H^\ell\left(\bar{u}_{i-r},\cdots,\bar{u}_{i+s}\right)= \mathcal{O}\left((\Delta x)^{k+1-\ell}\right)$.
	Consequently,	
	\begin{align*}
	&u^{(\ell)}(x_i) - R_i^{(\ell)} - \left(u^{(\ell)}(x_{i+1}) - R_{i+1}^{(\ell)}\right)\cr &\quad= \left\{u^{(\ell)}(x_i) - p_i^{(\ell)} -\left(u^{(\ell)}(x_{i+1}) - p_{i+1}^{(\ell)}\right)\right\} + \left(p_{i}^{(\ell)} - R_{i}^{(\ell)}\right) - \left(p_{i+1}^{(\ell)}  - R_{i+1}^{(\ell)}\right)\cr
	&\quad= \mathcal{O}\left((\Delta x)^{k+2-\ell}\right) + H_\ell \left(\{\bar{u}_{i-r},\cdots,\bar{u}_{i+s}\}\right)-H_\ell \left(\{\bar{u}_{i+1-r},\cdots,\bar{u}_{i+1+s}\}\right)\cr
	&\quad=\mathcal{O}\left((\Delta x)^{k+2-\ell}\right). 
	\end{align*} 	
	
	\section{Proof of Proposition \ref{1d consistency prop}}\label{Appendix consistency proof}
	\begin{proof}
		The proof is divided into three steps	
		
		\noindent\textbf{Step 1, Reconstruction of $R_i(x)$:} This polynomial reconstruction is also introduced in \cite{Shu98}. Given cell average values $\{\bar{u}_i\}$, we first consider a primitive function $U(x):= \int_{-\infty} ^{x}u(x) dx$ and compute its cell boundary values as 
		$$U(x_{i+j-\frac{1}{2}})= \int_{-\infty} ^{x_{i+j-\frac{1}{2}}}u(x) \,dx= \sum_{m=-\infty}^{i+j-1} \bar{u}_m \Delta x, \quad -r \leq j \leq r+1.$$
		Then, look for a polynomial $P_i(x)$ such that  $$P_i(x_{i+j-\frac{1}{2}})=U(x_{i+j-\frac{1}{2}}), \quad -r \leq j \leq r+1.$$
		By differentiating $P_i(x)$, we obtain the basic reconstruction $R_i(x)$ which satisfies 
		\[P_i'(x)=R_i(x)=u(x) + \mathcal{O}\left((\Delta x)^{k+1} \right)
		\]
		and the condition \eqref{Lag condition}.
		The resulting form of $R_i(x)$ is given by
		\begin{align*}
		\displaystyle		R_i(x)= \sum_{m=1}^{k+1}\sum_{j=0}^{m-1} \bar{u}_{i-r+j} \Delta x \left\{\frac{\sum_{\ell=0,\ell \neq m}^{k+1} \prod_{q=0,q \neq m,\ell }^{k+1} \left( x - x_{i-r+\ell-\frac{1}{2}}\right) }{\prod_{\ell=0,\ell \neq m}^{k+1}\left( x_{i-r+m-\frac{1}{2}} - x_{i-r+\ell-\frac{1}{2}}\right)}\right\}.
		\end{align*}	
		\noindent\textbf{Step 2, Reconstruction of $Q_{i+\theta}$:}
		To reconstruct $Q_{i+\theta}$, we first compute the cell average value of $R_i(x)$ on $[x_{i-\frac{1}{2}+\theta}, x_{i+\frac{1}{2}})$:
		\begin{align}\label{R i integral}
		\frac{1}{\Delta x}&\int_{x_{i-\frac{1}{2}+\theta}} ^{x_{i+\frac{1}{2}}}R_i(x) \,dx = \sum_{m=1}^{k+1}\sum_{j=0}^{m-1} \bar{u}_{i-r+j} \left\{\frac{\prod_{\ell=0,\ell \neq m}^{k+1} (r-\ell+1) -  \prod_{\ell=0,\ell \neq m}^{k+1} (\theta+r-\ell) }{\prod_{\ell=0,\ell \neq m}^{k+1}\left( m-\ell\right)}\right\},
		\end{align} 
		which directly come from
		\begin{align*}
		\int_{x_{i-\frac{1}{2}+\theta}} ^{x_{i+\frac{1}{2}}}&\left(\sum_{\ell=0,\ell \neq m}^{k+1} \prod_{q=0,q \neq m,\ell }^{k+1} \left( x - x_{i-r+\ell-\frac{1}{2}}\right)\right)  \,dx= 
		\prod_{\ell=0,\ell \neq m}^{k+1} x_{r-\ell+1} -  \prod_{\ell=0,\ell \neq m}^{k+1}  x_{\theta+r-\ell}.
		\end{align*} 
		Similarly, we compute the cell average of $R_{i+1}(x)$ on $[x_{i+\frac{1}{2}}, x_{i+\frac{1}{2}+\theta})$:
		\begin{align}\label{R ip1 integral}
		\frac{1}{\Delta x}&\int_{x_{i+\frac{1}{2}}} ^{x_{i+\frac{1}{2}+\theta}}R_{i+1}(x) \,dx
		= \sum_{m=1}^{k+1}\sum_{j=1}^{m} \bar{u}_{i-r+j} \left\{\frac{\prod_{\ell=0,\ell \neq m}^{k+1} (\theta +r-\ell) -  \prod_{\ell=0,\ell \neq m}^{k+1} (r-\ell) }{\prod_{\ell=0,\ell \neq m}^{k+1}\left( m-\ell\right)}\right\}.
		\end{align}
		Now, we insert \eqref{R i integral} and \eqref{R ip1 integral} into the identity for $Q_{i+\theta}$ in \eqref{Q decom}:
		\begin{align*}
		\begin{split}
		Q_{i+\theta}&=\frac{1}{\Delta x}\int_{x_{i-\frac{1}{2}+\theta}} ^{x_{i+\frac{1}{2}}}R_{i}(x) \,dx + \frac{1}{\Delta x}\int_{x_{i+\frac{1}{2}}} ^{x_{i+\frac{1}{2}+\theta}}R_{i+1}(x) \,dx,
		\end{split}
		\end{align*}
		and decompose this into four parts:
		\begin{align*}
		\begin{split}
		Q_{i+\theta}&=  \sum_{j=0}\sum_{m=j+1}^{k+1} \bar{u}_{i-r+j} \left\{\frac{\prod_{\ell=0,\ell \neq m}^{k+1} (r-\ell+1) -  \prod_{\ell=0,\ell \neq m}^{k+1} (\theta+r-\ell) }{\prod_{\ell=0,\ell \neq m}^{k+1}\left( m-\ell\right)}\right\}\cr
		&\quad+\sum_{j=k+1}\sum_{m=j}^{k+1} \bar{u}_{i-r+j} \left\{\frac{\prod_{\ell=0,\ell \neq m}^{k+1} (\theta +r-\ell) -  \prod_{\ell=0,\ell \neq m}^{k+1} (r-\ell) }{\prod_{\ell=0,\ell \neq m}^{k+1}\left( m-\ell\right)}\right\}\cr
		&\quad+ \sum_{j=1}^{k}\sum_{m=j} \bar{u}_{i-r+j} \left\{\frac{\prod_{\ell=0,\ell \neq m}^{k+1} (\theta +r-\ell) -  \prod_{\ell=0,\ell \neq m}^{k+1} (r-\ell) }{\prod_{\ell=0,\ell \neq m}^{k+1}\left( m-\ell\right)}\right\}\cr
		&\quad+ \sum_{j=1}^{k}\sum_{m=j+1}^{k+1} \bar{u}_{i-r+j} \left\{\frac{\prod_{\ell=0,\ell \neq m}^{k+1} (r-\ell+1) -  \prod_{\ell=0,\ell \neq m}^{k+1} (r-\ell) }{\prod_{\ell=0,\ell \neq m}^{k+1}\left( m-\ell\right)}\right\}\cr
		&=:I_1 + I_2 + I_3 +I_4.
		\end{split}
		\end{align*}
		With the identity $\displaystyle \sum_{m=0}^{k+1}\frac{ \prod_{\ell=0,\ell \neq m }^{k+1} \left(\theta+r-\ell\right) }{\prod_{\ell=0,\ell \neq m}^{k+1}\left( m -\ell \right)}=1$, we can simplify $I_1$ as
		\begin{align}\label{I_1}
		\begin{split}
		I_1&= \bar{u}_{i-r} \sum_{m=1}^{k+1}  \left\{\frac{\prod_{\ell=0,\ell \neq m}^{k+1} (r-\ell+1) -  \prod_{\ell=0,\ell \neq m}^{k+1} (\theta+r-\ell) }{\prod_{\ell=0,\ell \neq m}^{k+1}\left( m-\ell\right)}\right\}\cr
		&= \bar{u}_{i-r} \left(\sum_{m=r+1}  \frac{\prod_{\ell=0,\ell \neq m}^{k+1} (r-\ell+1) }{\prod_{\ell=0,\ell \neq m}^{k+1}\left( m-\ell\right)}- \sum_{m=1}^{k+1}  \frac{\prod_{\ell=0,\ell \neq m}^{k+1} (\theta+r-\ell) }{\prod_{\ell=0,\ell \neq m}^{k+1}\left( m-\ell\right)}\right)\cr
		&= \bar{u}_{i-r} \left( 1- \sum_{m=1}^{k+1}  \frac{\prod_{\ell=0,\ell \neq m}^{k+1} (\theta+r-\ell) }{\prod_{\ell=0,\ell \neq m}^{k+1}\left( m-\ell\right)}\right)\cr
		&= \bar{u}_{i-r}  \frac{\prod_{\ell=1}^{k+1} (\theta+r-\ell) }{\prod_{\ell=1}^{k+1}\left( -\ell\right)}.		
		\end{split}
		\end{align} 
		The $I_2$ and $I_3$ terms are calculated as
		\begin{align}\label{I_2}
		\begin{split}
		I_2&=\bar{u}_{i+r+1} \left\{\frac{\prod_{\ell=0}^{k} (\theta +r-\ell) -  \prod_{\ell=0}^{k} (r-\ell) }{\prod_{\ell=0}^{k}\left( k+1-\ell\right)}\right\}=\bar{u}_{i+r+1} \frac{\prod_{\ell=0}^{k} \theta +r-\ell}{\prod_{\ell=0}^{k} k+1-\ell}\cr
		I_3&= \sum_{j=1}^{k}\bar{u}_{i-r+j} \left\{\frac{\prod_{\ell=0,\ell \neq j}^{k+1} (\theta +r-\ell) -  \prod_{\ell=0,\ell \neq j}^{k+1} (r-\ell) }{\prod_{\ell=0,\ell \neq j}^{k+1}\left( j-\ell\right)}\right\}\cr
		&= \bar{u}_{i} \left(\frac{\prod_{\ell=0,\ell \neq r}^{k+1} (\theta +r-\ell) }{\prod_{\ell=0,\ell \neq r}^{k+1}\left( r-\ell\right)} -1\right) + \sum_{j=1, j\neq r}^{k}\bar{u}_{i-r+j} \left\{\frac{\prod_{\ell=0,\ell \neq j}^{k+1} \theta +r-\ell }{\prod_{\ell=0,\ell \neq j}^{k+1} j-\ell}\right\}.
		\end{split}
		\end{align}
		A direct computation leads to $I_4=\bar{u}_i$.
		This, combined with \eqref{I_1}, \eqref{I_2}, gives
		\begin{align}\label{consistency last Q}
		Q_{i+\theta}
		&=\sum_{j=0}^{k+1} \bar{u}_{i-r+j}\prod_{\ell=0, \ell \neq j}^{k+1}\frac{\theta+r-\ell}{j-\ell }.
		\end{align}
		\textbf{Step 3, Comparison of $Q_{i+\theta}$ and $L(x_{i+\theta})$:}	
		For the comparison, we consider a Lagrange polynomial $L(x)$ which satisfies $L(x_{i-r+j}) = \bar{u}_{i-r+j}$ for $0 \leq j\leq k+1$:
		\begin{align*}
		L(x)=\sum_{j=0}^{k+1} \bar{u}_{i-r+j}\prod_{\ell=0, \ell \neq j}^{k+1}\frac{x-x_{i-r+\ell }}{x_{i-r+j}-x_{i-r+\ell }}.
		\end{align*}
		Inserting $x=x_{i+\theta}$ into $L(x)$, we obtain $L(x_{i+\theta})=Q_{i+\theta}$. This completes the proof.
	\end{proof}


	\section{Proof that the condition \eqref{consistency condition} in Proposition \ref{Consistency} is satisfied by both CWENO23 and CWENO23Z.
	}\label{appendix CWENO23andZ}
	Here, we check if the condition \eqref{consistency condition} is satisfied by \eqref{R 1d}. For this, we assume that $u$ is smooth enough so that $\omega_C^i = \frac{1}{2}+ e_C^i $ for $e_C^i=O((\Delta x)^2)$. We refer to \cite{kolb2014full} for the assumption. Then, we have
	\begin{align}\label{R esti}
	\begin{split}
	R_i^{(0)}&= \bar{u}_i  -  \frac{1}{12}\bigg(C_C+ e_C^i \bigg)\left(\bar{u}_{i+1}-2\bar{u}_i+\bar{u}_{i-1}\right) \cr
	&= \bar{u}_i  -  \frac{1}{12}\bigg(\frac{1}{2}+ \mathcal{O}((\Delta x)^2) \bigg)\left((\Delta x)^2\bar{u}''_i+ \mathcal{O}((\Delta x)^4)\right)  \cr
	&= \bar{u}_i  -  \frac{1}{24}\left((\Delta x)^2\bar{u}''_i\right) + \mathcal{O}((\Delta x)^4) \cr
	&= u_i + \mathcal{O}((\Delta x)^4).
	\end{split}
	\end{align}
	Similarly, we write $\omega_L^i=\frac{1}{4}+ e_L^i$ and $\omega_R^i=\frac{1}{4}+ e_R^i$ with $e_L^i,e_R^i = \mathcal{O}((\Delta x)^2)$. Then,
	\begin{align*}
	R_i^{(1)}&=\omega_L^i \frac{\bar{u}_{i}-\bar{u}_{i-1}}{\Delta x} + \omega_R^i \frac{\bar{u}_{i+1}-\bar{u}_{i}}{\Delta x} + \omega_C^i\frac{\bar{u}_{i+1}-\bar{u}_{i-1}}{2\Delta x}\cr
	&=\bigg(\frac{1}{4}+ e_L^i\bigg) \bigg[\bar{u}'_{i}-\frac{\Delta x}{2}\bar{u}''_{i}+\frac{(\Delta x)^2}{6}\bar{u}'''_{i} \bigg] + \bigg(\frac{1}{4}+ e_R^i\bigg) \bigg[
	\bar{u}'_{i}+\frac{\Delta x}{2}\bar{u}''_{i}+\frac{(\Delta x)^2}{6}\bar{u}'''_{i}\bigg] \cr
	&\quad + \bigg(\frac{1}{2}+ e_C^i\bigg) \bigg[\bar{u}'_{i} +\frac{1}{6}(\Delta x)^2 \bar{u}'''_{i}\bigg]+ \mathcal{O}((\Delta x)^3)\cr
	&= \bar{u}'_i + \frac{(\Delta x)^2}{6} \bar{u}'''_{i} + \mathcal{O}((\Delta x)^3).
	\end{align*}
	In the last equality, we used $\sum_k \omega_k^i=1$. Hence, we can obtain
	\begin{align}\label{R' esti}
	\begin{split}
	R_i^{(1)}-R_{i+1}^{(1)}&= \bar{u}'_i-\bar{u}'_{i+1} +\frac{(\Delta x)^2}{6}(  \bar{u}'''_{i} - \bar{u}'''_{i+1} )+ \mathcal{O}((\Delta x)^3)\cr
	&= \bar{u}'_i-\bar{u}'_{i+1} -\frac{(\Delta x)^2}{24}(  \bar{u}'''_{i} - \bar{u}'''_{i+1} )+ \mathcal{O}((\Delta x)^3)\cr
	&= \bigg(\bar{u}'_i-\frac{(\Delta x)^2}{24} \bar{u}'''_{i}\bigg) -\bigg(\bar{u}'_{i+1} -\frac{(\Delta x)^2}{24} \bar{u}'''_{i+1}\bigg) + \mathcal{O}((\Delta x)^3)= u'_i-u'_{i+1} + \mathcal{O}((\Delta x)^3).
	\end{split}
	\end{align}
	It is straightforward to show that
	\begin{align}\label{R'' esti}
	\begin{split}
	R_i^{(2)}=2 \omega_C^i \frac{\bar{u}_{i+1}-2\bar{u}_i+\bar{u}_{i-1}}{(\Delta x)^2} &=2\bigg(\frac{1}{2}+ e_C^i\bigg)\bigg[\frac{\bar{u}_{i+1}-2\bar{u}_i+\bar{u}_{i-1}}{(\Delta x)^2}\bigg]\cr
	&=\bigg(1+\mathcal{O}((\Delta x)^2)\bigg)\bigg[\bar{u}''_i + \frac{(\Delta x)^2}{12} \bar{u}^{(4)}_i\bigg] =u''_i + \mathcal{O}((\Delta x))^2.
	\end{split}
	\end{align}
	From \eqref{R esti},\eqref{R' esti} and \eqref{R'' esti}, we confirm that \eqref{R 1d} satisfies the condition \eqref{conservation condition} with $k=2$. 

	\section{Explicit form of $R_i^{(\ell)}$.}\label{Appendix CWENO35}

	For $k = 4$ in the 1D Algorithm \ref{algorithm 1d}, we can take CWENO35 reconstruction as a basic reconstruction $R$. We refer to \cite{C-2008} for details on CWENO35 reconstruction. Here we can represent it as the following explicit form of $R_i(x)$: 
	\begin{align}\label{R 1d cweno5} 
	R_i(x) = \sum_{\ell=0}^4 \frac{R_i^{(\ell)}}{\ell !} (x-x_i)^{(\ell)},
	\end{align}  
	with
	\begin{align}\label{CWENO5 R form}
	\begin{split}
	R_i^{(0)}&=\omega_C\left(\frac{577}{480}u_i - \frac{29}{240}u_{i-1} + \frac{19}{960}u_{i-2} - \frac{29}{240}u_{i+1} + \frac{19}{960}u_{i+2}\right) \cr
	&\quad- \omega_2\left(\frac{u_{i-1}-26u_i + u_{i+1}}{24}\right) + \omega_1 \left(\frac{23}{24}u_i + \frac{1}{12}u_{i-1} - \frac{1}{24}u_{i-2}\right) +\omega_3\left(\frac{23}{24}u_i + \frac{1}{12}u_{i+1} - \frac{1}{24}u_{i+2}\right)  \cr
	R_i^{(1)}&=-\omega_C\frac{8u_{i-1} - u_{i-2} - 8u_{i+1} + u_{i+2}}{12 \Delta x}  + \omega_1\frac{3u_i - 4u_{i-1} + u_{i-2}}{2 \Delta x}\cr
	&\quad - \omega_3\frac{3u_i - 4u_{i+1} + u_{i+2}}{2 \Delta x} 
	- \omega_2\frac{u_{i-1} - u_{i+1}}{2 \Delta x}\cr
	R_i^{(2)}&=2\bigg( \omega_1\frac{u_i - 2u_{i-1} + u_{i-2}}{2(\Delta x)^2} + \omega_2\frac{u_{i-1} - 2u_i + u_{i+1}}{2(\Delta x)^2} + \omega_3\frac{u_i - 2u_{i+1} + u_{i+2}}{2(\Delta x)^2}\bigg) \cr
	&\quad- 2\omega_C\bigg( \frac{10u_i - 6u_{i-1} + u_{i-2} - 6u_{i+1} + u_{i+2}}{4(\Delta x)^2}/\bigg)\cr	
	R_i^{(3)}&=6\omega_C\left( \frac{u_{i-1} - u_{i+1}}{3(\Delta x)^3} - \frac{u_{i-2} - u_{i+2}}{6(\Delta x)^3}\right),\quad 
	R_i^{(4)}=	24\omega_C \left( \frac{u_{i-2} + 6u_i+ u_{i+2}}{12 (\Delta x)^4} - \frac{u_{i-1} + u_{i+1}}{3(\Delta x)^4}\right).
	\end{split}
	\end{align}
	where the non-linear weights $\omega_k^i$ are computed as in \eqref{CWENO omega}. (See also \cite{C-2008}.)
	
	The CWENOZ5 reconstruction also can be directly obtained from \cite{CPSV-2017} with the following non-linear weights:
	\begin{align}\label{CWENOZ5 omega}
	\omega_k^i=\frac{\alpha_k^i}{\sum_\ell \alpha_\ell^i}, \quad \alpha_k^i=C_i\left(1+\frac{\tau}{\epsilon + {\beta}_k^i}\right)^t, \quad k,\ell \in \{1,\,2,\,3,\,C\},
	\end{align}
	where $t\geq 1$ and $\tau=|\beta_3^i-\beta_1^i|$.

	\section{Proof of Proposition \ref{Consistency 2d}}\label{Appendix accuracy proof 2d}
	\begin{proof}
		Recall the index set in \eqref{decomposition1}. For each index set, apply corresponding approximations in \eqref{consistency condition 2d} to \eqref{scheme 2d}. Then,
		\begin{align}\label{Q explicit 2d}
		\begin{split}
		Q_{i+\theta,j+\eta}=\sum_{|\ell|=0}^k (\Delta)^{\ell} &\bigg(\alpha_{\ell_1}(\theta)\alpha_{\ell_2}(\eta)u_{i,j}^{(\ell)}+ 
		\beta_{\ell_1}(\theta)\alpha_{\ell_2}(\eta)u_{i+1,j}^{(\ell)}\cr
		&+ 
		\alpha_{\ell_1}(\theta)\beta_{\ell_2}(\eta)u_{i,j+1}^{(\ell)}+
		\beta_{\ell_1}(\theta)\beta_{\ell_2}(\eta)u_{i+1,j+1}^{(\ell)}\bigg)+ \mathcal{O}(h^{k+2}).
		\end{split}
		\end{align}
		Now, we consider Taylor's expansion of $u_{i+1,j}^{(\ell)}, u_{i,j+1}^{(\ell)}, u_{i+1,j+1}^{(\ell)}$:
		\begin{align*}
		u_{i+1,j}^{(\ell)}&=\sum_{m_1=0}^{k-|\ell|} \frac{u_{i,j}^{(\ell_1+m_1,\ell_2)}}{m_1 !}(\Delta x)^{m_1} +  \frac{u_{i,j}^{(\ell_1+k-|\ell|+s,\ell_2)}}{(k-|\ell|+s) !}(\Delta x)^{k-|\ell|+s}\cr
		u_{i,j+1}^{(\ell)}&=\sum_{m_2=0}^{k-|\ell|} \frac{u_{i,j}^{(\ell_1,\ell_2+m_2)}}{m_2 !}(\Delta y)^{m_2} + \frac{u_{i,j}^{(\ell_1,\ell_2+k-|\ell|+s)}}{(k-|\ell|+s) !}(\Delta x)^{k-|\ell|+s}\cr
		u_{i+1,j+1}^{(\ell)}&=\sum_{|m|=0}^{k-|\ell|} \frac{u_{i,j}^{(\ell+m)}}{m_1 ! m_2 !}(\Delta)^m + \sum_{|m|=k-|\ell|+s} \frac{u_{i,j}^{(\ell+m)}}{m_1 ! m_2 !}(\Delta)^m,
		\end{align*} 
		where $m=(m_1,m_2)$ is an multi index. Inserting this into \eqref{Q explicit 2d}, we obtain
		\begin{align}\label{Q with gamma}
		\begin{split}
		Q_{i+\theta,j+\eta}=:  \sum_{|\ell|=0}^k (\Delta)^{\ell} \Lambda_\ell(\theta,\eta) u_{i,j}^{(\ell)} + \Gamma(\theta,\eta)
		+ \mathcal{O}\left(h^{k+2}\right),
		\end{split}
		\end{align}
		where $\Lambda_\ell(\theta,\eta)$ and $\Gamma(\theta,\eta)$ are given by
		\begin{align*} 
		\Lambda_\ell(\theta,\eta)&= \alpha_{\ell_1}(\theta)\alpha_{\ell_2}(\eta)  + \alpha_{\ell_2}(\eta)\sum_{m_1=0}^{\ell_1}\beta_{m_1}(\theta) \frac{1}{(\ell_1-m_1)!}\cr
		&+ \alpha_{\ell_1}(\theta)\sum_{m_2=0}^{\ell_2}\beta_{m_2}(\eta) \frac{1}{(\ell_2-m_2)!} +\sum_{|m|=0}^{|\ell|} \beta_{m_1}(\theta)\beta_{m_2}(\eta)\frac{1}{(\ell_1-m_1)! (\ell_2-m_2)!}\cr
		\Gamma(\theta,\eta) &= \sum_{|\ell|=0}^{k} \beta_{\ell_1}(\theta)\alpha_{\ell_2}(\eta)\frac{u_{i,j}^{(\ell_1+k-|\ell|+s,\ell_2)}}{(k-|\ell|+s) !}(\Delta x)^{k-|\ell|+s}(\Delta)^{\ell}\cr
		&\quad+\sum_{|\ell|=0}^{k}\alpha_{\ell_1}(\theta)\beta_{\ell_2}(\eta)\frac{u_{i,j}^{(\ell_1,\ell_2+k-|\ell|+s)}}{(k-|\ell|+s) !}(\Delta y)^{k-|\ell|+s}(\Delta)^{\ell}\cr
		&\quad +\sum_{|\ell|=0}^{k}\beta_{\ell_1}(\theta)\beta_{\ell_2}(\eta)\sum_{|m|=k-|\ell|+s} \frac{u_{i,j}^{(\ell+m)}}{m_1 ! m_2 !}(\Delta)^{m+\ell}.
		\end{align*}
		Now, we add $0$ to $\Gamma(\theta,\eta)$ using the following identity: 
		\begin{align}\label{add 0 2d}
		\begin{split}
		0&=\sum_{|\ell|=k+1}\left(\alpha_{\ell_1}(\theta)+\beta_{\ell_1}(\theta)\right)\left(\alpha_{\ell_2}(\eta)+\beta_{\ell_2}(\eta)\right)u_{i,j}^{(\ell)}(\Delta)^{\ell}\cr
		&=\sum_{|\ell|=k+1}\alpha_{\ell_1}(\theta)\alpha_{\ell_2}(\eta)u_{i,j}^{(\ell)}(\Delta)^{\ell}+\sum_{|\ell|=k+1}\beta_{\ell_1}(\theta)\alpha_{\ell_2}(\eta)u_{i,j}^{(\ell)}(\Delta)^{\ell}\cr
		&\quad+\sum_{|\ell|=k+1}\alpha_{\ell_1}(\theta) \beta_{\ell_2}(\eta)u_{i,j}^{(\ell)}(\Delta)^{\ell}+\sum_{|\ell|=k+1}\beta_{\ell_1}(\theta)\beta_{\ell_2}(\eta)u_{i,j}^{(\ell)}(\Delta)^{\ell},
		\end{split}
		\end{align}
		then
		\begin{align*}
		\Gamma(\theta,\eta)+0 			 &=\sum_{|\ell|=k+1}\alpha_{\ell_1}(\theta)\alpha_{\ell_2}(\eta)u_{i,j}^{(\ell)}(\Delta)^{\ell}+\sum_{|\ell|=0}^{k+1} \beta_{\ell_1}(\theta)\alpha_{\ell_2}(\eta)\frac{u_{i,j}^{(\ell_1+k-|\ell|+s,\ell_2)}}{(k-|\ell|+s) !}(\Delta x)^{k-|\ell|+s}(\Delta)^{\ell}\cr
		&\quad+\sum_{|\ell|=0}^{k+1}\alpha_{\ell_1}(\theta)\beta_{\ell_2}(\eta)\frac{u_{i,j}^{(\ell_1,\ell_2+k-|\ell|+s)}}{(k-|\ell|+s) !}(\Delta y)^{k-|\ell|+s}(\Delta)^{\ell}\cr
		&\quad+\sum_{|\ell|=0}^{k+1}\beta_{\ell_1}(\theta)\beta_{\ell_2}(\eta)\sum_{|m|=k-|\ell|+s} \frac{u_{i,j}^{(\ell+m)}}{m_1 ! m_2 !}(\Delta)^{m+\ell}.
		\end{align*}
		This reduces to
		\begin{align*}
		\Gamma(\theta,\eta) = \sum_{|\ell|=k+1}(\Delta)^{\ell}\Lambda_{\ell}(\theta,\eta)u_{i,j}^{(\ell)}.  
		\end{align*}
		Based on this formula, we rearrange all terms in \eqref{Q with gamma} as follows:	
		\begin{align*}
		Q_{i+\theta,j+\eta}&= u_{i,j} + \theta \Delta x u'_{i,j} + \eta \Delta y u_{i,j}^{\backprime}+\left(\frac{\theta^2}{2} + \frac{1}{24}\right)(\Delta x)^2 u''_{i,j} + \left(\frac{\eta^2}{2} + \frac{1}{24}\right)(\Delta y)^2 u_{i,j}^{\backprime\backprime} + \eta\theta \Delta x\Delta y u'^{\backprime}_{i,j}\cr
		&\quad+ \left( \frac{\theta^3}{6} + \frac{\theta}{24} \right)  (\Delta x)^3u'''_{i,j} + \left(\frac{\eta \theta^2 }{2}+ \frac{\eta}{24}\right)(\Delta x)^2\Delta y u''^{\backprime}_{i,j} \cr
		&\quad + \left(\frac{\theta\eta^2}{2} + \frac{\theta}{24}\right)\Delta x(\Delta y)^2 u'^{\backprime\backprime}_{i,j} + \left(\frac{\eta^3 }{6} + \frac{\eta}{24}\right)(\Delta y)^3u_{i,j}^{\backprime\backprime\backprime} + \cdots + \mathcal{O}(h^{k+2})\cr
		&= u_{i+\theta,j+\eta} + \frac{(\Delta x)^2}{24}u''_{i+\theta,j+\eta} + \frac{(\Delta y)^2}{24}u_{i+\theta,j+\eta}^{\backprime\backprime} +\cdots + \mathcal{O}(h^{k+2})\cr
		&= \bar{u}_{i+\theta,j+\eta} + \mathcal{O}(h^{k+2}),
		\end{align*}
		which completes the proof.
	\end{proof}
	\section{Semi-Lagrangian schemes for hyperbolic system with BDF methods}
	The BDF methods \cite{HW} for an ordinary system $y'(t)=f(y)$ can be represented by
	\begin{align*}
	y^{n+1}= \sum_{k=1}^{s}a_k y^{n+1-s} + \beta_s f^{n+1}.
	\end{align*}
	where $\alpha_k$ and $\beta_s$ are coefficients corresponding to s-order BDF methods. Here, we consider two cases $s=2,3$:
	\begin{align*}
	\text{BDF2:}&\quad y^{n+1}= \frac{4}{3} y^{n} - \frac{1}{3} y^{n-1} + \frac{2}{3} f^{n+1}\cr
	\text{BDF3:}&\quad y^{n+1}= \frac{18}{11} y^{n} - \frac{9}{11} y^{n-1} + \frac{2}{11} y^{n-1} + \frac{6}{11} f^{n+1}.
	\end{align*}
	\subsection{BDF methods for Xin-Jin model}\label{BDF section Xin-Jin}
	Applying BDF method based SL methods to \eqref{hyper system lagrangian}, we obtain:
	\begin{align}\label{BDFschemes}
	\begin{array}{l}
	\displaystyle f^{n+1}_i= \sum_{k=1}^{s}\alpha_{k}\,\tilde{f}^{n,k} + \beta_s \,\frac{\Delta t}{\kappa} \, K^{n+1}_{i,1}\\
	\displaystyle g^{n+1}_i= \sum_{k=1}^{s}\alpha_{k}\,\tilde{g}^{n,k} + \beta_s \,\frac{\Delta t}{\kappa} \, K^{n+1}_{i,2}.
	\end{array}
	\end{align}
	Here we use the following notation:
	\begin{itemize}
		\item For $k=1,\dots,s$, the $(n+1-k)$th stage values of $f,g$ along the backward-characteristics which come from $x_i$ with characteristic speed $-1,1$ at time $t^{n+1}$:
		\begin{align*}
		\tilde{f}_{i}^{n,k} &\approx f(x_{i}+ k \Delta t, t^{n+1-k}),\quad \tilde{g}_{i}^{n,k} \approx g(x_{i}-k \Delta t, t^{n+1-k}).
		\end{align*}
		\item Fluxes at time $t^{n+1}$: 
		$$K^{n+1}_{i,1}\approx F(u_i^{n+1}) - v_i^{n+1}, \quad K^{n+1}_{i,2}\approx -K^{n+1}_{i,1}.$$
	\end{itemize}
	The algorithm can be summarized as follows:
	\subsubsection*{Algorithm of $s$-order BDF methods}\label{algorithm BDF}
	\begin{enumerate}
		\item For $k=1,2,\dots,s$, interpolate $\tilde{f}_{i}^{n,k}$ and $\tilde{g}_{i}^{n,k}$ on $x_{i}+k\Delta t$ and $x_{i}-k\Delta t$ from $\{f_{i}^{n+1-k}\}$ and $\{g_{i}^{n+1-k}\}$, respectively.
		\item
		By summing and subtracting two equations in \eqref{BDFschemes}, compute: 
		\begin{align*}
		u_i^{n+1}&=\frac{\sum_{k=1}^s \alpha_k\left(\tilde{g}_i^{n,k} + \tilde{f}_i^{n,k} \right)}{2}\cr
		v_i^{n+1}&=\frac{   \kappa \sum_{k=1}^s \alpha_k\left(\tilde{g}_i^{n,k} - \tilde{f}_i^{n,k} \right)/2  +\beta_{k} \Delta t F\left(u_i^{n+1}\right)    }{\kappa + \beta_{s} \Delta t}.
		\end{align*}
		\item Compute:
		\[
		f_i^{n+1}=u_i^{n+1}-v_i^{n+1}, \quad g_i^{n+1}=u_i^{n+1} + v_i^{n+1}
		\]	
	\end{enumerate}
	\subsection{BDF methods for Broadwell model}\label{BDF methods for Broadwell model}
	Now, we extend this to high order s-order BDF methods. The solutions are obtained by
	\begin{align}\label{s-BDF}
	\begin{split}
	f_i^{n+1}&= \sum_{\ell=1}^{s} \alpha_k f_i^{n,k} + \frac{\beta_{s}\Delta t}{\kappa} Q_i^{n+1}\cr
	g_i^{n+1}&= \sum_{\ell=1}^{s} \alpha_k g_i^{n,k} + \frac{\beta_{s}\Delta t}{\kappa} Q_i^{n+1}\cr
	h_i^{n+1}&= \sum_{\ell=1}^{s} \alpha_k h_i^{n+1-k} - \frac{\beta_{s}\Delta t}{\kappa} Q_i^{n+1}\cr
	\end{split}
	\end{align}
	where
	\begin{align*}
	Q_i^{n+1}&= (h_i^{n+1})^2-f_i^{n+1}g_i^{n+1}, \quad f_i^{n,k}\approx f(x_i-k \Delta t,t^{n+1-k}), \quad g_i^{n,k}\approx g(x_i+k \Delta t,t^{n+1-k}),
	\end{align*}
	Then, s-order BDF methods are summarized as follows:
	\subsubsection*{Algorithm of s-order BDF methods}
	\begin{enumerate}
		\item Reconstruct $f_i^{n,k}$ and 
		$g_i^{n,k}$ for $k=1,\cdots,s$.
		\item Compute $F_i^{n}$, $G_i^{n}$ and $H_i^{n}$ using 
		\begin{align*}
		\begin{split}
		F_i^{n}:=\sum_{\ell=1}^{s} \alpha_k f_i^{n,k},\quad 
		G_i^{n}:=\sum_{\ell=1}^{s} \alpha_k g_i^{n,k},\quad 
		H_i^{n}:=\sum_{\ell=1}^{s} \alpha_k h_i^{n+1-k}.
		\end{split}
		\end{align*}
		\item Solve \eqref{s-BDF} for
		\begin{align*}
		&h_i^{n+1} = \frac{\beta_s\Delta t( H_i^{n} + F_i^{n} )( H_i^{n} + G_i^{n})+ \kappa H_i^{n}}{\beta_s\Delta t\left(G_i^{n} + 2 H_i^{n} + F_i^{n}\right) + \kappa},\cr
		&f_i^{n+1} =  H_i^{n} + F_i^{n}-  h_i^{n+1},\qquad 
		g_i^{n+1}=  H_i^{n} + G_i^{n}  -  h_i^{n+1}.
		\end{align*}
	\end{enumerate}


	%

	\section*{Acknowledgments}
	S. Y. Cho has been supported by ITN-ETN Horizon 2020 Project ModCompShock, Modeling and Computation on Shocks and Interfaces, Project Reference 642768. 
	S.-B. Yun has been supported by Samsung Science and Technology Foundation under Project Number SSTF-BA1801-02. 
	All the authors would like to thank the Italian Ministry of Instruction, University and Research (MIUR) to support this research with funds coming from PRIN Project 2017 (No. 2017KKJP4X entitled “Innovative numerical methods for evolutionary partial differential equations and applications”). 
	S. Boscarino has been supported by the University of Catania (“Piano della Ricerca 2016/2018, Linea di intervento 2”). S. Boscarino and G. Russo are members of the INdAM Research group GNCS. 
	

\bibliographystyle{amsplain}
	\bibliography{references}
	
\end{document}